\documentclass[11pt, letterpaper]{amsart}
\usepackage[utf8]{inputenc}
\usepackage{amsmath}
\usepackage{amsfonts}
\usepackage{amssymb}
\usepackage{amsthm}
\usepackage{mathtools}
\usepackage{url}
\usepackage[hidelinks]{hyperref}

\numberwithin{equation}{section}

\theoremstyle{plain}
\newtheorem{theorem}{Theorem}[section]
\newtheorem{lemma}[theorem]{Lemma}
\newtheorem{corollary}[theorem]{Corollary}
\newtheorem{proposition}[theorem]{Proposition}
\newtheorem{definition}[theorem]{Definition}

\newtheorem{question}[theorem]{Question}
\newtheorem{remark}[theorem]{Remark}

\title[Binomial divisors avoiding an interval]{Binomial coefficients with divisors avoiding an interval}
\author{Hung M. Bui}
\address{Department of Mathematics, University of Manchester, Manchester M13 9PL, UK}
\email{hung.bui@manchester.ac.uk}
\author{Slava Naprienko}
\email{naprienko@alumni.stanford.edu}
\author{Kyle Pratt}
\address{Brigham Young University, Department of Mathematics, Provo, UT 84602, USA}
\email{kyle.pratt@mathematics.byu.edu}
\author{Alexandru Zaharescu}
\address{Department of Mathematics, University of Illinois at Urbana-Champaign, 1409 West Green Street, Urbana, IL 61801, USA
and Simion Stoilow Institute of Mathematics of the Romanian Academy, P.O. Box 1-764, RO-014700 Bucharest,
Romania}
\email{zaharesc@illinois.edu}


\allowdisplaybreaks

\begin{document}
\date{}

\begin{abstract}
We solve a fifty-year-old conjecture of Erd\H{o}s and Graham concerning whether the binomial coefficient $\binom{n}{k}$ with $1 \leq k \leq \frac{n}{2}$ must always have a divisor $\leq n$ that is ``close'' to $n$: that is, bigger than a constant times $n$. We show this is the case when $k$ is sufficiently large as a function of $n$. However, we show it is possible to find binomial coefficients $\binom{n}{k}$, where $k$ is small compared to $n$, such that $\binom{n}{k}$ does not have divisors $\leq n$ close to $n$. This latter, more substantial argument involves a restricted covering problem with residue classes, sieve methods, and various exponential sum estimates.
\end{abstract}

\maketitle

\tableofcontents

\section{Introduction}

\subsection{The problem and main results}

Given integers $n$ and $k$ with $1 \leq k < n$, we may define the binomial coefficient
\begin{align}\label{eq:defn of binomial coefficient}
\binom{n}{k} \coloneqq \frac{n(n-1)\cdots (n-k+1)}{k!},
\end{align}
which is always an integer. Binomial coefficients arise naturally in many areas of mathematics, such as combinatorics, algebra, and probability. Binomial coefficients are also natural objects of study in number theory, owing to their rich arithmetic structure. For instance, a classic result of Kummer \cite{Kum1852} describes the prime divisors $p$ of $\binom{n}{k}$ in terms of carries in base $p$ (see Lemma \ref{lem:kummers thm} below). As another example, work of S\'ark\"{o}zy \cite{Sar1985}, Granville--Ramar\'e \cite{GR1996}, and Velammal \cite{Vel1995} showed that the central binomial coefficients $\binom{2n}{n}$ are not squarefree for any $n > 4$, thereby answering a question of Erd\H{o}s (see \cite[p. 71]{EG1980}). There are many interesting open problems about number-theoretic aspects of binomial coefficients (see, e.g., \cite[B31 and B33]{Guy2004}).

In this work, we are interested in a conjecture of Erd\H{o}s and Graham regarding divisors of binomial coefficients. As can be seen in \eqref{eq:defn of binomial coefficient}, $\binom{n}{k}$ is a divisor of a product of consecutive integers near $n$. It is therefore natural to ask whether $\binom{n}{k}$ itself must have a divisor near $n$. Indeed, Erd\H{o}s asked whether one of the integers $n,n-1,\ldots,n-k+1$ must actually divide $\binom{n}{k}$. However, Schinzel showed this is too ambitious \cite{Schinzel1958}, giving the counterexample $\binom{99215}{15}$. (We shall return to Schinzel's counterexample in a moment.)

Still, it is tempting to ask if $\binom{n}{k}$ must always have a divisor ``close'' to $n$. The main question of interest for this work is the following open problem, first posed by Erd\H{o}s and Graham \cite[p. 351]{EG1976} (see also \cite{BloomErdos387}, \cite[p. 74]{EG1980}, \cite[B33]{Guy2004}).

\begin{question}\label{ques:main question}
Is there a positive constant $c$ such that every binomial coefficient $\binom{n}{k}$ with $1 \leq k < n$ has a divisor in the interval $(cn,n]$?
\end{question}

Question \ref{ques:main question} is delicate and, as we shall see below, even Erd\H{o}s's intuition about the problem changed over the years.

We note that, since $\binom{n}{k} = \binom{n}{n-k}$, there is no loss of generality for Question \ref{ques:main question} in restricting our attention to $1 \leq k \leq \frac{n}{2}$.

Numerical investigations may seem to indicate that the answer to Question \ref{ques:main question} is yes. If $\binom{n}{k} = \binom{4}{2}$, then the largest divisor of $\binom{n}{k}$ that is $\leq n$ is $3 = \frac{3}{4}n$. All other ``small'' binomial coefficients have divisors in $[\frac{3}{4}n,n]$, and, indeed, most of them have divisors $\leq n$ that are quite close to $n$. On the basis of this numerical evidence, one may suppose that $\binom{n}{k}$ always has a divisor in $[\frac{3}{4}n,n]$. It certainly appears to be the case that ``typical'' binomial coefficients have divisors $\leq n$ that are close to $n$, and there are ways to make this statement precise (see, for example, \cite{Har1979}). 

Our first result is that, if $n$ is sufficiently large and $k$ is sufficiently large as a function of $n$, then $\binom{n}{k}$ must have a divisor close to $n$. That is, the answer to Question \ref{ques:main question} is ``yes'' when $k$ is large enough compared to $n$.

\begin{theorem}\label{thm:large k regime}
Let $\epsilon > 0$ be a small positive constant, and let $n$ be sufficiently large with respect to $\epsilon$. Assume $k$ is an integer satisfying 
\begin{align*}
\exp((\log n)^{2/3+\epsilon}) \leq k \leq \frac{n}{2}.
\end{align*}
Then the binomial coefficient $\binom{n}{k}$ has a divisor in the interval
\begin{align}\label{eq:large k divisor interval}
\Big(n - \frac{n}{(\log n)^{1/4}},n \Big].
\end{align}
\end{theorem}

\begin{remark}
We note that it is possible to refine the lower bound on $k$ in Theorem \ref{thm:large k regime}, replacing $(\log n)^\epsilon$ with some power of $\log \log n$, but we will not do so here.
\end{remark}

The discussion above leaves open the matter of what happens when $k$ is small, and raises the question of whether ``atypical'' binomial coefficients could behave differently. It is profitable to return to Schinzel's counterexample
\begin{align*}
\binom{99215}{15} = \frac{99215 \cdot 99214 \cdots 99201}{15!}
\end{align*}
that refuted Erd\H{o}s's earlier strong conjecture about divisors of binomial coefficients. The main property of Schinzel's binomial coefficient is that each of the integers $99215,\ldots,99201$ has a common divisor with $15!$ that is $\geq 2$, and none of $99215,\ldots,99201$ divide $\binom{99215}{15}$ (see \cite[p. 198]{Schinzel1958}). Furthermore, Schinzel's construction actually yields an arithmetic progression of $n$'s such that each $\binom{n}{15}$ has these properties. By searching in this arithmetic progression, one can find the binomial coefficient $\binom{m_0}{15}$, where $m_0=13085213870159810495$, whose largest divisor $\leq m_0$ is $9502357691425576661$. The ratio of this divisor to $m_0$ is $\approx 0.72619\ldots$; this is the simplest binomial coefficient of which the authors are aware that does not have a divisor in the interval $[\frac{3}{4}n,n]$. On the Erd\H{o}s Problems forum, user RobertGerbicz \cite{BloomErdos387} pushed the computations further and found the example $\binom{n_0}{15}$, where $n_0 = 178256013517113649496495$. This binomial coefficient is the product of fifteen distinct primes, and it has no divisors in the interval $(\frac{1}{2}n_0,n_0]$. We shall have more to say in due course about Schinzel's ideas and these binomial coefficient examples.

Thus, there are rare, atypical binomial coefficients that have no divisors close to $n$. This suggests the answer to Question \ref{ques:main question} is ``no'' in general. Notably, Erd\H{o}s himself came to believe the answer to Question \ref{ques:main question} was negative (see \cite{Erd197882}), and we confirm this belief.

\begin{theorem}\label{thm:main theorem}
Let $k_0$ be any sufficiently large fixed constant, and let $\delta>0$ be a sufficiently small constant. There are infinitely many binomial coefficients $\binom{n}{k}$ for which $k_0 < k \leq \delta (\log \log n)^{1/2}$ and $\binom{n}{k}$ has no divisors in the interval
\begin{align*}
\Big(n \cdot \frac{241 \log \log k}{\log k}, n\Big].
\end{align*}
In fact, one can find such binomial coefficients with $K/2<k \leq K$ for any $K$ with $k_0 \ll K \ll \delta (\log \log n)^{1/2}$.
\end{theorem}

\begin{remark}
The constant $241$ in Theorem \ref{thm:main theorem} is of essentially no consequence, and can be ignored on a first reading.
\end{remark}

We have the following immediate corollary.

\begin{corollary}
There are infinitely many binomial coefficients $\binom{n}{k}$ with $k \asymp (\log \log n)^{1/2}$ such that $\binom{n}{k}$ has no divisors in the interval
\begin{align*}
\Big(n \cdot \frac{500\log \log \log \log n}{\log \log \log n},n \Big].
\end{align*}
\end{corollary}

\begin{corollary}
For every fixed positive constant $c$, there exists a binomial coefficient $\binom{n}{k}$ with $1 \leq k < n$ that has no divisors in the interval $(cn,n]$.
\end{corollary}

The results above suggest there is a transition range or threshold $k_0(n)$ such that, if $k \gtrapprox k_0(n)$, then for all large $n$ the binomial coefficient $\binom{n}{k}$ has a divisor in $(\epsilon n,n]$, but, if $k \lessapprox k_0(n)$, then there are large $n$ such that the binomial coefficient $\binom{n}{k}$ has no divisors in $(\epsilon n,n]$.

This threshold is probably the size of some power of $\log n$. For instance, if one assumes Cr\'amer-type conjectures on the largest possible gaps between consecutive primes, then easy arguments would show $k_0(n) \ll (\log n)^2$ (see \cite{BFT2023} for further discussion of Cr\'amer conjectures and prime gaps). On the other hand, assuming some kind of strong, uniform version of the prime $k$-tuples conjecture (like some suitable variant of \cite[Theorem 1.3]{BFT2023}), we might expect $k_0(n) \gg \frac{\log n}{\log \log n}$. We believe this lower bound to be closer to the truth, but establishing this seems difficult.

\subsection{Outline of the paper}

The structure of the remainder of the paper is as follows. In Section \ref{sec:notation}, we lay out the notation and conventions that shall be in force throughout the paper. In Section \ref{sec:outline}, we give an outline of the proofs of the theorems; in some ways, this section functions as a more technical extension of the introduction. We prove Theorem \ref{thm:large k regime} in Section \ref{sec:proof of large k thm}. 

We begin the proof of Theorem \ref{thm:main theorem} in Section \ref{sec:covering problem}. As explained in Section \ref{sec:outline} below, proving Theorem \ref{thm:main theorem} splits into two (essentially) independent sub-problems, the ``covering problem'' and the ``divisor problem.'' We attack the covering problem in Section \ref{sec:covering problem}. After introducing suitable notation, we reduce the proof of the important Theorem \ref{thm:main covering theorem} to a technical result, Proposition \ref{prop:some k has cover of weight B}. We introduce important definitions and results, and then give the proof of the proposition in a following subsection.

We begin our attack on the divisor problem in Section \ref{sec:key props for divisor problem}, and the rest of the paper is spent on this analysis. We use Theorem \ref{thm:main covering theorem} as a ``black box'' in studying the divisor problem, but otherwise the covering and divisor problems are independent. After introducing important notation in Section \ref{sec:key props for divisor problem}, we reduce the proof of Theorem \ref{thm:main theorem} to six key propositions. The remainder of the paper is devoted to proving these six key propositions.

All of the proofs of the key propositions involve sieve methods. We give an overview of the notation and results in sieve methods we need in Section \ref{sec:sieve}, and then prove three of the more elementary key propositions.

The proofs of the three remaining key propositions all require exponential sum estimates. In Section \ref{sec:exp sums}, we collect some results that will be useful in estimating these exponential sums (all basically of Kloosterman type), and prove one of the key propositions.

The proofs of the last two key propositions are somewhat more involved, and each is treated in its own section in the last two sections of the paper, Sections \ref{sec:bilinear} and \ref{sec:differencing}.

\section{Notation and conventions}\label{sec:notation}

The reader is encouraged to skip this section on a first reading, only returning when necessary.

We use the traditional asymptotic notation $O(\cdot), o(\cdot), \ll$, and $\gg$; dependence on other quantities may occasionally be denoted by subscripts. For instance, $o_{k \rightarrow \infty}(1)$ denotes a quantity that goes to zero as $k$ tends to infinity.

The notation $n \sim y$ means $n \in (y/2,y]$.

We write $\omega(n)$ for the number of distinct prime factors of $n$, and $\Omega(n)$ for the total number of prime factors of $n$. We write $\tau_k(n)$ for the $k$-fold divisor function of $n$, the number of ways to write $n$ as the product of $k$ positive integers. We denote the von Mangoldt function by $\Lambda(n)$.

We write $\mathbf{1}(E)$ for the indicator function of a condition or event $E$.

If $n > 1$ is a positive integer, we let $P^-(n)$ denote the least prime factor of $n$. Otherwise, we set the convention $P^-(1) = \infty$. Let $P^+(n)$ similarly denote the largest prime factor of $n$, and set $P^+(1)=1$.

Given coprime integers $n$ and $q$, we let $\overline{n} \pmod{q}$ denote the multiplicative inverse of $n$ modulo $q$.

If $y \geq 2$, then $n$ is $y$-smooth if $p \mid n$ implies $p \leq y$. Given $x \geq y \geq 2$, we let $\Psi(x,y)$ denote the number of $y$-smooth integers $\leq x$.

The notation in this paragraph is relevant for Section \ref{sec:covering problem}. For a nonempty set of primes $\mathcal{T} \subseteq \{p \leq k : p \text{ prime}\}$ and an integer $j\geq 1$, define
\begin{align*}
z_\mathcal{T}(j) \coloneqq \frac{j}{\prod_{p \in \mathcal{T}} p^{v_p(j)}}.
\end{align*}
Given a nonempty set of primes $\mathcal{T}$, suppose further that we have a choice of residue class $a_p \pmod{p}$ for every $p \in \mathcal{T}$. We then define
\begin{align*}
C_\mathcal{T}(j) \coloneqq z_\mathcal{T}(j) \prod_{\substack{p \in \mathcal{T} \\ j \equiv a_p(\text{mod }p)}} p.
\end{align*}
Let $a_p \pmod{p}$ be a residue class; we assume $0 \leq a_p < p$. We say the residue class $a_p \pmod{p}$ is \emph{nonzero} if $a_p > 0$. Given a positive integer $k$, we write $k \pmod{p}$ for the unique nonnegative integer with $0 \leq k \pmod{p} < p$.

For a compactly supported function $f: \mathbb{R} \rightarrow \mathbb{C}$, we define the Fourier transform $\widehat{f}(y)$ of $f$ by
\begin{align*}
\widehat{f}(y) \coloneqq \int_{\mathbb{R}} f(x) e(-xy) dx.
\end{align*}

We need a suitable auxiliary smooth function. Throughout, we let $\psi_0: \mathbb{R}\rightarrow\mathbb{R}$ be a fixed nonnegative smooth function supported on $[\frac{1}{4}, \frac{5}{4}]$ that is identically equal to one on $[\frac{1}{2},1]$, and whose Fourier transform $\widehat{\psi_0}(y)$ satisfies the bound
\begin{align*}
|\widehat{\psi_0}(y)| \ll \exp(-|y|^{1/2}).
\end{align*}
We may additionally require that the derivatives of $\psi_0$ satisfy $|\psi_0^{(j)}(x)| \ll j^{O(j)}$ for every $j \geq 0$, where the implied constants are absolute. (See \cite[Appendix A]{iwaniecLectures} for the construction of such a function.) By repeated integration by parts, this also implies that the Mellin transform of $\psi_0$ has rapid decay.

The following notation is in force from Section \ref{sec:key props for divisor problem} onwards. The integer $k$ is large and fixed. We set $B = \frac{\log k}{241\log \log k}$. (In Section \ref{sec:outline} below, we describe the strategy for finding binomial coefficients with no divisors in $(\epsilon n,n]$, and one should think $B = \epsilon^{-1}$.) The arithmetic progression $\gamma \pmod{M}$ is fixed, where the integer $M = e^{O(k)}$. The real number $x$ is large compared to $k$, so that we may assume $k \leq \delta (\log \log x)^{1/2}$ for some small fixed constant $\delta > 0$. We set $\epsilon_k = 3^{-k}$, and let $z = x^{\epsilon_k}$. We work with $n \sim x$ such that $P^-(\binom{n}{k})\geq z$. Any divisor $d$ of $\binom{n}{k}$ with $P^-(d) \geq z$ decomposes uniquely into divisors 
\begin{align*}
d = d_0d_1 \cdots d_{k-1},
\end{align*}
where $d_i$ divides $n-i$. Thus, $n \equiv i \pmod{d_i}$. Observe that the $d_i$ are pairwise coprime. These congruence conditions combine to give a congruence condition $n \equiv \gamma_d \pmod{d}$ through the Chinese remainder theorem. Each divisor $d_i$ is either one, or is $\geq z$. It will sometimes be notationally beneficial to gather together all but one of the $d_i$ into a divisor, which we shall denote by $D$, so that $d = Dd_i$ for some $i$. If we write $d$ in this way, then the congruence condition on $n$ is equivalent to the pair of congruences
\begin{align*}
n \equiv \gamma_D \pmod{D}, \ \ \ \ n \equiv i \pmod{d_i},
\end{align*}
where $\gamma_D$ arises from the Chinese remainder theorem in the natural way.

\section{Outline of the proofs}\label{sec:outline}

\subsection{Outline of the proof of Theorem \ref{thm:large k regime}}

We first describe the short proof of Theorem \ref{thm:large k regime}. We want to show that $\binom{n}{k}$ has a divisor that is very near $n$, as long as $k \gtrapprox \exp((\log n)^{2/3})$. When $k$ is larger than a power of $n$, we can use classical results on primes in short intervals to show that one of the terms $n,n-1,\ldots,n-k+1$ in the numerator of $\binom{n}{k}$ is a prime $>k$, and therefore this term divides $\binom{n}{k}$. When $k$ is smaller than a power of $n$, we use Kummer's theorem (Lemma \ref{lem:kummers thm} below) and exponential sum estimates to show there are many primes $\approx \exp((\log n)^{2/3})$ that divide $\binom{n}{k}$. We can then multiply some of these primes together to get a divisor of $\binom{n}{k}$ near $n$. The $\exp((\log n)^{2/3})$ barrier arises from the familiar limitations of Vinogradov-type exponential sum estimates, and similar barriers appear in other problems involving binomial coefficients (see, e.g., \cite{EK1999}, \cite{MRSTT2022}).

\subsection{Outline of the proof of Theorem \ref{thm:main theorem}}

The proof of Theorem \ref{thm:main theorem} is rather involved, and constitutes the bulk of the paper. Given any fixed, small $\epsilon > 0$, we would like to produce binomial coefficients 
\begin{align*}
\binom{n}{k} = \frac{n(n-1)\cdots (n-k+1)}{k!}
\end{align*}
that have no divisors in $(\epsilon n,n]$. Since $\binom{n}{k}$ is an integer, the $k!$ in the denominator cancels out with terms in the numerator, so we can write
\begin{align*}
\binom{n}{k} = \prod_{i=0}^{k-1} \frac{n-i}{g_i}
\end{align*}
for some integers $g_i \mid (n-i)$ that multiply to $k!$. We now observe that there is an ``obvious'' way that $\binom{n}{k}$ could have a divisor in $(\epsilon n,n]$, and there are ``non-obvious'' ways that $\binom{n}{k}$ could have a divisor in $(\epsilon n,n]$. We now describe the strategy to show there are binomial coefficients that avoid both of these ways of having divisors near $n$. (A similar strategy to the one described below was independently suggested by Terence Tao \cite{BloomErdos387}.)

\subsubsection{The ``covering'' problem}

The ``obvious'' way is if some term $\frac{n-i}{g_i}$ is $> \epsilon n$, for then $\frac{n-i}{g_i}$ itself would be a divisor of $\binom{n}{k}$ in $(\epsilon n,n]$. Thus, we want to find conditions on $n$ and $k$ that ensure every term $\frac{n-i}{g_i}$ is $\leq \epsilon n$. One simple way to do this is to try to force each $g_i$ as above to be $> B=\epsilon^{-1}$. It is also beneficial to force $\binom{n}{k}$ to not be divisible by any primes $\leq k$. (See Theorem \ref{thm:main covering theorem} for a precise statement of the result.) By arguments similar to those of Schinzel \cite{Schinzel1958} mentioned in the introduction, we use the Chinese remainder theorem to reduce to a kind of ``covering'' problem. The general flavor of such covering problems is to choose residue classes $a_p$ modulo primes $p \leq k$ such that every integer $j \in [1,k]$ satisfies $j \equiv a_p \pmod{p}$ for some $p$. In this formulation, the problem seems similar to the arguments used to establish large gaps between primes (see \cite{FGKMT2018} for the current state-of-the-art), but there are some key differences. For one thing, we need each integer $1\leq j\leq k$ to be ``covered enough'' by residue classes: an integer $j$ is covered enough if $C_\mathcal{R}(j) \geq B$ for a suitable set of primes $\mathcal{R}$. Related to ``covering enough'' is that we can use some prime-power components of $k!$ in our covering, and not just primes themselves. The other restriction is that we want $\binom{n}{k}$ to have no prime factors $\leq k$, so we need the $p$-adic valuation of $k!$ to equal the $p$-adic valuation of $n(n-1)\cdots (n-k+1)$ for every $p \leq k$. In the covering problem, this has the effect of disallowing certain residue classes modulo $p$ for each prime $p \leq k$, so we have a kind of ``restricted'' covering problem. More precisely, we require $k \pmod{p} < a_p$ for every prime $p$ where we choose a nonzero residue class $a_p$ (see Definition \ref{defn:cover of weight B} and surrounding discussion).

Our approach to the covering problem has some similarities to an argument of Erd\H{o}s, where he studied a related covering problem (see \cite[p. 203]{Schinzel1958}). A key feature in Erd\H{o}s's proof is a small auxiliary prime $q$ that is used to cover many large primes $p$. The large primes are then used to cover ``survivors'' that are not covered by choices of residue classes made in previous stages of the argument. Thus, Erd\H{o}s's argument can be thought of as unfolding over time: making some initial choice of residue classes, then seeing what is still uncovered, then choosing some more residue classes and seeing what is still uncovered, and so forth. Our argument has similar features.

A main idea behind our argument is that choosing residue classes $a_p= 0 \pmod{p}$ already covers almost all integers $1 \leq j \leq k$ in a suitable fashion (after accounting for additional covering by prime powers). In fact, all $B \leq j \leq k$ are suitably covered. In our argument, we still choose $a_p = 0\pmod{p}$ for most primes $p$, but make different choices of residue classes for a few select primes in order to obtain a valid covering. For instance, it is clear that we have to choose $a_p=1 \pmod{p}$ for some prime $p$ in order to cover the integer $1$.

Let us be somewhat more precise (while still avoiding or glossing over some more technical points). We work at a large scale $K$, and aim to show by averaging arguments that there are many $k \sim K$ such that we can solve the covering problem for $k$ (see Proposition \ref{prop:some k has cover of weight B}). By the large sieve inequality (see Lemma \ref{lem:bound on card Z}), we may fix a prime $Q \approx \exp(\frac{\log K}{\log \log K})$ such that the primes are approximately equidistributed in primitive residue classes modulo $Q$, and we choose the corresponding residue class $a_Q = 1 \pmod{Q}$. In order to satisfy the condition $k \pmod{Q} < a_Q$, we restrict our attention to $k$'s that are divisible by $Q$. Large primes $\equiv 1 \pmod{Q}$ are ``covered'' by $Q$, and will be used in the last stage of the argument to cover a few remaining stragglers.

Having fixed our auxiliary prime $Q$, we introduce a ``deficient'' set $\mathcal{D}$ (see \eqref{eq:deficient set mathcal D}), which is the set of ``small'' integers that are not covered enough after choosing $a_Q=1$ and $a_p=0 \pmod{p}$ for all other primes $p \lessapprox \exp(\frac{\log K}{\log \log K})$. The set $\mathcal{D}$ is quite small, having cardinality $\ll B$ (see Lemma \ref{lem:basic props of mathcal D}). We would like to cover $d \in \mathcal{D}$ by choosing a suitable prime factor $q_d$ of $k-d+1$ and setting $a_{q_d} = d \pmod{q_d}$. If we could do so, we would then have $k \pmod{q_d} = d-1 < d = a_{q_d}$, so this choice of residue class is admissible for our covering problem. In order to make this argument work, however, we need the prime factors $q_d$ to be large but not too large, and this necessitates estimates for $k \sim K$ in various exceptional sets. We obtain bounds for these exceptional sets with trivial estimations, counts for smooth numbers, and sieve theory tools (see Lemma \ref{lem:covering problem mathcal E0} through Lemma \ref{lem:covering problem mathcal E3}).

At this point, we have some special primes $Q$ and $q_d$ for $d \in \mathcal{D}$. We choose $a_p = 0 \pmod{p}$ for all other primes, except for primes $\equiv 1 \pmod{Q}$ near $k$ in a sparse set $\mathcal{P}$. We note that the elements of $\mathcal{P}$ are covered by $Q$, so we do not need to worry about covering them. At this stage, the number of uncovered elements remaining is rather small, on the order of $\lessapprox \exp((\log \log K)^2)$. By a first moment argument (see Lemmas \ref{lem:covering problem raw supply of primes} and \ref{lem:covering problem mathcal E4}), we can find many $k \sim K$ such that there are sufficiently many primes $\equiv 1 \pmod{Q}$ near $k$ available to cover these uncovered elements.

After solving the restricted covering problem, we have constructed an arithmetic progression $n \equiv \gamma \pmod{M}$ for some $M \approx e^k$ such that for all $n$'s in this arithmetic progression, each of the terms $\frac{n-i}{g_i}$ dividing $\binom{n}{k}$ is $\leq \epsilon n$ and $\binom{n}{k}$ is not divisible by any primes $\leq k$. (Actually, it turns out to be technically convenient to augment the arithmetic progression so $\binom{n}{k}$ is not divisible by any primes $\leq 2k$, but this is easily done after the covering problem.) We fix such a $k$ and arithmetic progression. Thus, there are no ``obvious'' ways that $\binom{n}{k}$ could have a divisor in $(\epsilon n,n]$ if $n \equiv \gamma \pmod{M}$.

We remark that our argument gives a quantitative relationship between the size of $\epsilon$ and the size of $k$, with $\epsilon$ allowed to be as small as $\approx \frac{\log \log k}{\log k}$. It would be interesting to determine the optimal relationship between $\epsilon$ and $k$ for this restricted covering problem.

For simplicity, we refer to the whole problem described above as ``the covering problem.'' (We note that, in the first version of this paper \cite{BPZ2026v1}, we had a slightly different approach to the covering problem, which we could handle assuming the Generalized Riemann Hypothesis. All the arguments in the present version of the paper are unconditional.)

\subsubsection{The ``divisor'' problem}

It remains to show there are binomial coefficients $\binom{n}{k}$ with no ``non-obvious'' ways of having divisors in $(\epsilon n,n] = (\frac{n}{B},n]$. A non-obvious way is to have some product of divisors of the terms $\frac{n-i}{g_i}$ multiply together to yield a divisor in $(\frac{n}{B},n]$.

The desired conclusion is straightforward if we allow ourselves to assume some version of a prime $k$-tuples conjecture. If $n \equiv \gamma \pmod{M}$, then we may write $n-i = g_i m_i$ for each $i$, where $g_i$ is divisible only by primes $\leq k$, and $m_i$ is divisible only by primes $> 2k$. Since $n \equiv \gamma \pmod{M}$, we may write $n = \gamma + Mr$ for some positive integer $r$, say. If we reinterpret $m_i = \frac{n-i}{g_i}$, we have
\begin{align*}
\binom{n}{k} = \prod_{i=0}^{k-1} m_i = \prod_{i=0}^{k-1} \left(\frac{\gamma-i}{g_i} + \frac{M}{g_i}r \right).
\end{align*}
The collection of linear forms
\begin{align*}
\left\{L_i(r) \right\}_{i=0}^{k-1} = \left\{\frac{\gamma-i}{g_i} + \frac{M}{g_i}r \right\}_{i=0}^{k-1}
\end{align*}
is ``admissible'' (see \cite{May2016} for more discussion), and therefore, if we assume a suitable prime $k$-tuples conjecture for linear forms $\{a_i r + b_i\}$, we can find infinitely many $r$ such that each $\frac{\gamma-i}{g_i} + \frac{M}{g_i}r= m_i$ is prime. 

Any divisor $d$ of $\binom{n}{k}$ in the interval $(\frac{n}{B},n]$ must be composed of some number of the primes $m_i$. It cannot be that $d$ is equal to a single $m_i$, since each $m_i < \frac{n}{B}$. However, the product of two or more of the $m_i$ is too large, since
\begin{align*}
m_i m_j = \frac{(n-i)(n-j)}{g_ig_j} \geq \frac{(n-k)^2}{k!} > n,
\end{align*}
the latter inequality holding provided $n$ is sufficiently large in terms of $k$ and $k$ is sufficiently large (take $n > 2\cdot k!$, say). Thus, we can find infinitely many $r$ such that each linear form is prime, and we find infinitely many $\binom{n}{k}$ that have no divisors in the interval $(\epsilon n,n]$. The binomial coefficient example due to RobertGerbicz mentioned in the introduction is an explicit example of this potential approach.

We are a long way from proving any kind of prime $k$-tuple conjecture. The closest unconditional approximation is available through sieve theory, whereby one can show the existence of many $r$ such that each linear form $L_i(r)$ has no small prime factors, rather than each $L_i(r)$ being simultaneously prime. We shall make heavy use of sieves in this work. Of course, if we only know that $\binom{n}{k}$ has no small prime factors (rather than being the product of $k$ large primes) it is more difficult to show that $\binom{n}{k}$ has no divisors in the interval $(\frac{n}{B},n]$. It will transpire that this step in our argument is, indeed, rather intricate.

We let $x$ be very large compared to $k$, and sum over $n \sim x$ with $n \equiv \gamma \pmod{M}$ such that $\binom{n}{k}$ has no prime divisors $\leq z = x^{\epsilon_k}$. Here $\epsilon_k$ is some sufficiently small positive constant that depends on $k$ (it is not important for the purposes of this outline, but we take $\epsilon_k = 3^{-k}$). By the fundamental lemma of sieve theory, the number of such $n \sim x$ is
\begin{align*}
\approx \mathfrak{S}_k \frac{x}{M (\log x)^k},
\end{align*}
where $\mathfrak{S}_k$ is a constant (a ``singular series'') depending on $k$. If we let $\mathcal{E}$ denote the count of $n \sim x$ with $n \equiv \gamma \pmod{M}$ with $\binom{n}{k}$ having no prime divisors $\leq z$ and $\binom{n}{k}$ having a divisor in $(\frac{n}{B},n]$, then we wish to show $\mathcal{E} = o\left( \mathfrak{S}_k \frac{x}{M (\log x)^k}\right)$ as $x \rightarrow \infty$.

Assume for the sake of simplicity in this sketch that $d$ is squarefree. By a union bound and swapping the order of summation, we have
\begin{align*}
\mathcal{E} \leq \sum_{\substack{\frac{x}{2B} < d \leq x \\ P^-(d) \geq z}} \mu^2(d) \sum_{\substack{n \sim x \\ n \equiv \gamma (M) \\ d \mid \binom{n}{k} \\ P^-(\binom{n}{k})\geq z}} 1.
\end{align*}
If we introduce an upper-bound sieve with weights $\lambda_g^+$ to control the condition $P^-(\binom{n}{k})\geq z$, we get
\begin{align*}
\mathcal{E} &\leq \sum_{\substack{\frac{x}{2B} < d \leq x \\ P^-(d) \geq z}} \mu^2(d)\sum_{v \in S(d)} \sum_{g \leq G} \lambda_g^+ \sum_{w \in S'(g)} \sum_{\substack{n \sim x \\ n \equiv \gamma (M) \\ n \equiv v (d) \\ n \equiv w (g)}} 1,
\end{align*}
for some sets of residue classes $S(d)$ mod $d$, $S'(g)$ mod $g$, of sizes $k^{\omega(d)}$ and $k^{\omega(g)}$, respectively. We can combine the congruence conditions on $n$ into a single congruence by the Chinese remainder theorem. However, the modulus of the congruence is very large (on the order of $x$, or a bit bigger) so treating the error terms is nontrivial. However, if one ignores any error terms and focuses only on the main term, we have
\begin{align*}
\mathcal{E} &\lessapprox \frac{x}{2M} \left(\sum_{\substack{\frac{x}{2B} < d \leq x \\ P^-(d) \geq z}} \mu^2(d) \frac{k^{\omega(d)}}{d} \right) \left(\sum_{g \leq G} \lambda_g^+ \frac{k^{\omega(g)}}{g} \right).
\end{align*}
When we execute the sum over sieve weights, we obtain
\begin{align*}
\mathcal{E} &\lessapprox \mathfrak{S}_k \frac{x}{M(\log x)^k} \cdot \sum_{\substack{\frac{x}{2B} < d \leq x \\ P^-(d) \geq z}} \mu^2(d) \frac{k^{\omega(d)}}{d}.
\end{align*}
Since $d$ has no small prime factors and we are summing $\approx \frac{1}{d}$ over a (logarithmically) short interval, the sum over $d$ is $o(1)$. Comparing our bound, we conclude that most $n$ do not have a divisor in $(\frac{n}{B},n]$.

Of course, the argument above overlooked one very important point, which is that we cannot just ignore the error terms in estimating $\mathcal{E}$. Since the modulus of the congruence condition on $n$ is the same size as $n$ or larger, one cannot treat the error terms trivially. Rather, one must attempt to get cancellation in the error terms as one sums over the $d$ variable. This task is difficult, because the congruence condition modulo $d$ arises from the Chinese remainder theorem, and therefore depends in a highly nontrivial way on the factors of $d$.

At this point, we recall notation from Section \ref{sec:notation}. Any divisor $d$ of $\binom{n}{k}$ with $P^-(d) \geq z$ decomposes uniquely into divisors 
\begin{align*}
d = d_0d_1 \cdots d_{k-1},
\end{align*}
where $d_i$ divides $n-i$. We assume that $d$ is large (near $x$ in size) and since $n \equiv \gamma \pmod{M}$ it is not possible to have all but one $d_i = 1$, for then the remaining $d_j \mid m_j \frac{n-j}{g_j}$ would be a divisor of $\binom{n}{k}$ in $(n/B,n]$, but each $n-j < \frac{n}{B}$. In other words, there are at least two different $d_i$ that are $\geq z$.

We use different techniques to handle the error terms, depending on the sizes of the $d_i$, and depending on how the $d_i$ factor.

If some $d_i$ is large (a bit bigger than $x^{1/2}$), then we can avoid using exponential sums, and we instead use an elementary divisor-switching argument (Proposition \ref{prop:one very large di bigger than x^1/2}).

If some $d_i$ is a bit bigger than $x^{1/3}$, but of size $\lessapprox x^{1/2}$, then we can use Fourier analysis and completion techniques, along with the Weil bound for Kloosterman sums, to control the error term, as in this case the variable $d_i$ is long enough to get some cancellation (Proposition \ref{prop:some di bigger than x^1/3}). This argument requires ``Type I'' and ``Type II'' estimates for exponential sums arising from a combinatorial decomposition in the spirit of Vaughan's identity or Heath-Brown's identity for the von Mangoldt function (see Chapter 17 of \cite{opera}, particularly \cite[Propositions 17.2 and 17.3]{opera}).

The next situation we consider is one where some $d_i$ has a ``convenient'' factorization $d_i = rs$ (Proposition \ref{prop:some di has convenient factorization}). A convenient factorization is one in which neither $r$ nor $s$ is too small. After applying Fourier analysis, the resulting incomplete Kloosterman sums are too short to treat by completion and the Weil bound. However, thanks to the convenient factorization, we can bound the resulting exponential sums by exploiting a bilinear structure in the error term. The arguments here are inspired by those of Irving \cite{Irving2014} and Bourgain--Garaev \cite{BG2014}. The main idea of the argument is to use H\"older's inequality to duplicate variables and use equidistribution of the sums to replace ``rough'' variables by ``smooth'' variables. This basic idea goes back, at least, to Vinogradov's pioneering work on exponential sums (see \cite[Section 8.5]{IK2004}), but in our particular context the original idea seems to trace back to Karatsuba (see \cite{Kar1995}, \cite[p. 260]{opera}).

If we are in a situation not covered by any of the previous results, then each $d_i$ is $\lessapprox x^{1/3}$, and no $d_i$ has a convenient factorization. Thus, whenever we try to factor $d_i=rs$, one of $r$ or $s$ is small. Hence, each $d_i$ is ``almost'' prime, and for the purposes of this sketch let us assume that each $d_i=q_i$ is a prime that is not too small. If the largest and second-largest of these primes are close together, then we can use a simple estimate to show this contribution must be small (Proposition \ref{prop:exist 2 big primes dividing (n choose k) that are close}). 

Therefore, we can assume that the largest of the primes $q_i$ is a bit bigger than all of the other primes. After combinatorial manipulations and Fourier analysis, we arrive at exponential sums (essentially) of the form
\begin{align*}
\sum_{n \asymp N} e \left(\frac{\overline{n}}{q_1 \cdots q_{k-1}} \right),
\end{align*}
where $N$ is a bit larger than each of $q_1,\ldots,q_{k-1}$, but is much smaller than the product $q_1 \cdots q_{k-1}$. This is an example of a (very) short incomplete Kloosterman sum, and such sums are usually quite difficult to estimate. There are conjectures about bounds for short incomplete Kloosterman sums (see \cite[p. 75, Hypothesis $R$]{Hooley1972}, \cite[p. 44, Hypothesis $R^*$]{Hooley1978}), but in this particular situation the exponential sum can be treated unconditionally by the $q$-van der Corput method of Heath-Brown and Graham and Ringrose (see \cite{GR1990}, \cite{HB2001}). The main point is that the scale $N$ is large enough to accommodate nontrivial shifts in the primes $q_i$, and the $q_i$ can then be successively eliminated with a Cauchy--Schwarz argument (see Proposition \ref{prop:bound for E5 that relies on q van der corput} and Lemma \ref{lem:weyl differencing}).

We remark that we save a small power of $x$ with this $q$-van der Corput method, but the exponent of the savings is exponentially small in $k$. This is the reason we must choose $\epsilon_k$ to be exponentially small in $k$, and this limitation contributes to the upper bound $k \ll (\log \log n)^{1/2}$ in Theorem \ref{thm:main theorem}.

\section{Proof of Theorem \ref{thm:large k regime} }\label{sec:proof of large k thm}

We prove Theorem \ref{thm:large k regime} in this section. We state a proposition, show how to use the proposition to prove the theorem, and then prove the proposition.

\begin{proposition}\label{prop:main prop for large k regime}
Let $\epsilon > 0$ be fixed and sufficiently small. Let $n$ be an integer that is sufficiently large with respect to $\epsilon$. Assume $k$ is an integer satisfying
\begin{align*}
\exp((\log n)^{2/3+\epsilon}) \leq k \leq n^{2/3}.
\end{align*}
Let $J$ be the largest positive integer such that
\begin{align*}
n^{1/J} > \exp((\log n)^{2/3+\epsilon^2}).
\end{align*}
Let $P = n^{1/J}$ and $\lambda = J^{-2}$. Then there are
\begin{align*}
\gg \lambda \frac{P}{\log P}
\end{align*}
primes $p$ in the interval $(\frac{P}{1+\lambda},P]$ such that $p \mid \binom{n}{k}$.
\end{proposition}

Note that $J = \left(\log n \right)^{1/3-\epsilon^2} + O(1)$ in Proposition \ref{prop:main prop for large k regime}.

\begin{proof}[Proof of Theorem \ref{thm:large k regime} assuming Proposition \ref{prop:main prop for large k regime}]
We split into two cases, depending on the size of $k$ with respect to $n$. 

We consider first the case with $k > n^{2/3}$. If $n$ is sufficiently large, then there is a prime $p$ in the interval $[n-\lfloor \frac{1}{2} n^{2/3}\rfloor,n] \subseteq [n-(k-1),n]$ by classical results on primes in short intervals \cite{Hux1972,Ing1937}. Since $n-(k-1) > k$, we see that $p$ divides the numerator but not the denominator of
\begin{align*}
\binom{n}{k} = \frac{n(n-1)\cdots (n-(k-1))}{k!},
\end{align*}
and therefore $p \mid \binom{n}{k}$. Since $n - O(n^{2/3}) < p \leq n$, we see that $\binom{n}{k}$ certainly has a divisor in the interval \eqref{eq:large k divisor interval}.

Now assume $k \leq n^{2/3}$. If $k \geq \exp((\log n)^{2/3+\epsilon})$, then we may apply Proposition \ref{prop:main prop for large k regime} to deduce that $\binom{n}{k}$ has $\gg \lambda \frac{P}{\log P}$ distinct prime divisors in the interval $(\frac{P}{1+\lambda},P]$. Since $n$ is sufficiently large, we can find $J$ distinct primes $p_1,\ldots,p_J$ in this interval. Hence $d=p_1\cdots p_J$ is a divisor of $\binom{n}{k}$, and 
\begin{align*}
(1-O(J^{-1})) n &< \frac{P^J}{(1+\lambda)^J} < d \leq P^J = n. \qedhere
\end{align*}
\end{proof}

We prove Proposition \ref{prop:main prop for large k regime} with exponential sums. We use a classical result due to Kummer \cite[pp. 115--119]{Kum1852} to investigate the $p$-adic valuation of binomial coefficients.

\begin{lemma}[Kummer's theorem]\label{lem:kummers thm}
Let $p$ be a prime. The $p$-adic valuation of the binomial coefficient $\binom{n}{k}$ is equal to the number of carries that occur when $k$ and $n-k$ are added in base $p$.
\end{lemma}

With Lemma \ref{lem:kummers thm} and some basic Fourier analysis, the proof of Proposition \ref{prop:main prop for large k regime} reduces to a bound for some exponential sums over primes.

\begin{proposition}[Exponential sum over primes]\label{prop:exp sum over primes}
Let $\eta > 0$ and $P \geq 2$ and let $I$ be an interval contained in $[P/2,P]$. Let $N$ be a real number with $N=O (\exp(\log^{3/2-\eta}P))$. For all $A > 0$, we have
\begin{align*}
\sum_{p \in I} e \left(\frac{N}{p} \right) = \int_I e \left(\frac{N}{t} \right) \frac{dt}{\log t} + O_{\eta,A}(P \log^{-A}P).
\end{align*}
\end{proposition}
\begin{proof}
This is part (i) of \cite[Proposition 1.13]{MRSTT2022} with $M=0$.
\end{proof}

\begin{proof}[Proof of Proposition \ref{prop:main prop for large k regime}]
By Lemma \ref{lem:kummers thm}, a prime $p \mid \binom{n}{k}$ if there is at least one carry when $k$ is added to $n-k$ in base $p$. If we write $k$ and $n-k$ in base $p$, we have
\begin{align*}
k &= a_0 + a_1 p + \cdots + a_r p^r, \\
n-k &=b_0 + b_1p + \cdots + b_sp^s,
\end{align*}
for some integers $r,s \geq 0$ and $a_i,b_j \in \{0,1,\ldots,p-1\}$. There will be at least one carry when we add $k$ and $n-k$ if $a_0$ and $b_0$ are both $> \frac{p}{2}$. Since
\begin{align*}
\frac{k}{p} \equiv \frac{a_0}{p} \pmod{1} \ \ \  \text{ and } \ \ \  \frac{n-k}{p} \equiv \frac{b_0}{p} \pmod{1},
\end{align*}
we see that $p \mid \binom{n}{k}$ if $\{\frac{k}{p}\}, \{\frac{n-k}{p}\}$ are both $> \frac{1}{2}$.

Let $0 \leq w(\theta) \leq 1$ be a smooth, 1-periodic function that is supported on $\theta$ with $\{\theta\} \in [\frac{3}{5},\frac{9}{10}]$, and which is equal to one on $[\frac{7}{10},\frac{4}{5}]$. We may choose $w$ so that $|w^{(j)}(\theta)| \ll_j 1$ for every $j\geq 0$. Observe that for any $\theta \in \mathbb{R}$ we then have the inequality $\mathbf{1}_{\{\theta\} > \frac{1}{2}} \geq w(\theta)$. The function $w$ has a Fourier expansion
\begin{align*}
w(\theta) = \sum_{h \in \mathbb{Z}} \widehat{w}(h) e(h\theta),
\end{align*}
where
\begin{align*}
\widehat{w}(h) = \int_0^1 w(t) e(-ht) dt.
\end{align*}
We have $\widehat{w}(0) \geq \frac{1}{10}$, and we may choose $w$ so that $|\widehat{w}(h)| \ll_j (1+|h|^j)^{-1}$ for any positive integer $j$.

By the above discussion, the number of primes $p \in (\frac{P}{1+\lambda},P]$ with $p \mid \binom{n}{k}$ is
\begin{align*}
&\geq \mathcal{S} \coloneq \sum_{\frac{P}{1+\lambda} < p \leq P} w \left(\frac{k}{p} \right) w\left(\frac{n-k}{p} \right).
\end{align*}
We expand the $w$ functions in Fourier series and obtain
\begin{align*}
\mathcal{S} &= \mathop{\sum\sum}_{h_1,h_2 \in \mathbb{Z}} \widehat{w}(h_1) \widehat{w}(h_2) \sum_{\frac{P}{1+\lambda} < p \leq P} e \left(\frac{h_1k + h_2(n-k)}{p} \right).
\end{align*}
The main term arises from the term with $h_1 = h_2 = 0$. The main term is
\begin{align*}
= \widehat{w}(0)^2\sum_{\frac{P}{1+\lambda} < p \leq P} 1\geq \frac{\lambda}{200} \frac{P}{\log P}
\end{align*}
by the prime number theorem and the lower bound $\widehat{w}(0) \geq \frac{1}{10}$. If $H = \lambda^{-B}$ for a sufficiently large constant $B$, then the rapid decay of $|\widehat{w}(h)|$ implies the contribution from $|h_i| > H$ is $\leq \frac{\lambda}{400} \frac{P}{\log P}$. It follows that
\begin{align*}
\mathcal{S}\geq \frac{\lambda}{400} \frac{P}{\log P} &- O \left(\sum_{1 \leq |h_1| \leq H}\left|\sum_{\frac{P}{1+\lambda} < p \leq P} e \left(\frac{h_1k}{p} \right)\right| \right) \\ 
&- O \left(\mathop{\sum\sum}_{\substack{|h_1|,|h_2| \leq H \\ h_2 \neq 0}} \left|\sum_{\frac{P}{1+\lambda} < p \leq P} e \left(\frac{h_1k + h_2(n-k)}{p} \right) \right| \right).
\end{align*}
Since $h_1,h_2$ are $\leq \lambda^{-B}$ in absolute value and $k \leq n^{2/3}$, we see that $|h_1k + h_2(n-k)| \asymp n$ when $h_2 \neq 0$. By the definition of $P$, we have by Proposition \ref{prop:exp sum over primes} that each of the sums over $p$ is 
\begin{align*}
=\int_I e \left(\frac{N}{t} \right) \frac{dt}{\log t} + O_{\epsilon,A}(P(\log P)^{-A}),
\end{align*}
say, where $N = h_1k$ or $N=h_1k+h_2(n-k)$, and $A$ is chosen sufficiently large in terms of $B$. Since $|N| \gg k > \exp((\log n)^{2/3+\epsilon})$, integration by parts yields
\begin{align*}
\int_I e \left(\frac{N}{t} \right) \frac{dt}{\log t} \ll \frac{P^2}{k \log P} \ll_{\epsilon,A}\frac{P}{(\log P)^A}.
\end{align*}
It follows that $\mathcal{S} \gg \lambda\frac{P}{\log P}$, as desired.
\end{proof}

\section{The covering problem}\label{sec:covering problem}

In this section, we investigate the covering problem described in the outline in Section \ref{sec:outline}. We encourage the reader to review the sketch of the argument there.

\subsection{Reduction to Proposition \ref{prop:some k has cover of weight B}}

\begin{theorem}\label{thm:main covering theorem}
Let $K$ be sufficiently large, and let $2 \leq B \leq \frac{\log K}{240\log \log K}$. Then there exists $k \sim K$ such that if we set
\begin{align*}
N_k \coloneqq \prod_{p \leq k} p^{\lfloor \frac{\log k}{\log p}\rfloor + 1},
\end{align*}
then the following holds:

There exists a residue class $\alpha_k \pmod{N_k}$ and positive integers $g_0, g_1,\ldots,g_{k-1}$ with $g_i \geq B$ for each $i$ and $\prod_{i=0}^{k-1} g_i = k!$, and for every integer $n > k$ with $n \equiv \alpha_k \pmod{N_k}$ we have
\begin{enumerate}
\item $g_i \mid (n-i)$ for each $i$,
\item $\binom{n}{k}$ is not divisible by any prime $\leq k$,
\item $\binom{n}{k} = \prod_{i=0}^{k-1}\frac{n-i}{g_i}$.
\end{enumerate}
Hence, if we set $n-i=m_i(n)g_i = m_ig_i$, then $\binom{n}{k} = \prod_{i=0}^{k-1} m_i$, each $m_i$ is only divisible by primes $>k$, and $m_i \leq \frac{n}{B}$ for every $i$.
\end{theorem}

\begin{remark}
One should not take the constant $240$ in Theorem \ref{thm:main covering theorem} too seriously, as we have made no attempt to optimize it. 
\end{remark}

Recall from Section \ref{sec:notation} that
\begin{align*}
z_\mathcal{T}(j) \coloneqq \frac{j}{\prod_{p \in \mathcal{T}} p^{v_p(j)}}
\end{align*}
and
\begin{align*}
C_\mathcal{T}(j) \coloneqq z_\mathcal{T}(j) \prod_{\substack{p \in \mathcal{T} \\ j \equiv a_p(p)}} p.
\end{align*}
The function $C_\mathcal{T}(j)$ will give us a measure of whether an integer $j$ has been adequately covered by residue classes modulo primes $p \leq k$.

We have the following important definition.

\begin{definition}\label{defn:cover of weight B}
Let $B \geq 2$ and let $k$ be a positive integer. A \emph{cover of weight $B$} for $k$ is a nonempty set of primes $\mathcal{R} \subseteq \{p \leq k : p \textup{ prime}\}$ together with a nonzero residue class $a_p \pmod{p}$ for each $p \in \mathcal{R}$ such that
\begin{enumerate}
\item $a_p > k \pmod{p}$ for every $p \in \mathcal{R}$,
\item $C_\mathcal{R}(j) \geq B$ for every $1 \leq j \leq k$.
\end{enumerate}
\end{definition}

Definition \ref{defn:cover of weight B} only discusses nonzero residue classes $a_p$, but one should think of Definition \ref{defn:cover of weight B} as giving a covering of integers $1 \leq j \leq k$ by residue classes where $a_p = 0 \pmod{p}$ for primes $p \leq k$ with $p \not \in \mathcal{R}$. Furthermore, since we are trying to distribute $k!$ among the numbers $g_0\cdots g_{k-1}$, we have access to prime powers as well as the primes themselves. The function $z_\mathcal{R}(j)$ measures how much $j$ is covered by the zero residue classes and prime powers; if $z_\mathcal{R}(j) \geq B$, then $j$ is covered enough. ``Most'' integers $j \leq k$ will be sufficiently covered by the zero residue classes, and there will only be a few exceptional $j$'s with $z_\mathcal{R}(j) < B$ that still need to be covered. We essentially cover these exceptional $j$'s one at a time with carefully chosen primes $p$, and this covering with nonzero residue classes is accounted for in the product over primes in the definition of $C_\mathcal{R}(j)$.

\begin{remark}
Schinzel's binomial coefficient $\binom{99215}{15}$ mentioned in the introduction essentially arises from showing that $k=15$ has a covering of weight $2$. This follows from taking $\mathcal{R} = \{5,11\}$ and setting $a_5 = 1, a_{11} = 5$. Schinzel also showed \cite[p. 198]{Schinzel1958} (with different terminology) that $15$ is the least positive integer that has a cover with weight $2$.

It would be interesting to try to determine the size (as a function of $B$) of the least positive integer $k$ that has a covering of weight $B$. Our arguments show this integer has size $\leq B^{O(B)}$, and it may be possible to extend our results to show this integer has size $\leq B^{O(1)}$.
\end{remark}

In order to prove Theorem \ref{thm:main covering theorem}, it suffices to show that there is some $k$ that has a cover of weight $B$. In fact, there are many such $k$.

\begin{proposition}\label{prop:some k has cover of weight B}
Let $K$ be sufficiently large, and assume $B=\frac{\log K}{240\log \log K}$. Then there are $\geq K^{1-o(1)}$ choices of $k \sim K$ such that $k$ has a cover of weight $B$.
\end{proposition}

\begin{proof}[Proof of Theorem \ref{thm:main covering theorem} assuming Proposition \ref{prop:some k has cover of weight B}]
Let $k \sim K$ have a cover of weight $B=\frac{\log K}{240\log \log K}$, and let $\mathcal{R}$ and $a_p \pmod{p}, p \in \mathcal{R}$ be as in Definition \ref{defn:cover of weight B}. Observe that a cover of weight $B$ is also a cover of weight $B'$ for any $B' \leq B$.

For every prime $p \leq k$, we define an exponent $e_p= \lfloor \frac{\log k}{\log p} \rfloor + 1$, and we set
\begin{align*}
N_k = \prod_{p \leq k} p^{e_p}.
\end{align*}
Note that $p^{e_p} > k$ for every $p$.

For $p \in \mathcal{R}$ and an integer $u \geq 1$, we define $a_p^{(u)} = p^u-p+a_p$. Note that $a_p^{(1)} = a_p$. We claim, since $a_p > k \pmod{p}$, that $a_p^{(u)} > k \pmod{p^u}$ for all $1 \leq u \leq e_p$. Let $k \pmod{p} = r$, with $0 \leq r < p$. Any congruence class modulo $p^u$ that is congruent to $r \pmod{p}$ is at most $p^u-p+r$, so
\begin{align*}
k \pmod{p^u} \leq p^u-p+r < p^u-p+a_p = a_p^{(u)}.
\end{align*}
Since $p^{e_p} > k$, this implies in particular that $a_p^{(e_p)} > k$.

We now use the Chinese remainder theorem to create a congruence class $\alpha_k$ modulo $N_k$. If $p \leq k$ with $p \not \in \mathcal{R}$, we require $\alpha_k \equiv k \pmod{p^{e_p}}$. If $p \in \mathcal{R}$, then we require $\alpha_k \equiv k-a_p^{(e_p)} \pmod{p^{e_p}}$. Now fix any integer $n>k$ with $n \equiv \alpha_k \pmod{N_k}$, and we show that this progression has the desired properties.

For $1 \leq j \leq k$, we define
\begin{align*}
\beta_j = \prod_{p \leq k} p^{v_p(n-k+j)}.
\end{align*}
It suffices to prove $\beta_j \geq B$ and $\prod_{j=1}^k \beta_j = k!$, for then the theorem follows by setting $g_i = \beta_{k-i}$. (In the arguments below, it will be slightly more convenient to work with the $\beta_j$'s than the $g_i$'s.)

We first investigate the local valuations. Let $p \leq k$ with $p \not \in \mathcal{R}$. Then $n \equiv k \pmod{p^{e_p}}$, so $n-k+j \equiv j \pmod{p^{e_p}}$. Since $1 \leq j \leq k < p^{e_p}$, we see that $v_p(n-k+j) = v_p(j)$. Now suppose $p \in \mathcal{R}$. Then $n-k+j \equiv j-a_p^{(e_p)} \pmod{p^{e_p}}$. For $1 \leq u < e_p$, we have $a_p^{(e_p)} \equiv a_p^{(u)}\pmod{p^u}$, and therefore $p^u \mid (n-k+j)$ if and only if $j \equiv a_p^{(u)}\pmod{p^u}$.

We next show that each $\beta_j \geq B$. We want to compare $\beta_j$ with the covering weight $C_\mathcal{R}(j)$. First, we observe that $z_\mathcal{R}(j)$ divides $\beta_j$, since
\begin{align*}
z_\mathcal{R}(j) = \prod_{\substack{p \leq k \\ p \not \in \mathcal{R}}} p^{v_p(j)} = \prod_{\substack{p \leq k \\ p \not \in \mathcal{R}}} p^{v_p(n-k+j)}  \mid \beta_j.
\end{align*}
Second, if $p \in \mathcal{R}$ and $j \equiv a_p \pmod{p}$, then $p \mid (n-k+j)$, so the corresponding factor $p$ in the product defining $C_\mathcal{R}(j)$ also divides $\beta_j$. Since $C_\mathcal{R}(j) \geq B$ by assumption, we see each $\beta_j \geq B$.

Lastly, we show that the product of the $\beta_j$ is equal to $k!$. We prove this by comparing the $p$-adic valuations for every prime $p \leq k$.

Suppose $p \not \in \mathcal{R}$. Then $v_p(n-k+j) = v_p(j)$, so
\begin{align*}
\sum_{j=1}^k v_p(n-k+j) = \sum_{j=1}^k v_p(j) = v_p(k!).
\end{align*}

Suppose $p \in \mathcal{R}$. We have $v_p(n-k+j) \geq u$ if and only if $j \equiv a_p^{(u)} \pmod{p^u}$. Since $a_p^{(u)} > k \pmod{p^u}$, the number of such $j$ with $1 \leq j \leq k$ is exactly $\lfloor \frac{k}{p^u} \rfloor$ (here we use that $a_p$ is positive). There are no solutions with $u = e_p$ since $a_p^{(e_p)} > k$. Hence
\begin{align*}
\sum_{j=1}^k v_p(n-k+j) = \sum_{1 \leq u < e_p} \#\{1 \leq j \leq k : p^u \mid(n-k+j)\} = \sum_{1 \leq u < e_p} \left\lfloor \frac{k}{p^u} \right\rfloor.
\end{align*}
Since $p^{e_p} > k$, all terms with $u \geq e_p$ vanish, and the sum is therefore equal to
\begin{align*}
\sum_{u \geq 1}\left\lfloor \frac{k}{p^u} \right\rfloor = v_p(k!),
\end{align*}
by Legendre's formula.

Thus, in either case, we have
\begin{align*}
\sum_{j=1}^k v_p(n-k+j) = v_p(k!),
\end{align*}
and therefore
\begin{align*}
v_p \left(\prod_{j=1}^k \beta_j \right) = v_p(k!).
\end{align*}
Since both sides are composed of primes $\leq k$, we have
\begin{align*}
\prod_{j=1}^k \beta_j = k!,
\end{align*}
as desired.
\end{proof}

\subsection{Tools for the proof of Proposition \ref{prop:some k has cover of weight B}}\label{subsec:tools for proof of covering thm}

Before giving the proof of Proposition \ref{prop:some k has cover of weight B} proper, there are several items of notation we must introduce, and several preliminary results we must prove. We recommend that the reader skim this subsection to see the definitions and results, and to return to read the proofs after seeing how everything comes together to prove Proposition \ref{prop:some k has cover of weight B}.

Throughout, we take $K$ to be sufficiently large, and we define $Y \coloneqq \exp(\frac{\log K}{\log \log K}) = K^{1/\log \log K}$. For this subsection, we do not need the precise form of $B$; we only use the crude bound $B \ll \log K$, and the fact that $B$ is large if $K$ is sufficiently large.

Our first order of business is to obtain a suitable prime $Q \asymp Y^{1/2}$ on which to build our subsequent arguments. The prime $Q$ will be suitable if the count of primes in arithmetic progressions $a \pmod{Q}$ is approximately what we would expect. The existence of such a prime follows almost immediately from the Bombieri--Vinogradov theorem (see, e.g., \cite[Chapter 28]{Dav2000}). However, as the Bombieri--Vinogradov theorem in its standard form involves ineffective constants, we work just a bit harder in order to make our arguments completely effective.

\begin{lemma}\label{lem:bound on card Z}
Let $\mathcal{Z}$ denote the set of those primes $q \in (Y^{1/2},2Y^{1/2}]$ such that there is some primitive residue class $b \pmod{q}$ with
\begin{align*}
\Big|\sum_{\substack{n \sim K \\ n \equiv b (q)}} \Lambda(n) - \frac{K}{2\phi(q)} \Big| > \frac{K}{4\phi(q)}.
\end{align*}
Then $|\mathcal{Z}| \ll Y^{1/2}/(\log Y)^2$, where the implied constant is effectively computable.
\end{lemma}
\begin{proof}
By Markov's inequality, we have
\begin{align*}
|\mathcal{Z}| \leq \frac{8Y^{1/2}}{K} \sum_{\substack{Y^{1/2} < q \leq 2Y^{1/2} \\ q \text{ prime}}} \max_{(a,q)=1} \left|\sum_{\substack{n \sim K \\ n \equiv a (q)}} \Lambda(n) - \frac{K}{2\phi(q)} \right|.
\end{align*}
We introduce Dirichlet characters to detect the congruence condition modulo $q$, and separate the principal character from the non-principal characters. Some easy estimation then gives
\begin{align*}
|\mathcal{Z}| &\ll \frac{Y^{1/2}}{K \log Y}\left|\sum_{\substack{n \sim K}} \Lambda(n) - \frac{K}{2} \right| + \frac{Y^{1/2} (\log K)^2}{K} \\ 
&+ \frac{Y^{1/2}}{K}\sum_{\substack{Y^{1/2} < q \leq 2Y^{1/2} \\ q \text{ prime}}} \frac{1}{\phi(q)}\sideset{}{^*}\sum_{\chi(q)}\left|\sum_{n \sim K} \Lambda(n) \chi(n) \right|,
\end{align*}
where the sum over $\chi$ is a sum over primitive characters; in making this reduction, we have used the fact that we are summing over prime $q$. With Vaughan's identity, the P\'olya-Vinogradov inequality, and the large sieve inequality, one can show
\begin{align*}
\sum_{\substack{Y^{1/2} < q \leq 2Y^{1/2} \\ q \text{ prime}}} \frac{1}{\phi(q)}\sideset{}{^*}\sum_{\chi(q)}\left|\sum_{n \sim K} \Lambda(n) \chi(n) \right| \ll \left(\frac{K}{Y^{1/2}} + K^{5/6} + K^{1/2} Y \right)(\log K)^4,
\end{align*}
where the implied constant is effectively computable (see \cite[equation (2) on p. 162]{Dav2000} for the result, and \cite[pp. 162-168]{Dav2000} for the proof). Combining this with the prime number theorem with classical error term, we obtain
\begin{align*}
|\mathcal{Z}| &\ll \frac{Y^{1/2}}{(\log Y)^2} + (\log K)^4 \ll \frac{Y^{1/2}}{(\log Y)^2}
\end{align*}
with an effectively computable implied constant, as desired.
\end{proof}

By the prime number theorem, there are $\asymp Y^{1/2}/\log Y$ primes in $(Y^{1/2},2Y^{1/2}]$. Since $K$ is sufficiently large, Lemma \ref{lem:bound on card Z} implies we may choose a prime $Q \in (Y^{1/2},2Y^{1/2}]$ such that for every primitive residue class $b \pmod{Q}$ we have
\begin{align*}
\left|\sum_{\substack{n \sim K \\ n \equiv b (Q)}} \Lambda(n) - \frac{K}{2\phi(Q)} \right| \leq \frac{K}{4\phi(Q)}.
\end{align*}
We henceforth fix such a prime $Q$. Observe that for any primitive residue class $b$ modulo $Q$, we have
\begin{align}\label{eq:many primes for good Q}
\frac{K}{5Q \log K} \leq \sum_{\substack{p \sim K \\ p \equiv b(Q)}} 1  \leq \frac{4K}{5Q\log K},
\end{align}
where the sum is over primes $p$.

Having chosen the prime $Q$, we now define a ``deficient'' set
\begin{align}\label{eq:deficient set mathcal D}
\mathcal{D} \coloneqq \{1 \leq d \leq Y : z_{\{Q\}}(d) Q^{\mathbf{1}(d\equiv 1 (Q))} < B\}.
\end{align}
The set $\mathcal{D}$ is a set of small integers that will still need to be covered by residue classes, after we have made an initial choice of some residue classes.

We record here some basic properties of $\mathcal{D}$.

\begin{lemma}\label{lem:basic props of mathcal D}
If $d \in \mathcal{D}$, then $d \not \equiv 1 \pmod{Q}$, and $d = uQ^a$, where $u$ is a positive integer $<B$ and $a$ is a nonnegative integer. Further, we have $|\mathcal{D}| \leq 3B$.
\end{lemma}
\begin{proof}
Since $Q$ is large compared to $B$, we see that if $d \in \mathcal{D}$, then $d \not \equiv 1 \pmod{Q}$. Since $z_{\{Q\}}(d) < B$, if $d \in \mathcal{D}$ then $d = uQ^a$, where $u$ is a positive integer $<B$ and $a$ is a nonnegative integer. Since $Q > Y^{1/2}$, we see that $a \leq 1$, and therefore $|\mathcal{D}| \leq 3B$, say.
\end{proof}

In order to produce suitable coverings of weight $B$, we need sufficiently many primes $\equiv 1 \pmod{Q}$ near $k$, and we need the integers $k-d+1, d \in \mathcal{D}$, to have desirable anatomical properties. The next several results give us the tools we need in this regard.

We begin with a basic double-counting estimate.
\begin{lemma}\label{lem:covering problem raw supply of primes}
We have
\begin{align*}
\sum_{\substack{k \sim K \\ Q \mid k}} \sum_{\substack{k-Y/2 < p < k \\ p \equiv 1 (Q)}} 1 \geq \frac{K}{6Q} \cdot \frac{Y}{2Q\log K},
\end{align*}
where the inner sum is over primes $p$.
\end{lemma}
\begin{proof}
Let $\Omega$ denote the double sum in question. In the inner sum, we write $p=k-h$. Writing $k=p+h$ and swapping the order of summation gives
\begin{align*}
\Omega &= \sum_{\substack{1 \leq h < Y/2 \\ h \equiv -1 (Q)}} \sum_{\substack{K/2-h < p \leq K-h \\ p \equiv 1 (Q)}} 1.
\end{align*}
Trivially estimating the edge contributions from $K/2-h < p \leq K/2$ and $K-h \leq p < K$, we deduce
\begin{align*}
\Omega &= \sum_{\substack{1 \leq h < Y/2 \\ h \equiv -1 (Q)}} \sum_{\substack{K/2< p \leq K \\ p \equiv 1 (Q)}} 1 + O \left(Y^2/Q \right).
\end{align*}
We use \eqref{eq:many primes for good Q} to obtain a lower bound on the sum over $p$, and therefore obtain
\begin{align*}
\Omega &\geq (1+o(1))\frac{Y}{2Q}\cdot \frac{K}{5Q\log K} + O(Y^2/Q) \geq \frac{K}{6Q} \cdot \frac{Y}{2Q\log K},
\end{align*}
as desired.
\end{proof}

We now start counting contributions from $k \sim K$ with various ``bad'' properties. We say $k \sim K$ is an element of $\mathcal{E}_0$ if there exists an integer $t \in [1,Y+1]$ such that $z_{\{Q\}}(k-t+1) < B$.

\begin{lemma}\label{lem:covering problem mathcal E0}
We have
\begin{align*}
\sum_{\substack{k \in \mathcal{E}_0 \\ Q \mid k}} \sum_{\substack{k-Y/2 < p < k \\ p \equiv 1 (Q)}} 1 \ll \frac{BY (\log K)^2}{\log Q} \cdot \frac{Y}{Q\log K}.
\end{align*}
\end{lemma}
\begin{proof}
Let $\Sigma_0$ denote the double sum in question. We ignore the condition that $p$ is prime and count trivially to obtain
\begin{align*}
\Sigma_0 &\ll \frac{Y}{Q}\sum_{\substack{k \in \mathcal{E}_0 \\ Q \mid k}} 1 \leq \frac{Y}{Q}\sum_{k \in \mathcal{E}_0} 1.
\end{align*}
If $k \in \mathcal{E}_0$, then for some $1 \leq t \leq Y+1$ we have $k-t+1= uQ^a$, for integers $1 \leq u < B$ and $a \geq 0$ (compare the proof of Lemma \ref{lem:basic props of mathcal D}). It follows that
\begin{align*}
|\mathcal{E}_0| &\ll Y \cdot B \cdot \frac{\log K}{\log Q},
\end{align*}
and this completes the proof.
\end{proof}

We say $k \sim K$ is an element of the set $\mathcal{E}_1$ if, for some $d \in \mathcal{D}$, there is a prime $q > 2Y$ such that $q^2 \mid(k-d+1)$.

\begin{lemma}\label{lem:covering problem mathcal E1}
We have
\begin{align*}
\sum_{\substack{k \in \mathcal{E}_1 \\ Q \mid k}} \sum_{\substack{k-Y/2 < p < k \\ p \equiv 1 (Q)}} 1 \ll \frac{K}{Q^2}\cdot \frac{Y}{Q\log K}.
\end{align*}
\end{lemma}
\begin{proof}
Let $\Sigma_1$ be the double sum in question. We ignore the condition that $p$ is prime and count trivially to obtain
\begin{align*}
\Sigma_1 &\ll \frac{Y}{Q}\sum_{\substack{k \in \mathcal{E}_1 \\ Q \mid k}} 1 \leq \frac{Y}{Q} \sum_{\substack{k \sim K \\ Q \mid k}} \sum_{d \in \mathcal{D}} \sum_{\substack{q^2 \mid (k-d+1) \\ q>2Y \\ q \text{ prime}}} 1.
\end{align*}
Rearranging the order of summation gives
\begin{align*}
\Sigma_1 &\ll \frac{Y}{Q}\sum_{d \in \mathcal{D}} \sum_{\substack{2Y < q \leq K^{1/2} \\ q \text{ prime}}}\sum_{\substack{k \sim K \\ Q \mid k \\ q^2 \mid (k-d+1)}} 1.
\end{align*}
The integers $Q$ and $q^2$ are coprime, since $Q$ is a prime $\ll Y^{1/2}$ and $q$ is a prime $\gg Y$. We may therefore combine the congruence conditions on $k$ into a single congruence modulo $Qq^2$ by the Chinese remainder theorem. Trivial estimation then gives
\begin{align*}
\Sigma_1 &\ll \frac{Y}{Q}\sum_{d \in \mathcal{D}} \sum_{\substack{2Y < q \leq K^{1/2} \\ q \text{ prime}}} \left(\frac{K}{Qq^2}+1 \right) \ll \frac{Y}{Q} \cdot|\mathcal{D}|\cdot \frac{K}{QY} + \frac{Y}{Q} \cdot |\mathcal{D}| \cdot K^{1/2} \ll \frac{K}{Q^2}\cdot \frac{Y}{Q\log K},
\end{align*}
say.
\end{proof}

We turn to the next case. We say $k \sim K$ is an element of the set $\mathcal{E}_2$ if, for some $d \in \mathcal{D}$, the integer $k-d+1$ is $2Y$-smooth.

\begin{lemma}\label{lem:covering problem mathcal E2}
We have
\begin{align*}
\sum_{\substack{k \in \mathcal{E}_2 \\ Q \mid k}} \sum_{\substack{k-Y/2 < p < k \\ p \equiv 1 (Q)}} 1 \ll \frac{K}{Q (\log K)} \cdot \frac{Y}{Q \log K}.
\end{align*}
\end{lemma}
\begin{proof}
Let $\Sigma_2$ denote the double sum in question. We trivially have
\begin{align*}
\Sigma_2&\ll \frac{Y}{Q}\sum_{\substack{k \in \mathcal{E}_2 \\ Q \mid k}} 1 \leq \frac{Y}{Q}\sum_{\substack{k \sim K \\ Q \mid k}} \sum_{\substack{d \in \mathcal{D} \\ P^+(k-d+1) \leq 2Y}} 1.
\end{align*}
We rearrange the order of summation to make the sum over $k$ the innermost sum, and change variables $m=k-d+1$. After trivially estimating some boundary contributions, we have
\begin{align*}
\Sigma_2 &\ll \frac{Y}{Q} \sum_{d \in \mathcal{D}} \sum_{\substack{m \sim K \\ P^+(m) \leq 2Y \\ m \equiv -d+1 (Q)}} 1 + BY^2/Q.
\end{align*}
Lemma \ref{lem:basic props of mathcal D} implies that $Q \nmid (-d+1)$. Therefore, a result of Granville \cite[Theorem 1]{Gran1993} implies
\begin{align*}
\sum_{\substack{m \sim K \\ P^+(m) \leq 2Y \\ m \equiv -d+1 (Q)}} 1 \ll \frac{\Psi(K,2Y)}{Q},
\end{align*}
with the implied constant being absolute. Since $\frac{\log K}{\log 2Y} \asymp \log \log K$, basic smooth number estimates (see, e.g. \cite[(1.12)]{Gran2008}) imply
\begin{align*}
\Psi(K,2Y) \ll \frac{K}{(\log K)^{3}},
\end{align*}
say. It follows that
\begin{align*}
\Sigma_2 &\ll \frac{K}{Q (\log K)} \cdot \frac{Y}{Q \log K}. \qedhere
\end{align*}
\end{proof}

The next ``bad'' case to consider is more involved than the ones we treated above. We say $k \sim K$ is an element of the set $\mathcal{E}_3$ if, for some $d \in \mathcal{D}$, the integer $k-d+1$ has a prime factor $> \frac{k-Y}{B}$. Before proceeding to the main estimate, we state the following sieve-theoretic result.

\begin{lemma}\label{lem:Halberstam Richert sieve estimate}
Let $a,b,k, \ell$ be integers, and let $y,x$ be real numbers, satisfying $ab \neq 0, \gcd(a,b)=1, 2 \mid ab, \gcd(\ell,k)=1$, $\gcd(a\ell+b,k)=1$, $ 1\leq k < y \leq x$. Then
\begin{align*}
\sum_{\substack{x-y < p \leq x \\ p \equiv \ell (k) \\ ap+b \textup{ prime}}}1 \leq 16 \prod_{p>2} \left(1 - \frac{1}{(p-1)^2}\right)\prod_{\substack{q \mid kab \\ q > 2}} \frac{q-1}{q-2} \cdot \frac{y}{\phi(k)\log^2(y/k)} \left\{1 + O \left(\frac{\log \log (3|ab|(y/k))}{\log(y/k)} \right) \right\},
\end{align*}
where the implied constant in the $O(\cdot)$-notation is absolute.
\end{lemma}
\begin{proof}
This is \cite[Corollary 5.8.1]{HR2011}.
\end{proof}

\begin{lemma}\label{lem:covering problem mathcal E3}
We have
\begin{align*}
\sum_{\substack{k \in \mathcal{E}_3 \\ Q \mid k}} \sum_{\substack{k-Y/2 < p < k \\ p \equiv 1 (Q)}} 1 \leq (24+o(1))\prod_{p>2} \left(1 + \frac{1}{p(p-2)} \right)\cdot \frac{K}{Q} \cdot \frac{Y}{2Q\log K} \cdot \frac{B \log B}{\log K}.
\end{align*}
\end{lemma}
\begin{proof}
Let $\Sigma_3$ denote the double sum in question. We have the upper bound
\begin{align*}
\Sigma_3 &\leq \sum_{\substack{k \sim K \\ Q \mid k}} \sum_{\substack{d \in \mathcal{D} \\ \exists q > (k-Y)/B \text{ such that} \\ q \mid (k-d+1)}} \sum_{\substack{k-Y/2 < p < k \\ p \equiv 1 (Q)}} 1.
\end{align*}
If there is a prime $q > \frac{k-Y}{B}$ that divides $k-d+1$, then $k-d+1 = aq$ for some positive integer $a < (1+o(1))B$. Therefore, upon inserting a sum over $a$ for an upper bound, we have
\begin{align*}
\Sigma_3 &\leq \sum_{a < (1+o(1))B} \sum_{d \in \mathcal{D}} \sum_{\substack{k \sim K \\ Q \mid k \\ a \mid (k-d+1) \\ \frac{k-d+1}{a} \text{ prime}}} \sum_{\substack{k-Y/2 < p < k \\ p \equiv 1 (Q)}} 1.
\end{align*}
We change variables $k=p+h$, swap the order of summation, and trivially estimate some boundary terms to obtain
\begin{align*}
\Sigma_3 &\leq \sum_{a < (1+o(1))B} \sum_{d \in \mathcal{D}} \sum_{\substack{1 \leq h < Y/2 \\ h \equiv -1 (Q)}} \sum_{\substack{p \sim K \\ p \equiv 1 (Q) \\ p \equiv d-h-1 (a) \\ \frac{p+h-d+1}{a} \text{ prime}}} 1 + O(B^2Y^2/Q).
\end{align*}
In the inner sum over $p$, we now change variables $q = \frac{p+h-d+1}{a}$, where $q$ is a prime. This yields
\begin{align*}
\Sigma_3 &\leq \sum_{a < (1+o(1))B} \sum_{d \in \mathcal{D}} \sum_{\substack{1 \leq h < Y/2 \\ h \equiv -1 (Q)}} \sum_{\substack{q \sim K/a \\ q \equiv \overline{a}(-d+1) (Q) \\ aq+d-h-1 \text{ prime}}} 1 + O(B^2Y^2/Q).
\end{align*}
If $d-h-1 = 0$, then $aq$ is never prime unless $a=1$. The contribution from terms with $d-h-1=0$ to $\Sigma_3$ is, by \eqref{eq:many primes for good Q},
\begin{align*}
\leq \sum_{d \in \mathcal{D}}\sum_{\substack{1 \leq h < Y/2 \\ h \equiv -1(Q) \\ h=d-1}} \sum_{\substack{q \sim K \\ q \equiv -d+1 (Q)}} 1 \ll \frac{BK}{Q\log K}.
\end{align*}
We have also used Lemma \ref{lem:basic props of mathcal D} here to show $-d+1$ is a primitive residue class modulo $Q$.

It therefore suffices to consider terms with $d-h-1 \neq 0$. If $\gcd(a,d-h-1) \neq 1$, then it is not possible for $aq+d-h-1$ to be prime, and the sum over $q$ is empty, so we may assume $\gcd(a,d-h-1)=1$. Similarly, if $a(d-h-1)$ is odd, then $aq+d-h-1$ is always even and cannot be prime, so we may assume $a(d-h-1)$ is even. 

With the above conditions in play, Lemma \ref{lem:Halberstam Richert sieve estimate} yields
\begin{align*}
\sum_{\substack{q \sim K/a \\ q \equiv \overline{a}(-d+1) (Q) \\ aq+d-h-1 \text{ prime}}} 1 &\leq (8+o(1)) \prod_{p > 2} \left(1 - \frac{1}{(p-1)^2} \right) \cdot \prod_{\substack{q \mid Qa(d-h-1) \\ q > 2}} \frac{q-1}{q-2} \cdot \frac{K}{a\phi(Q)\log^2(K/(2aQ))} \\
&= (8+o(1)) \prod_{p > 2} \left(1 - \frac{1}{(p-1)^2} \right) \cdot \prod_{\substack{q \mid a(d-h-1) \\ q > 2}} \frac{q-1}{q-2} \cdot \frac{K}{aQ\log^2(K)}.
\end{align*}
Therefore
\begin{align*}
\Sigma_3 &\leq (8+o(1))\cdot \prod_{p > 2} \left(1 - \frac{1}{(p-1)^2} \right) \cdot \frac{K}{Q\log^2(K)} \sum_{a < (1+o(1))B} \frac{j(a)}{a}\sum_{d \in \mathcal{D}}\sum_{\substack{1 \leq h < Y/2 \\ h \equiv -1 (Q) \\ h \neq d-1}} j(d-h-1),
\end{align*}
where we have written
\begin{align*}
j(n) = j(|n|) = \prod_{\substack{q \mid |n| \\ q > 2}}\frac{q-1}{q-2}
\end{align*}
and dropped some conditions for an upper bound.

It remains to execute the sums over $a,d,h$. It is a straightforward analytic number theory exercise (or see below for hints) to show
\begin{align*}
\sum_{a < (1+o(1))B} \frac{j(a)}{a} = (1+o_{B \rightarrow \infty}(1)) \prod_{p > 2} \left(1 + \frac{1}{p(p-2)} \right) \log B.
\end{align*}
Hence
\begin{align*}
\Sigma_3 &\leq (8+o(1)) \frac{K\log B}{Q\log^2(K)}\sum_{d \in \mathcal{D}}\sum_{\substack{1 \leq h < Y/2 \\ h \equiv -1 (Q) \\ h \neq d-1}} j(d-h-1),
\end{align*}
since
\begin{align*}
\left( 1 - \frac{1}{(p-1)^2}\right)\left(1 + \frac{1}{p(p-2)} \right) = 1.
\end{align*}
To handle the arithmetic function $j$, we define $\phi^*(n)$ to be a multiplicative function defined on squarefree odd numbers with $\phi^*(p)=p-2$, and observe that
\begin{align*}
j(n) = \sum_{\substack{e \mid n \\ e \text{ odd}}} \frac{\mu^2(e)}{\phi^*(e)}.
\end{align*}
It follows that
\begin{align*}
\sum_{d \in \mathcal{D}}\sum_{\substack{1 \leq h < Y/2 \\ h \equiv -1 (Q) \\ h \neq d-1}} j(d-h-1) &= \sum_{d \in \mathcal{D}} \sum_{\substack{e \leq Y \\ e \text{ odd}}} \frac{\mu^2(e)}{\phi^*(e)} \sum_{\substack{1 \leq h < Y/2 \\ h \equiv -1 (Q) \\ h \equiv d-1 (e) \\ h \neq d-1}} 1.
\end{align*}
If $Q$ and $e$ are coprime, then we may combine the congruence conditions on $h$ into a single congruence condition modulo $Qe$. If $Q \mid e$, then since $h \equiv -1 \pmod{Q}$ we must have $Q \mid d$. Splitting into cases depending on whether or not $Q$ divides $e$, we see the sum over $d$ and $h$ is
\begin{align*}
&\sum_{d \in \mathcal{D}} \sum_{\substack{e \leq Y \\ e \text{ odd} \\ Q \nmid e}} \frac{\mu^2(e)}{\phi^*(e)} \sum_{\substack{1 \leq h < Y/2 \\ h \equiv -1 (Q) \\ h \equiv d-1 (e) \\ h \neq d-1}} 1 + \sum_{\substack{d \in \mathcal{D} \\ Q \mid d}} \sum_{\substack{e \leq Y \\ e \text{ odd} \\ Q \mid e}} \frac{\mu^2(e)}{\phi^*(e)} \sum_{\substack{1 \leq h < Y/2 \\ h \equiv d-1 (e) \\ h \neq d-1}} 1.
\end{align*}
The first sum over $d$ and $h$ is
\begin{align*}
\frac{Y}{2Q}\sum_{d \in \mathcal{D}}\sum_{\substack{e \leq Y \\ e \text{ odd} \\ Q \nmid e}} \frac{\mu^2(e)}{e\phi^*(e)} + O(B\log Y) \leq (3+o(1))\prod_{p>2} \left(1 + \frac{1}{p(p-2)} \right) \cdot \frac{BY}{2Q},
\end{align*}
where we have used Lemma \ref{lem:basic props of mathcal D} to bound $|\mathcal{D}|$. The second sum over $d$ and $h$, where $Q \mid d$ and $Q \mid e$, is
\begin{align*}
&\leq \sum_{\substack{d \in \mathcal{D} \\ Q \mid d}} \sum_{\substack{e \leq Y \\ e \text{ odd} \\ Q \mid e}} \frac{\mu^2(e)}{\phi^*(e)} \left(\frac{Y}{2e} + O(1) \right) \ll \frac{BY}{Q^2} + B \log Y \ll B \log Y.
\end{align*}
It follows that
\begin{align*}
\Sigma_3 &\leq (24+o(1))\prod_{p>2} \left(1 + \frac{1}{p(p-2)} \right)\cdot \frac{K}{Q} \cdot \frac{Y}{2Q\log K} \cdot \frac{B \log B}{\log K},
\end{align*}
as desired.
\end{proof}

We have one remaining case to consider. Let us say $k \in \mathcal{E}_4$ if $k \sim K$ and
\begin{align*}
\sum_{\substack{k-Y/2 < p < k \\ p \equiv 1 (Q)}} 1 \leq \frac{Y^{3/4}}{Q}.
\end{align*}
The following lemma is immediate.

\begin{lemma}\label{lem:covering problem mathcal E4}
We have
\begin{align*}
\sum_{\substack{k \in \mathcal{E}_4 \\ Q \mid k}}\sum_{\substack{k-Y/2 < p < k \\ p \equiv 1 (Q)}} 1 \ll \frac{K}{Q} \cdot \frac{Y^{3/4}}{Q}.
\end{align*}
\end{lemma}

\subsection{The proof of Proposition \ref{prop:some k has cover of weight B}}

Now that all our ingredients are in place, we can prove Proposition \ref{prop:some k has cover of weight B}.

\begin{proof}[Proof of Proposition \ref{prop:some k has cover of weight B}]
Let $K$ be sufficiently large. We assume all the notation and results in Subsection \ref{subsec:tools for proof of covering thm}. We wish to prove there are $\geq K^{1-o(1)}$ different $k \sim K$ that have a cover of weight $B=\frac{\log K}{240 \log \log K}$ (recall Definition \ref{defn:cover of weight B}).

Let $Q$ be a prime in $(Y^{1/2},2Y^{1/2}]$ that satisfies \eqref{eq:many primes for good Q}. The prime $Q$ will be an element of any future set of primes $\mathcal{R}$, and we set $a_Q = 1 \pmod{Q}$. We will work with $k$'s such that $Q \mid k$, so $a_Q > k \pmod{Q}$. 

If we think of choosing $a_p = 0 \pmod{p} $ for all $p \leq Y$ with $p \neq Q$, then we obtain the set $\mathcal{D}$ of \eqref{eq:deficient set mathcal D}; this is the set of integers $j \leq Y$ that are not yet sufficiently covered by residue classes and prime powers (a rigorous accounting is given at the end of the proof, so these kinds of remarks should be viewed as motivational only).

Recall the sets $\mathcal{E}_0,\mathcal{E}_1,\mathcal{E}_2, \mathcal{E}_3,\mathcal{E}_4$ introduced in Subsection \ref{subsec:tools for proof of covering thm}. We write $\mathcal{E} = \mathcal{E}_0\cup\mathcal{E}_1 \cup \mathcal{E}_2 \cup \mathcal{E}_3\cup \mathcal{E}_4$, and consider $k \sim K$ with $Q \mid k$ such that $k \not \in \mathcal{E}$. We note that
\begin{align*}
\prod_{p>2} \left(1 + \frac{1}{p(p-2)}\right) = 1.514\ldots,
\end{align*}
so by Lemmas \ref{lem:covering problem raw supply of primes}, \ref{lem:covering problem mathcal E0}, \ref{lem:covering problem mathcal E1}, \ref{lem:covering problem mathcal E2}, \ref{lem:covering problem mathcal E3}, and \ref{lem:covering problem mathcal E4} we have
\begin{align*}
\sum_{\substack{k \sim K \\ k \not \in \mathcal{E} \\ Q \mid k}} \sum_{\substack{k-Y/2 < p < k \\ p \equiv 1 (Q)}} 1 \geq \frac{K}{100Q} \cdot \frac{Y}{2Q\log K}.
\end{align*}
Since we trivially have
\begin{align*}
\sum_{\substack{k-Y/2 < p < k \\ p \equiv 1 (Q)}} 1 \ll \frac{Y}{Q}
\end{align*}
for every $k$, we see that
\begin{align*}
\sum_{\substack{k \sim K \\ k \not \in \mathcal{E} \\ Q \mid k}}  1 \gg \frac{K}{Q\log K} \geq K^{1-o(1)}.
\end{align*}
In order to complete the proof of the theorem, it suffices to show that if $k \sim K, k \not \in \mathcal{E}$, and $Q \mid k$, then $k$ has a cover of weight $B$.

Fix one such $k$ for the remainder of the proof. Consider now any $d \in \mathcal{D}$. Since $k \not \in \mathcal{E}_2$, there is some prime $q>2Y$ that divides $k-d+1$. Since $k \not \in \mathcal{E}_1$, we see that $q_d^2 \nmid (k-d+1)$, and, since $k \not \in \mathcal{E}_3$, we see that $q \leq \frac{k-Y}{B}$. For each $d$, we may therefore choose a prime divisor $q_d$ of $k-d+1$ such that $q_d \in (2Y,\frac{k-Y}{B}]$ and $q_d^2 \nmid (k-d+1)$.  Observe that if $d,d'$ are two distinct elements of $\mathcal{D}$, then $q_d$ and $q_{d'}$ are distinct, since $d,d' \leq Y$ and $q_d,q_{d'}>2Y$.

We can set $a_{q_d} = d \pmod{q_d}$ to cover $d \in \mathcal{D}$. Indeed, since $q_d \mid (k-d+1)$ we have $k \pmod{q_d} = d-1 < a_{q_d}$, as required for Definition \ref{defn:cover of weight B}. Furthermore, the fact that each $q_d > 2Y$ means $d \in \mathcal{D}$ will be covered ``enough'' by $q_d$.

Let $\mathcal{S} = \{Q\} \cup \{q_d : d \in \mathcal{D}\}$. In terms of covering by residue classes, we now choose $a_p = 0 \pmod{p}$ for all $p \leq k$ with $p \not \in \mathcal{S}$ and $p \not \in \{k-Y/2 < q < k : q \text{ prime}, q \equiv 1 \pmod{Q}\}$. The set of integers that are still not covered enough at this stage is
\begin{align*}
\mathcal{U} = \mathcal{U}(k) = \{1 \leq j \leq k : z_\mathcal{S}(j) < B, j \not \equiv 1 \pmod{Q}, j\not \equiv d \pmod{q_d} \text{ for any }d \in \mathcal{D}\}.
\end{align*}
We make several claims about $\mathcal{U}$.

First, we claim that $\mathcal{U}$ has no elements $\leq Y$. Suppose for contradiction that there exists $j \in \mathcal{U} \cap [1,Y]$. Since every $q_d > 2Y$, we see that $j$ is not divisible by any $q_d$, and therefore $z_\mathcal{S}(j) = z_{\{Q\}}(j) < B$. Since $j \not \equiv 1 \pmod{Q}$, we see from \eqref{eq:deficient set mathcal D} that $j \in \mathcal{D}$. So $j =d$ for some $d \in \mathcal{D}$, but then we have $j \equiv d \pmod{q_d}$, a contradiction.

Second, we claim that $\mathcal{U}$ has no elements $>k-Y$. Suppose for contradiction there exists $j \in \mathcal{U} \cap (k-Y,k]$. We write $j=k-t+1$, where $1 \leq t < Y+1$. Since $k \not \in \mathcal{E}_0$, we have $z_{\{Q\}}(k-t+1) \geq B$. If $k-t+1$ is not divisible by any $q_d$, then $z_\mathcal{S}(k-t+1) = z_{\{Q\}}(k-t+1) \geq B$, contradicting the definition of $\mathcal{U}$. Thus, we may assume there is some $d \in \mathcal{D}$ with $q_d \mid (k-t+1)$. Since we also have $q_d \mid (k-d+1)$, we see that $t=d$, since $t < Y+1, d \leq Y$, and $q_d > 2Y$. We similarly see that, for any other $d' \in \mathcal{D}$ with $d' \neq d$, we have $q_{d'} \nmid (k-d+1)$. Thus, we may write $j=k-d+1 = m q_d$, where $m$ is not divisible by any $q_{d'}$ for $d' \in \mathcal{D}$ (here we use that $k \not \in \mathcal{E}_1$). By Lemma \ref{lem:basic props of mathcal D}, we find $Q \nmid m$ (since $Q \mid k$), and therefore $z_\mathcal{S}(j) = z_\mathcal{S}(k-d+1) = m$. Since $k \not \in \mathcal{E}_3$, we have
\begin{align*}
k-Y < j = mq_d \leq m \cdot \frac{k-Y}{B},
\end{align*}
so $z_\mathcal{S}(j) = m \geq B$, and this contradicts the definition of $\mathcal{U}$.

Our third claim is that
\begin{align*}
|\mathcal{U}| &\leq \exp (O((\log \log K)^2)).
\end{align*}
If $j \leq k \leq K$ with $z_\mathcal{S}(j) < B$, then we have
\begin{align*}
j = u Q^a \prod_{d \in \mathcal{D}} q_d^{e_d},
\end{align*}
where $u$ is a positive integer $<B$ and $a,e_d$ are nonnegative integers. The number of choices for $a$ is trivially $\ll \frac{\log K}{\log Y} = \log \log K$. Since each $q_d > Y$, we see that
\begin{align*}
\sum_{d \in \mathcal{D}} e_d \leq \frac{\log K}{\log Y} = \log \log K,
\end{align*}
and by ``stars and bars'' the number of tuples $(e_d : d \in \mathcal{D})$ satisfying this constraint is
\begin{align*}
\leq \binom{|\mathcal{D}|+\lceil \log \log K \rceil}{\lceil \log \log K \rceil}\leq (3B+\log \log K+1)^{\log \log K+1},
\end{align*}
where we have used Lemma \ref{lem:basic props of mathcal D} to get an upper bound on $|\mathcal{D}|$. Multiplying the counts for $u,a$, and $(e_d)$ gives the desired bound.

Since $k \not \in \mathcal{E}_4$, we have
\begin{align*}
\sum_{\substack{k-Y/2 < p < k \\ p \equiv 1 (Q)}} 1 > \frac{Y^{3/4}}{Q} \gg Y^{1/4}.
\end{align*}
By the claim above we have $|\mathcal{U}| = o(Y^{1/4})$, so we may choose some set $\mathcal{P}$ of primes $p \equiv 1 \pmod{Q}$ with $k-Y/2 < p < k$ such that $|\mathcal{P}| = |\mathcal{U}|$. For every $p \equiv 1 \pmod{Q}$ with $k-Y/2 < p < k$ and $p \not \in \mathcal{P}$ we set $a_p = 0 \pmod{p}$. 

We may bijectively pair the elements of $\mathcal{U}$ and $\mathcal{P}$, and for each $j \in \mathcal{U}$ we choose $a_{p(j)} = j \pmod{p}$ for the corresponding $p(j) \in \mathcal{P}$. (Note that if $j \in \mathcal{U}$ then by a claim above we have $j \leq k-Y < k-Y/2 < p$ for all $p \in \mathcal{P}$.) We then claim that $k \pmod{p} < a_p$ for every $p \in \mathcal{P}$. If $p \in \mathcal{P}$, then $k/2 < k-Y/2 < p < k$, so $k \pmod{p} = k-p < k-(k-Y/2) = Y/2 < Y < j$ for any $j \in \mathcal{U}$, by a claim above.

The final choice $\mathcal{R}$ of a set of primes is
\begin{align*}
\mathcal{R} = \{Q\} \cup \{q_d : d \in \mathcal{D}\} \cup \{p : p \in \mathcal{P}\},
\end{align*}
with corresponding nonzero residue classes $a_Q = 1, a_{q_d} = d$, $a_{p(j)} = j$. The ranges in which these primes live are disjoint ( $(Y^{1/2},2Y^{1/2}]$, $(2Y,\frac{k-Y}{B}]$, and $(k-Y/2,k)$) so the primes are all distinct. We have already seen that $k \pmod{p} < a_p$ for every $p \in \mathcal{R}$, so it remains to show that $C_\mathcal{R}(j) \geq B$ for every $1 \leq j \leq k$.

Fix $1 \leq j \leq k$. If $j \equiv 1 \pmod{Q}$, then $C_\mathcal{R}(j) \geq Q > Y^{1/2} > B$. If $j \equiv d \pmod{q_d}$ for some $d \in \mathcal{D}$, then $C_\mathcal{R}(j) \geq q_d > 2Y > B$. If $j \in \mathcal{U}$, then $j \equiv a_p \pmod{p}$ for some $p \in \mathcal{P}$, so $C_\mathcal{R}(j) \geq p > k-Y/2 > B$. Finally, if $j \not \equiv 1 \pmod{Q}, j \not \equiv d \pmod{q_d}$ for any $d \in \mathcal{D}$, and $j \not \in \mathcal{U}$, then by definition of $\mathcal{U}$ we have $C_\mathcal{R}(j) \geq z_\mathcal{R}(j) = z_\mathcal{S}(j) \geq B$. (We cannot have $p \mid j$ for any $p \in \mathcal{P}$, because otherwise we would have $j=p \equiv 1 \pmod{Q}$.)
\end{proof}

\begin{remark}
We have formalized in Lean a weaker version of Proposition \ref{prop:some k has cover of weight B}, in which $B$ is fixed rather than growing with $K$. See our GitHub repository \cite{Slava2026} for further information and discussion.
\end{remark}

\section{Key propositions for the divisor problem}\label{sec:key props for divisor problem}

By Theorem \ref{thm:main covering theorem}, for any large $K$ we can find $k \in (\frac{K}{2},K]$ and an arithmetic progression $n \equiv \alpha_k \pmod{N_k}$ such that for any $n$ in this arithmetic progression the binomial coefficient
\begin{align*}
\binom{n}{k} = \frac{n(n-1)(n-2) \cdots (n-(k-1))}{k!}
\end{align*}
is not divisible by any primes $\leq k$. Moreover, each term $n-i, 0 \leq i \leq k-1$, is divisible by $g_i > B \coloneqq \frac{\log k}{241 \log \log k}$.  The modulus $N_k$ is given by
\begin{align*}
N_k = \prod_{p \leq k} p ^{\lfloor \frac{\log k}{\log p}\rfloor + 1}.
\end{align*}
By the prime number theorem, we have
\begin{align*}
e^{(1-o(1))k}\leq N_k \leq e^{(2+o(1))k}.
\end{align*}
We fix one such large $k$, with all the above properties, for all of our forthcoming arguments.

In order to simplify slightly some sieve computations, it will be helpful for us if $\binom{n}{k}$ is not divisible by any primes in $(k,2k)$. This can be accomplished by imposing the conditions $n \equiv k \pmod{q}$ for every prime $q \in (k,2k)$. We therefore prefer to work with $n$ in an arithmetic progression modulo $M = M_k$, where
\begin{align*}
M \coloneqq N_k \cdot \prod_{k < q < 2k} q.
\end{align*}
By the prime number theorem, we have
\begin{align*}
e^{(2-o(1))k} \leq M \leq e^{(3+o(1))k}.
\end{align*}
The arithmetic progression in which we work will be $n \equiv \gamma \pmod{M}$, where $\gamma$ arises via the Chinese remainder theorem as the solution to the congruences
\begin{align*}
\gamma &\equiv \alpha_k \pmod{N_k}, \\
\gamma &\equiv k \pmod{q} \text{ for all }q \in (k,2k).
\end{align*}

We now introduce some further notation. We will let $x$ be a large real number. We always consider $x$ to be much larger than $k$, which is itself a large positive integer. In particular, we require
\begin{align}\label{eq:upp bound on k in terms of x}
k \leq \delta (\log \log x)^{1/2},
\end{align}
where $\delta>0$ is a sufficiently small positive constant. We wish to study $n \sim x$ with $n \equiv \gamma \pmod{M}$ such that $\binom{n}{k}$ has no prime divisors $< z$, where $z \coloneqq x^{\epsilon_k}$ is some small power of $x$. Here $\epsilon_k>0$ can be any small number that is sufficiently small depending on $k$, but for definiteness we take
\begin{align}\label{eq:defn of epsilon_k}
\epsilon_k \coloneqq 3^{-k}.
\end{align}

We consider the sum
\begin{align}\label{eq:defn of mathcal S}
\mathcal{S} := \sum_{\substack{n \sim x \\ n \equiv \gamma (M) \\ P^-(\binom{n}{k})\geq z}} 1.
\end{align}
We can get a good lower bound on $\mathcal{S}$ with a sieve method. (In fact, with more effort we could get an asymptotic for $\mathcal{S}$, but we will not need this extra precision.)

\begin{proposition}\label{prop:main term lower bound for S}
Assume all the notation above. Then
\begin{align*}
\mathcal{S} &\geq (1+o_{k \rightarrow \infty}(1)) \mathfrak{S}_k \frac{x}{2M} (\log z)^{-k},
\end{align*}
where
\begin{align}\label{eq:defn of singular series}
\mathfrak{S}_k = e^{-c k} \prod_{p < 2k} \left(1 - \frac{1}{p} \right)^{-k} \cdot \prod_{p \geq 2k} \frac{1 - \frac{k}{p}}{\left(1 - \frac{1}{p} \right)^k}
\end{align}
with $c>0$ a positive constant.
\end{proposition}

In actuality, the positive constant $c$ in Proposition \ref{prop:main term lower bound for S} is the Euler--Mascheroni constant $\gamma$, but we use a different letter to avoid conflict with our notation $n\equiv \gamma \pmod{M}$.

Define the sum
\begin{align}\label{eq:defn of mathcal E}
\mathcal{E} &:= \sum_{\substack{n \sim x \\ n \equiv \gamma (M) \\ P^-(\binom{n}{k})\geq z \\ \exists d \in (n/B,n] \text{ s.t.} \\ d \mid \binom{n}{k}}} 1.
\end{align}
In order to show there are infinitely many $\binom{n}{k}$ having no divisor in $(\frac{n}{B},n]$, it suffices to show that $\mathcal{E} = o(\mathcal{S})$.

We recall some notation introduced in Section \ref{sec:notation}. Any divisor $d$ of $\binom{n}{k}$ with $P^-(d) \geq z$ decomposes uniquely into divisors 
\begin{align*}
d = d_0d_1 \cdots d_{k-1},
\end{align*}
where $d_i$ divides $n-i$. Thus, $n \equiv i \pmod{d_i}$. Observe that the $d_i$ are pairwise coprime. These congruence conditions combine to give a congruence condition $n \equiv \gamma_d \pmod{d}$ through the Chinese remainder theorem. Each divisor $d_i$ is either $1$, or is $\geq z$. We assume that $d$ is large (near $x$ in size) and since $n \equiv \gamma \pmod{M}$ it is not possible to have all but one $d_i = 1$, for then the remaining $d_j \mid \frac{n-j}{g_j}$ would be a divisor of $\binom{n}{k}$ in $(n/B,n]$, but each $\frac{n-j}{g_j} < \frac{n}{B}$. In other words, there are at least two different $d_i$ that are $\geq z$. It will sometimes be notationally beneficial to gather together all but one of the $d_i$ into a divisor, which we shall denote by $D$, so that $d = Dd_i$ for some $i$. If we write $d$ in this way, then the congruence condition on $n$ is equivalent to the pair of congruences
\begin{align*}
n \equiv \gamma_D \pmod{D}, \ \ \ \ n \equiv i \pmod{d_i},
\end{align*}
where $\gamma_D$ arises from the Chinese remainder theorem in the natural way.

We use different techniques to handle the error terms, depending on the sizes of the $d_i$, and depending on how the $d_i$ factor. The most elementary case is when there is one $d_i$ that is large. By ``large'' we mean something a bit bigger than $x^{1/2}$.

\begin{proposition}\label{prop:one very large di bigger than x^1/2}
Define
\begin{align*}
\mathcal{E}_1 &:= \sum_{\substack{n \sim x \\ n \equiv \gamma (M) \\ P^-(\binom{n}{k})\geq z \\ \exists d \in (n/B,n] \textup{ such that} \\ d \mid \binom{n}{k} \textup{ and } \\ \exists d_i > x^{101/200}}} 1.
\end{align*}
We have $\mathcal{E}_1 = o (\mathcal{S})$.
\end{proposition}

The error term $\mathcal{E}_1$ can be handled using elementary methods; the main idea is to switch $d_i$ to a complementary divisor. By contrast, most of our later work will require various estimates for exponential sums.

Proposition \ref{prop:one very large di bigger than x^1/2} handles the situation where some $d_i$ is bigger than $x^{1/2}$. The next-easiest situation to handle is when there is some $d_i$ that is a bit larger than $x^{1/3}$.

\begin{proposition}\label{prop:some di bigger than x^1/3}
Define
\begin{align*}
\mathcal{E}_2 &:= \sum_{\substack{n \sim x \\ n \equiv \gamma (M) \\ P^-(\binom{n}{k})\geq z \\ \exists d \in (n/B,n] \textup{ such that} \\ d \mid \binom{n}{k} \textup{ and } \\ \exists d_i \textup{ such that } x^{101/300} < d_i \leq x^{101/200}}} 1.
\end{align*}
We have $\mathcal{E}_2 = o (\mathcal{S})$.
\end{proposition}

The proof of Proposition \ref{prop:some di bigger than x^1/3} relies on the Weil bound for Kloosterman sums. The proof is not so difficult, since $d_i > x^{101/300}$ is long, and we can get cancellation by summing over this variable.

The next situation we consider is one where some $d_i$ has a ``convenient'' factorization $d_i = rs$. A convenient factorization is one in which neither $r$ nor $s$ is too small.

\begin{proposition}\label{prop:some di has convenient factorization}
Let $x,k,z,d_i$ be as above. Define
\begin{align*}
\mathcal{E}_3 &:= \sum_{\substack{n \sim x \\ n \equiv \gamma (M) \\ P^-(\binom{n}{k})\geq z \\ \exists d \in (n/B,n] \textup{ such that} \\ d \mid \binom{n}{k} \textup{ and }d_i \leq x^{101/300} \, \forall i \textup{ and } \\ \exists d_i=rs \textup{ with } r,s > x^{k^{-100}}}} 1.
\end{align*}
We have $\mathcal{E}_3 \ll (\frac{1}{k}+o(1)) \mathcal{S}$.
\end{proposition}

If we are not in a situation handled by any of the preceding propositions, then each $d_i \leq x^{101/300}$ and whenever we factor $d_i = rs$ we have one of $r$ or $s$ is $\leq x^{1/k^{100}}$. We factor $d_i$ into its prime factors $\leq x^{1/k^{100}}$, and its prime factors $> x^{1/k^{100}}$. By assumption, each $d_i$ can have at most one prime factor that is $> x^{1/k^{100}}$. Moreover, if we let $f_i$ denote the product of all primes dividing $d_i$ that are $\leq x^{1/k^{100}}$, then we must have $f_i \leq x^{3/k^{100}}$, otherwise $d_i$ will have a convenient factorization in the sense described above. 

Therefore, it remains to consider the situation in which each $d_i = f_i q_i$, where every $f_i$ is small and each $q_i$ is either $1$ or is a distinct prime $> x^{1/k^{100}}$. (Note that each $q_i \leq x^{101/300}$.) We think of this as saying that each $d_i$ is ``almost'' prime. It is useful in this situation to think of the sizes of the largest and second-largest factors $q_{i_0}$ and $q_{j_0}$ (say). Our forthcoming exponential sum argument will only work if $q_{j_0}$ is a fair bit smaller than $q_{i_0}$. If this fails, however, then $\binom{n}{k}$ will have two prime factors that are fairly close to each other, and this is a rare event that can be handled in an elementary fashion.

\begin{proposition}\label{prop:exist 2 big primes dividing (n choose k) that are close}
Define
\begin{align*}
\mathcal{E}_4 &:= \sum_{\substack{n \sim x \\ n \equiv \gamma (M) \\ P^-(\binom{n}{k})\geq z \\ \exists \textup{ primes } q,r \mid \binom{n}{k} \textup{ with } \\ x^{(2k)^{-1}} < r < q \leq x^{101/300} \\ r > q x^{-k^{-4}}}} 1.
\end{align*}
We have $\mathcal{E}_4 = o (\mathcal{S})$.
\end{proposition}

Let $q_{i_0}$ and $q_{j_0}$ be the largest and second-largest of the factors $> x^{1/k^{100}}$ in the factorizations described above. The last situation to consider is the one where $d_\ell = f_\ell q_\ell$ for every $\ell \in\{0,1,\ldots,k-1\}$, and $q_{j_0} \leq q_{i_0} x^{-k^{-4}}$. Thus, the largest prime factor $q_{i_0}$ of $d$ is a bit larger than all the other prime factors of $d$, and $q_{i_0}$ is also much larger than the factor $f_0f_1\cdots f_{k-1}$. We return to exponential sums to handle the error terms, and find now that, with these convenient facts at our disposal, the resulting short incomplete Kloosterman sums may be treated by combinatorial decompositions and a variant of the $q$-van der Corput process.

\begin{proposition}\label{prop:bound for E5 that relies on q van der corput}
Factor each $d_j = f_j q_j$ as above, with $f_j \leq x^{3/k^{100}}$ and $q_j=1$ or a prime $> x^{1/k^{100}}$. Let $\widetilde{f} = f_0 \cdots f_{k-1}$. Assume $q_{i_0}$ is the largest of the $q_j$. Let $q_*$ denote the second-largest of the $q_j$. Define
\begin{align*}
\mathcal{E}_5 &:= \sum_{\substack{n \sim x \\ n \equiv \gamma (M) \\ P^-(\binom{n}{k})\geq z \\ \exists d \in (n/B,n] \textup{ such that} \\ d \mid \binom{n}{k} \textup{ and } \\ d = \widetilde{f}q_0\cdots q_{k-1} \textup{ where} \\ q_j \leq q_{i_0} x^{-k^{-4}} \, \forall j \neq i_0 \textup{ and} \\ q_* > x^{1/(2k)}}} 1.
\end{align*}
We have $\mathcal{E}_5 = o (\mathcal{S})$.
\end{proposition}

With all the key propositions above at our disposal, it is now an easy matter to prove Theorem \ref{thm:main theorem}.

\begin{proof}[Proof of Theorem \ref{thm:main theorem} assuming Propositions \ref{prop:main term lower bound for S}, \ref{prop:one very large di bigger than x^1/2}, \ref{prop:some di bigger than x^1/3}, \ref{prop:some di has convenient factorization}, \ref{prop:exist 2 big primes dividing (n choose k) that are close}, and \ref{prop:bound for E5 that relies on q van der corput}]
We wish to find many $n \sim x$ such that $\binom{n}{k}$ has no divisors in $(n/B,n]$. By inclusion-exclusion, we have
\begin{align*}
\sum_{\substack{n \sim x \\ n \equiv \gamma (M) \\ P^-(\binom{n}{k})\geq z \\ \textup{there is no } d \in (n/B,n] \text{ s.t.} \\ d \mid \binom{n}{k}}} 1 = \sum_{\substack{n \sim x \\ n \equiv \gamma (M) \\ P^-(\binom{n}{k})\geq z}} 1 - \sum_{\substack{n \sim x \\ n \equiv \gamma (M) \\ P^-(\binom{n}{k})\geq z \\ \exists d \in (n/B,n] \text{ s.t.} \\ d \mid \binom{n}{k}}} 1 = \mathcal{S} - \mathcal{E},
\end{align*}
by \eqref{eq:defn of mathcal S} and \eqref{eq:defn of mathcal E}. Throughout we recall the factorization $d = d_0 d_1 \cdots d_{k-1}$ described earlier.

We first consider those $n$'s counted in $\mathcal{E}$ where $\binom{n}{k}$ has a divisor $d$ in $(n/B,n]$ with some $d_i > x^{101/200}$. All such $n$'s are counted by $\mathcal{E}_1$, and the contribution of these $d$ is acceptably small by Proposition \ref{prop:one very large di bigger than x^1/2}. We may therefore assume that for every $n$ left to count in $\mathcal{E}$ with a divisor $d \mid \binom{n}{k}$ in $(n/B,n]$, we have $d_i \leq x^{101/200}$ for every $i$.

Next, we consider those $n$'s counted by $\mathcal{E}$ where $d \mid \binom{n}{k}$, $d \in (n/B,n]$, and there is some $d_i$ with $x^{101/300} < d_i \leq x^{101/200}$. This contribution is acceptably small by Proposition \ref{prop:some di bigger than x^1/3}. We may therefore assume that for every $d$ we have that $d_i \leq x^{101/300}$ for every $i$.

Next, we consider those $n$'s counted by $\mathcal{E}$ where $d \mid \binom{n}{k}$ has a $d_i$ with a convenient factorization. In particular, we focus on those $n$'s with a $d\mid\binom{n}{k}$ where there exists a $d_i = rs$ where $r$ and $s$ are $> x^{1/k^{100}}$. This contribution is acceptably small by Proposition \ref{prop:some di has convenient factorization}. Therefore, we may assume that every $d_i \leq x^{101/300}$, and that whenever we factor a $d_i$ into positive integers $d_i = rs$ we either have $r \leq x^{1/k^{100}}$ or $s \leq x^{1/k^{100}}$.

Observe that any $d_i$ can have at most one prime factor that is $> x^{1/k^{100}}$, otherwise $d_i$ would have a factorization $d_i=rs$ where each of $r$ and $s$ is $> x^{1/k^{100}}$. Write $f_i$ for the product of the primes $\leq x^{1/k^{100}}$ that divide $d_i$. We claim that we must have $f_i \leq x^{3/k^{100}}$, otherwise $d_i$ has a convenient factorization. Indeed, suppose $f_i > x^{3/k^{100}}$. Factor $f_i = p_1 p_2 \cdots p_J$ with $x^{1/k^{100}}\geq p_1 \geq \cdots \geq p_J$, and let $t$ be the maximal positive integer such that $p_1 \cdots p_t \leq x^{1/k^{100}}$. Then
\begin{align*}
x^{1/k^{100}} < p_1 \cdots p_t p_{t+1} \leq x^{1/k^{100}} \cdot p_{t+1} \leq x^{2/k^{100}}.
\end{align*}
Since
\begin{align*}
\frac{d_i}{p_1 \cdots p_{t+1}}\geq \frac{f_i}{p_1\cdots p_{t+1}} > \frac{x^{3/k^{100}}}{x^{2/k^{100}}} = x^{1/k^{100}},
\end{align*}
we see that
\begin{align*}
d_i = (p_1 \cdots p_{t+1}) \cdot \frac{d_i}{p_1 \cdots p_{t+1}}
\end{align*}
is a convenient factorization of $d_i$.

Therefore, we may factor each $d_i = f_i q_i$, where $f_i \leq x^{3/k^{100}}$ and $q_i$ is either $1$ or a prime $> x^{1/k^{100}}$. Let $q_{i_0}$ be the largest of these factors, and let $q_{j_0}$ be the second-largest. Of course we have $q_{i_0} \leq x^{101/300}$. Since $n \sim x$ and $d \in (n/B,n]$, we see that $d > \frac{x}{2B}$, and therefore
\begin{align*}
\frac{x}{2B} < d = f_0 \cdots f_{k-1} q_0 \cdots q_{k-1} \leq x^{3/k^{99}} q_{i_0} q_{j_0}^{k-1} \leq x^{101/300 + 3/k^{99}} q_{j_0}^{k-1}.
\end{align*}
We see that $q_{j_0}^{k-1} > x^{33/50}$, and
\begin{align*}
q_{j_0} > x^{1/(2k)},
\end{align*}
say. The contribution to $\mathcal{E}$ from those $n$'s with a $d$ where $q_{j_0} > q_{i_0} x^{-1/k^4}$ is sufficiently small by Proposition \ref{prop:exist 2 big primes dividing (n choose k) that are close}, since $k$ is large.

Lastly, we are left to consider the contribution to $\mathcal{E}$ of those $n$'s where $q_j \leq q_{i_0} x^{-1/k^4}$ for every $j \neq i_0$. The contribution from these $d$ to $\mathcal{E}$ is sufficiently small by Proposition \ref{prop:bound for E5 that relies on q van der corput}.

It follows that
\begin{align*}
\sum_{\substack{n \sim x \\ n \equiv \gamma (M) \\ P^-(\binom{n}{k})\geq z \\ \textup{there is no } d \in (n/B,n] \text{ s.t.} \\ d \mid \binom{n}{k}}} 1 = (1-O(k^{-1})-o_{x \rightarrow \infty}(1)) \mathcal{S},
\end{align*}
and the result follows by Proposition \ref{prop:main term lower bound for S}.
\end{proof}

\section{Sieve methods, and the proofs of some key propositions}\label{sec:sieve}

The proofs of all the key propositions rely on sieve methods. We give a brief overview here of the definitions and facts we need. We refer the reader to \cite{opera} for a fuller treatment of sieve methods.

Given an integer $n>1$, we let $P^-(n)$ denote the least prime divisor of $n$. We establish the convention that $P^-(1) = \infty$. For $z>1$ real, we define
\begin{align*}
P(z) \coloneqq \prod_{p < z} p.
\end{align*}
The sieve is a device for gaining control on conditions of the form $\mathbf{1}(\gcd(n,P(z))=1)$. By M\"obius inversion, we may write this condition as a divisor sum
\begin{align*}
\sum_{\substack{g \mid n \\ g\mid P(z)}} \mu(g),
\end{align*}
but there are typically too many terms in this sum to enable effective estimation. The sieve (particularly the \emph{combinatorial sieve}) truncates the M\"obius function in order to obtain a more tractable divisor sum. Given a parameter $G\geq 1$, there are sets $\mathcal{G}^\pm$ supported on integers $\leq G$ such that
\begin{align*}
\sum_{\substack{g \mid n \\ g \mid P(z) \\ g \in \mathcal{G}^-}} \mu(g) \leq \mathbf{1}(\gcd(n,P(z))=1) \leq \sum_{\substack{g \mid n \\ g \mid P(z) \\ g \in \mathcal{G}^+}} \mu(g).
\end{align*}
We have $1 \in \mathcal{G}^\pm$. We will usually abbreviate the notation by writing
\begin{align*}
\sum_{g \leq G} \lambda_g^{\pm}
\end{align*}
for the divisor sums over sieve weights. Note that the sieve weights $\lambda_g^\pm$ are supported on squarefree integers.

If the parameter $G$ is sufficiently large with respect to $z$, then the sieve is accurate and captures all the essential features of the condition $\mathbf{1}(\gcd(n,P(z))=1)$. This is made precise through the introduction of the \emph{sifting variable}
\begin{align*}
s = \frac{\log G}{\log z}.
\end{align*}
If $s$ is sufficiently large, then the sieve is of ``fundamental lemma-type.'' This is the only regime we shall require. The sieve weights are usually coupled together with a multiplicative function $h(g)$, so that one wishes to obtain good estimates for the main terms
\begin{align*}
\sum_{g \leq G} \lambda_g^\pm h(g).
\end{align*}
We assume that this multiplicative function $h$ satisfies the one-sided inequality
\begin{align*}
\prod_{w \leq p < y} (1-h(p))^{-1} \leq K \left(\frac{\log y}{\log w} \right)^\kappa
\end{align*}
for all $y > w \geq 2$, for some constant $K>1$ (see \cite[(5.38)]{opera}). Here $\kappa>0$ is known as the \emph{sieving dimension}. In practice, we think of this condition as saying that $h(p) \approx 
\frac{\kappa}{p}$ on average. Given $h(g)$ satisfying the above, one can choose sieve weights $\lambda_g^\pm$ such that
\begin{align*}
\sum_{g \leq G} \lambda_g^\pm h(g) = \left\{1 + O \left(K^{10}\exp(9\kappa - s) \right) \right\} \cdot \prod_{p < z} (1-h(p)).
\end{align*}

We need the following result concerning sieve weights. We shall use this lemma repeatedly in proving our key propositions.

\begin{lemma}\label{lem:main sieve result}
Let $k$ be a sufficiently large positive integer. Let $z > 2k$ and $G$ be given, with $s = \frac{\log G}{\log z}$ sufficiently large. There are sieve weights $\lambda_g^\pm$ (that depend on $G,z,k$) such that

\begin{align*}
\sum_{\substack{g \leq G \\ g \mid P(z) \\ P^-(g) \geq 2k}} \lambda_g^\pm \frac{k^{\omega(g)}}{g} = \left(1+O \left(\exp(10k-s) + \frac{k^2}{z\log z} + \frac{k}{\log z}\right)\right)\cdot \mathfrak{S}_k (\log z)^{-k},
\end{align*}
where $\mathfrak{S}_k$ is defined in \eqref{eq:defn of singular series}. As $k \rightarrow \infty$, we have $\mathfrak{S}_k = \exp((1+o(1))k \log \log k)$.
\end{lemma}
\begin{proof}
This follows from the fundamental lemma of sieve theory. In particular, we apply \cite[Lemma 6.8]{opera}, thereby taking the sieve weights to be the combinatorial beta-sieve weights of some dimension depending on $k$. However, in order to apply the result, we need to verify that condition (5.38) in \cite{opera} holds with a suitable multiplicative function.

Let $h(g)$ be the multiplicative function given by
\begin{align*}
\mathbf{1}(P^-(g) \geq 2k) \cdot \mathbf{1}(P^+(g) < z) \frac{k^{\omega(g)}}{g}.
\end{align*}
We must show there exists some $K = K(k)$ such that
\begin{align}\label{eq:5.38 Opera de Cribro}
\prod_{w \leq p < y} (1-h(p))^{-1} \leq K \left(\frac{\log y}{\log w} \right)^k
\end{align}
for all $y > w \geq 2$. It is for the purpose of establishing this condition that we forced the condition that $\binom{n}{k}$ has no prime factors $< 2k$. We consider several cases, depending on the sizes of $y$ and $w$.

If $y < 2k$, then \eqref{eq:5.38 Opera de Cribro} holds trivially with $K = 1$. Assume $w < 2k$ and $y \geq 2k$. Then
\begin{align*}
\prod_{w \leq p < y} (1-h(p))^{-1} \leq \prod_{2k \leq p < y} \left(1 - \frac{k}{p}\right)^{-1} = \prod_{2k \leq p < y} \left( 1 - p^{-1} \right)^{-k} \cdot \prod_{2k \leq p < y} \frac{\left(1 - \frac{1}{p}\right)^k}{1 - \frac{k}{p}}.
\end{align*}
By Mertens's theorem and some easy estimations, this is
\begin{align*}
\exp \left(O \left(\frac{k}{\log k} \right) \right)\cdot  \left(\frac{\log y}{\log 2k} \right)^k \leq \exp \left(O \left(\frac{k}{\log k} \right) \right)\cdot  \left(\frac{\log y}{\log w} \right)^k.
\end{align*}
Therefore, \eqref{eq:5.38 Opera de Cribro} holds with $K = \exp(C \frac{k}{\log k})$, where $C>0$ is some sufficiently large constant.

Lastly, we consider the case where $2k \leq w$. We find that \eqref{eq:5.38 Opera de Cribro} holds by arguments identical to those in the previous case.

By \cite[Lemma 6.8]{opera}, we therefore have
\begin{align*}
\sum_{\substack{g \leq G \\ g \mid P(z) \\ P^-(g) \geq 2k}} \lambda_g^\pm \frac{k^{\omega(g)}}{g} = \left(1 + O \left(\exp \left(10k-s \right) \right) \right) \cdot \prod_{2k \leq p < z} \left(1 - \frac{k}{p} \right)
\end{align*}
for $k$ sufficiently large, where we recall that $s = \frac{\log G}{\log z}$. We rewrite the product over $p$ as
\begin{align*}
\prod_{p < z} \left(1 - \frac{1}{p}\right)^k \cdot \prod_{p < 2k} \left( 1 - \frac{1}{p}\right)^{-k} \prod_{2k \leq p < z} \frac{1 - \frac{k}{p}}{\left(1 - \frac{1}{p}\right)^k}.
\end{align*}
We can complete the product over $2k \leq p < z$ at the cost of an error of size
\begin{align*}
\exp \left(O\left(\frac{k^2}{z\log z}\right)\right) = 1 + O \left(\frac{k^2}{z\log z} \right).
\end{align*} 
By Mertens's theorem, we have
\begin{align*}
\prod_{p < z} \left(1 - \frac{1}{p}\right)^k = \exp \left(O \left(\frac{k}{\log z} \right) \right)e^{-c k} (\log z)^{-k} = \left(1 + O \left(\frac{k}{\log z} \right) \right) \cdot e^{-c k} (\log z)^{-k}.
\end{align*}
We therefore have
\begin{align*}
\sum_{\substack{g \leq G \\ g \mid P(z) \\ P^-(g) \geq 2k}} \lambda_g^\pm \frac{k^{\omega(g)}}{g} = \left(1+O \left(\exp(10k-s) + \frac{k^2}{z\log z} + \frac{k}{\log z}\right)\right) \cdot \mathfrak{S}_k (\log z)^{-k},
\end{align*}
where $\mathfrak{S}_k$ is defined in \eqref{eq:defn of singular series}. By Mertens's theorem again and some estimations, we find that
\begin{align*}
\mathfrak{S}_k &= \exp \left(O \left(\frac{k}{\log k} \right) \right)\cdot (\log 2k)^k = \exp\left(k \log \log k + O \left(\frac{k}{\log k} \right)\right). \qedhere
\end{align*}
\end{proof}

With the above facts and results set in place, we can prove some of the key propositions stated in Section \ref{sec:key props for divisor problem}.

\subsection{Proof of Proposition \ref{prop:main term lower bound for S}}

\begin{proof}[Proof of Proposition \ref{prop:main term lower bound for S}]
Recall that
\begin{align*}
\mathcal{S} = \sum_{\substack{n \sim x \\ n \equiv \gamma (M) \\ P^-(\binom{n}{k})\geq z}} 1.
\end{align*}
Among other things, the congruence $n \equiv \gamma \pmod{M}$ implies that $\binom{n}{k}$ is not divisible by any primes $< 2k$. Using the sieve weights of Lemma \ref{lem:main sieve result}, and taking $G = x^{1/2}$, say, we have
\begin{align*}
\mathcal{S} &\geq \sum_{\substack{n \sim x \\ n \equiv \gamma (M)}} \sum_{\substack{g \leq G \\ g \mid P(z) \\ P^-(g) \geq 2k \\ g \mid \binom{n}{k}}} \lambda_g^- = \sum_{\substack{g \leq G \\ g \mid P(z) \\ P^-(g) \geq 2k}} \lambda_g^- \sum_{\substack{n \sim x \\ n \equiv \gamma (M) \\ g \mid \binom{n}{k}}} 1.
\end{align*}

Since $\lambda_g^-$ is supported on squarefree integers $g$, we may write $g = p_1 \cdots p_r$, where the $p_i$ are distinct primes with $2k \leq p_i < z$. We have that $g \mid \binom{n}{k}$ if and only if $p_i \mid \prod_{j=0}^{k-1}(n-j)$ for every $i \in \{1,2,\ldots,r\}$. Since $p_i \geq 2k$, each $p_i$ can divide at most a single term $n-j$. Each of the $r$ primes $p_i$ can independently divide any of the $k$ terms $n-j$, so there are $k^r = k^{\omega(g)}$ different ways for $d$ to divide $\binom{n}{k}$. By the Chinese remainder theorem, we have that $g \mid \binom{n}{k}$ if and only if $n \equiv v \pmod{g}$ for some residue class $v \pmod{g}$ in a set $S(g)$ of size $k^{\omega(g)}$. Therefore
\begin{align*}
\mathcal{S} &\geq \sum_{\substack{g \leq G \\ g \mid P(z) \\ P^-(g) \geq 2k}} \lambda_g^- \sum_{v \in S(g)} \sum_{\substack{n \sim x \\ n \equiv \gamma (M) \\ n \equiv v (g)}} 1.
\end{align*}
As $M$ is only divisible by primes $< 2k$, and $g$ is only divisible by primes $\geq 2k$, the integers $M$ and $g$ are coprime. Therefore, we may combine the congruence conditions on $n$ into a single congruence condition modulo $gM$ by the Chinese remainder theorem. We execute the sum over $n$ trivially to obtain
\begin{align*}
\mathcal{S} &\geq \frac{x}{2M}\sum_{\substack{g \leq G \\ g \mid P(z) \\ P^-(g) \geq 2k}} \lambda_g^- \frac{k^{\omega(g)}}{g} + O \left(\sum_{\substack{g \leq G}} \mu^2(g) k^{\omega(g)} \right).
\end{align*}
We evaluate the main term using Lemma \ref{lem:main sieve result}, noting that $s = \frac{\log G}{\log z} = (2\epsilon_k)^{-1}$ (recall \eqref{eq:defn of epsilon_k}). To evaluate the error term, observe that
\begin{align*}
\sum_{\substack{g \leq G}} \mu^2(g) k^{\omega(g)} &\leq G \sum_{\substack{g \leq G}} \mu^2(g) \frac{k^{\omega(g)}}{g} \leq G \prod_{p \leq G} \left(1 + \frac{k}{p} \right) \leq G \prod_{p \leq G} \left(1 + \frac{1}{p}\right)^k \\
&\leq G \prod_{p \leq G} \left(1- \frac{1}{p} \right)^{-k} \ll x^{1/2} (\log x)^k,
\end{align*}
where in the penultimate step we have used Mertens's theorem. The upper bound \eqref{eq:upp bound on k in terms of x} on $k$ implies the error term is $O(x^{5/6})$, say. This is of smaller order than the main term.
\end{proof}

\subsection{Proof of Proposition \ref{prop:one very large di bigger than x^1/2}}

We next prove Proposition \ref{prop:one very large di bigger than x^1/2}.

\begin{proof}[Proof of Proposition \ref{prop:one very large di bigger than x^1/2}]
Recall that
\begin{align*}
\mathcal{E}_1 &= \sum_{\substack{n \sim x \\ n \equiv \gamma (M) \\ P^-(\binom{n}{k})\geq z \\ \exists d \in (n/B,n] \textup{ such that} \\ d \mid \binom{n}{k} \textup{ and } \\ \exists d_i > x^{101/200}}} 1.
\end{align*}
We shall write $d_i$ for the portion of $d$ that is larger than $x^{101/200}$, and write $D$ for the product of the $d_j$ with $j \neq i$. If $d = Dd_i \in (n/B,n]$, then $D \leq \frac{x}{d_i} \leq x^{99/200}$. By the union bound, we have
\begin{align*}
\mathcal{E}_1 &\leq \sum_{i=0}^{k-1}  \mathop{\sum\cdots\sum}_{\substack{d_0,\ldots,d_{i-1},d_{i+1},\ldots,d_{k-1} \\ z \leq D \leq x^{99/200} \\ P^-(D) \geq z}} \sum_{\substack{d_i > x^{101/200} \\ P^-(d_i) \geq z \\ \frac{x}{2B} < Dd_i \leq x \\ (d_j,d_{j'})=1 \text{ for } j \neq j'}} \,  \sum_{\substack{n \sim x \\ n \equiv \gamma (M) \\ P^-(\binom{n}{k}) \geq z \\ n \equiv i (d_i) \\ n \equiv \gamma_D(D)}} 1.
\end{align*}
We may write $n-i = d_i b_i g_i$, say, where $g_i$ is only divisible by primes $\leq k$, and both $d_i$ and $b_i$ have all their prime divisors $\geq z$. We have $g_i \mid k!$ by the particular properties of the residue class $\gamma \pmod{M}$. Moreover, we have that the integer $g_i$ depends only on $i$. 

We wish to switch from $d_i$ to the complementary divisor $b_i$. We observe that $b_i > 1$ since
\begin{align*}
x \ll n-i =d_ib_i g_i \leq d_ib_i (k!) \ll \frac{x}{D} b_i (k!),
\end{align*}
so $b_i \gg \frac{D}{k!} \gg \frac{z}{k!} > 1$. Since $n \sim x$, we have $\frac{x}{3} < n-i \leq x$, say, and therefore
\begin{align*}
\frac{x}{3d_ig_i} < b_i \leq \frac{x}{d_ig_i}.
\end{align*}
We split the size of $d = Dd_i$ into dyadic ranges $Dd_i \sim y$, with $\frac{x}{B} \ll y \ll x$. We also split $D$ into dyadic ranges $D \sim E$, where $z \ll E \ll x^{99/200}$. It follows that
\begin{align*}
\frac{xE}{6g_i y} < b_i \leq \frac{2xE}{g_i y}.
\end{align*}
We switch the congruence condition $n \equiv i \pmod{d_i}$ for the congruence condition $n \equiv i \pmod{b_i}$ to get
\begin{align*}
\mathcal{E}_1 &\leq \sum_{x/B \ll y \ll x} \sum_{z \ll E \ll x^{99/200}} \sum_{i=0}^{k-1}\mathop{\sum\cdots\sum}_{\substack{d_0,\ldots,d_{i-1},d_{i+1},\ldots,d_{k-1} \\ D \sim E \\ P^-(D) \geq z \\ (d_j,d_{j'})=1 \text{ for } j \neq j'}} \sum_{\substack{\frac{xE}{6g_iy} < b_i \leq \frac{2xE}{g_i y} \\ P^-(b_i) \geq z \\ (b_i,D)=1}} \,  \sum_{\substack{n \sim x \\ n \equiv \gamma (M) \\ P^-(\binom{n}{k}) \geq z \\ n \equiv i (b_i) \\ n \equiv \gamma_D(D)}} 1.
\end{align*}
The effect of these manipulations is that the congruence condition on $n$ modulo $Dd_i$ has been replaced by a congruence condition modulo $Db_i$, and $Db_i$ is a bit less than $x$. This gives us room to insert sieve weights and still keep the error terms under control.

We insert upper-bound sieve weights $\lambda_g^+$ to control the condition $P^-(\binom{n}{k}) \geq z$. We take the level of the sieve to be $G=x^{100k\epsilon_k}$, where $\epsilon_k$ is defined in \eqref{eq:defn of epsilon_k}. We swap the order of summation, and handle the condition $g \mid \binom{n}{k}$ as in the proof of Proposition \ref{prop:main term lower bound for S}, thereby obtaining
\begin{align*}
\mathcal{E}_1 &\leq \sum_{x/B \ll y \ll x} \sum_{z \ll E \ll x^{99/200}} \sum_{i=0}^{k-1} \sum_{\substack{g \leq G \\ g \mid P(z) \\ P^-(g) \geq 2k}}\lambda_g^+ \sum_{v \in S(g)} \mathop{\sum\cdots\sum}_{\substack{d_0,\ldots,d_{i-1},d_{i+1},\ldots,d_{k-1} \\ D \sim E \\ P^-(D) \geq z \\ (d_j,d_{j'})=1 \text{ for } j \neq j'}} \sum_{\substack{\frac{xE}{6g_iy} < b_i \leq \frac{2xE}{g_i y} \\ P^-(b_i) \geq z \\ (b_i,D)=1}} \,  \sum_{\substack{n \sim x \\ n \equiv \gamma (M) \\ n \equiv v (g) \\ n \equiv i (b_i) \\ n \equiv \gamma_D(D)}} 1.
\end{align*}
The integers $M,g,b_i,D$ are pairwise coprime, and therefore we may combine the separate congruence conditions on $n$ into a single congruence condition modulo $gMDb_i$. We then sum over $n$ trivially to obtain
\begin{align*}
\mathcal{E}_1 &\leq \frac{x}{2M} \Bigg(\sum_{\substack{g \leq G \\ g \mid P(z) \\ P^-(g) \geq 2k}} \lambda_g^+ \frac{k^{\omega(g)}}{g} \Bigg)\Bigg(\sum_{x/B \ll y \ll x} \sum_{z \ll E \ll x^{99/200}} \sum_{i=0}^{k-1} \mathop{\sum\cdots\sum}_{\substack{d_0,d_1,\ldots,d_{i-1},d_{i+1},\ldots,d_{k-1} \\ D \sim E \\ P^-(D) \geq z \\ (d_j,d_{j'})=1 \text{ for } j \neq j'}} \frac{1}{D} \sum_{\substack{\frac{xE}{6g_iy} < b_i \leq \frac{2xE}{g_i y} \\ P^-(b_i) \geq z \\ (b_i,D)=1}} \frac{1}{b_i} \Bigg) \\
&+ O \left(\sum_{x/B \ll y \ll x} \sum_{z \ll E \ll x^{99/200}}\sum_{\substack{D \sim E \\ P^-(D) \geq z}}\mu^2(D) \tau_{k-1}(D)  \sum_{\substack{g \leq G}}\mu^2(g) k^{\omega(g)} \sum_{i=0}^{k-1}\sum_{\substack{\frac{xE}{6g_iy} < b_i \leq \frac{2xE}{g_i y}}} 1 \right).
\end{align*}
By trivial estimation, the error term has size
\begin{align*}
\ll k(\log B)\sum_{z \ll E \ll x^{99/200}} BGE^2 (10\epsilon_k)^{-k} (\log G)^{k} \ll x^{999/1000},
\end{align*}
say, since $E \ll x^{99/200}$, and $B$ and $k$ are small compared to $x$.

Now we turn to bounding the main term. We use Lemma \ref{lem:main sieve result} to evaluate the sum over $g$. Since $\frac{xE}{g_i y} \gg z^{1/2}$, say, we can use an easy upper-bound sieve argument to show
\begin{align*}
\sum_{\substack{\frac{xE}{6g_iy} < b_i \leq \frac{2xE}{g_i y} \\ P^-(b_i) \geq z \\ (b_i,D)=1}} \frac{1}{b_i} \ll \frac{1}{\log z}.
\end{align*}
Given $D$, we sum over the dyadic ranges $E$ to de-localize the scale of $D$, and thereby obtain
\begin{align*}
\mathcal{E}_1 &\ll \frac{x}{M} \mathfrak{S}_k (\log z)^{-k} \cdot \frac{k \log B}{\log z} \mathop{\sum\cdots \sum}_{\substack{f_1,\ldots,f_{k-1} \\ P^-(f_i) \geq z \\ z \ll f_1 \cdots f_{k-1} \ll x^{99/200}}} \frac{1}{f_1\cdots f_{k-1}} \\
&\leq \frac{x}{M} \mathfrak{S}_k (\log z)^{-k} \cdot \frac{k \log B}{\log z} \left(\sum_{\substack{f \ll x \\ P^-(f) \geq z}} \frac{1}{f} \right)^{k-1} \ll \frac{x}{M} \mathfrak{S}_k (\log z)^{-k} \cdot \frac{k \log B}{\log z} \epsilon_k^{-(k-1)}.
\end{align*}
By Proposition \ref{prop:main term lower bound for S}, we have
\begin{align*}
\mathcal{S} \gg \mathfrak{S}_k \frac{x}{M} (\log z)^{-k}.
\end{align*}
Since $B \leq k$ and $k$ satisfies \eqref{eq:upp bound on k in terms of x}, we see
\begin{align*}
\frac{k \log B}{\log z} \epsilon_k^{-(k-1)} = \frac{k \log B}{\log x} \epsilon_k^{-k} = o(1),
\end{align*}
and we have the desired result.
\end{proof}

\subsection{Proof of Proposition \ref{prop:exist 2 big primes dividing (n choose k) that are close}}

In this subsection, we prove the last of the ``elementary'' key propositions.

\begin{proof}[Proof of Proposition \ref{prop:exist 2 big primes dividing (n choose k) that are close}]
Recall that
\begin{align*}
\mathcal{E}_4 &= \sum_{\substack{n \sim x \\ n \equiv \gamma (M) \\ P^-(\binom{n}{k})\geq z \\ \exists \textup{ primes } q,r \mid \binom{n}{k} \textup{ with } \\ x^{(2k)^{-1}} < r < q \leq x^{101/300} \\ r > q x^{-k^{-4}}}} 1,
\end{align*}
and we want to show that $\mathcal{E}_4 = o(\mathcal{S})$. By a union bound, we have
\begin{align*}
\mathcal{E}_4 &\leq \mathop{\sum\sum}_{0 \leq i,j \leq k-1} \mathop{\sum\sum}_{\substack{x^{(2k)^{-1}} < r < q \leq x^{101/300} \\ r > q x^{-k^{-4}}}} \sum_{\substack{n \sim x \\ n \equiv \gamma (M) \\ P^-(\binom{n}{k})\geq z \\ n \equiv i (q) \\ n \equiv j (r)}} 1.
\end{align*}
Since $q,r \leq x^{101/300} \leq x^{1/2-\delta}$, there is enough room to insert sieve weights and still control all the error terms. Handling the sieve weights as in the proof of Proposition \ref{prop:main term lower bound for S}, we have
\begin{align*}
\mathcal{E}_4 &\leq \mathop{\sum\sum}_{0 \leq i,j \leq k-1} \mathop{\sum\sum}_{\substack{x^{(2k)^{-1}} < r < q \leq x^{101/300} \\ r > q x^{-k^{-4}}}} \sum_{\substack{g \leq G \\ g \mid P(z) \\ P^-(g) \geq 2k}}\lambda_g^+\sum_{v \in S(g)}\sum_{\substack{n \sim x \\ n \equiv \gamma (M) \\ n \equiv v (g) \\ n \equiv i (q) \\ n \equiv j (r)}} 1,
\end{align*}
where we set $G = x^{100k\epsilon_k}$. Since $M,g,q,r$ are pairwise coprime, we may combine all the congruences on $n$ together using the Chinese remainder theorem. Summing trivially in $n$, we obtain
\begin{align*}
\mathcal{E}_4 &\leq \frac{x}{2M}\left(\sum_{\substack{g \leq G \\ g \mid P(z) \\ P^-(g) \geq 2k}} \lambda_g^+ \frac{k^{\omega(g)}}{g} \right) \left(\mathop{\sum\sum}_{0 \leq i,j \leq k-1} \mathop{\sum\sum}_{\substack{x^{(2k)^{-1}} < r < q \leq x^{101/300} \\ r > q x^{-k^{-4}}}} \frac{1}{qr}\right) \\
&+ O \left(k^2 \mathop{\sum\sum}_{q,r \leq x^{101/300}}\sum_{\substack{g \leq G \\ P^-(g) \geq 2k}} \mu^2(g) k^{\omega(g)} \right).
\end{align*}
The error term has size
\begin{align*}
\ll 2^k x^{101/150} G (\log x)^k \ll x^{99/100},
\end{align*}
say. In the main term, we use Lemma \ref{lem:main sieve result} to evaluate the sum over $g$. The double sum over $q$ and $r$ is
\begin{align*}
\ll \sum_{x^{(2k)^{-1}} < q \leq x^{101/300}} \frac{1}{q}\sum_{q x^{-k^{-4}} < r < q} \frac{1}{r}.
\end{align*}
By Mertens's theorem, the sum over $r$ is
\begin{align*}
\log \log q - \log \log (qx^{-k^{-4}}) + O \left(\frac{k}{\log x} \right) \ll \frac{\log x}{k^4 \log q} + \frac{k}{\log x}.
\end{align*}
Thus, the double sum over $q$ and $r$ is
\begin{align*}
\ll \sum_{x^{(2k)^{-1}} < q \leq x^{101/300}} \frac{1}{q} \left(\frac{\log x}{k^4 \log q} + \frac{k}{\log x} \right) \ll \frac{1}{k^3} + \frac{k \log k}{\log x}.
\end{align*}
We deduce that
\begin{align*}
\mathcal{E}_4 &\ll \mathfrak{S}_k \frac{x}{M} (\log z)^{-k} \left(\frac{1}{k} + \frac{k^3 \log k}{\log x} \right),
\end{align*}
and we obtain the desired result by comparison with Proposition \ref{prop:main term lower bound for S}.
\end{proof}

\section{Some tools for exponential sum estimates}\label{sec:exp sums}

The proofs of the remaining key propositions all rely on exponential sum estimates. We introduce exponential sums through the Fourier transform and Poisson summation.

\begin{lemma}[Poisson summation]\label{lem:poisson summation}
Let $\psi_0$ be the smooth function described above, and let $M,q \leq x^{100}$ for some sufficiently large real number $x$. Let $H$ be any positive real number $\geq \frac{q}{M}(\log x)^3$. For any integer $a$, we have
\begin{align*}
\sum_{m \equiv a (q)} \psi_0 \left( \frac{m}{M} \right) = \frac{M}{q}\widehat{\psi_0}(0) + \frac{M}{q}\sum_{1 \leq |h| \leq H} \widehat{\psi_0} \left(\frac{hM}{q} \right)e \left(\frac{ah}{q} \right) + O(x^{-1000}).
\end{align*}
\end{lemma}
\begin{proof}
This is a small variant of \cite[Lemma 13.4]{Maynard2025}.
\end{proof}

We shall need the following bound for incomplete Kloosterman sums.

\begin{lemma}[Weil bound]\label{lem:bd on incomplete kloosterman sum}
Let $I$ be an interval of length $|I|\geq 1$. Let $a,c$ be integers with $c\geq 1$. Then
\begin{align*}
\Big|\sum_{\substack{x \in I \\ (x,c)=1}} e \left(\frac{a \overline{x}}{c} \right) \Big| \ll \tau(c) \gcd(a,c)^{1/2} c^{1/2} \log(2c) + \frac{|I|}{c^{1/2}} \tau(c) \gcd(a,c)^{1/2}.
\end{align*}
\end{lemma}
\begin{proof}
This follows from \cite[Lemma 8]{DFI1997}.
\end{proof}

We shall frequently encounter exponential phases that arise from the Chinese remainder theorem in some way. The following reciprocity lemma shall prove useful.

\begin{lemma}[B\'ezout's identity]\label{lem:bezout}
Let $q_1,q_2$ be coprime positive integers. For any integer $a$, we have
\begin{align*}
\frac{a}{q_1q_2} = \frac{a \overline{q_1}}{q_2} + \frac{a\overline{q_2}}{q_1} \pmod{1}.
\end{align*}
\end{lemma}
\begin{proof}
This is \cite[Lemma 13.1]{Maynard2025}.
\end{proof}

We use the following identity in order to decompose exponential sums with a ``rough'' variable into linear and bilinear sums.

\begin{lemma}[Buchstab's identity]\label{lem:buchstab}
Let $z > 1$ be a real number. Define a function $\rho(n,z)$ by
\begin{align*}
\rho(n,z) = 
\begin{cases}
1, &\textup{if } n \in \mathbb{N} \textup{ and } P^-(n) \geq z, \\
0, &\textup{otherwise}.
\end{cases}
\end{align*}
For any real number $w$ with $1 \leq w < z$, we have
\begin{align*}
\rho(n,z) = \rho(n,w) - \sum_{w \leq p < z} \rho\left(\frac{n}{p},p \right),
\end{align*}
where the sum is a sum over primes $p$.
\end{lemma}
\begin{proof}
This is \cite[(3.1.1)]{HarmanPDS}.
\end{proof}

\begin{lemma}[Combinatorial decomposition]\label{lem:buchstab comb decomp}
Let $x$ be large, and let $n \leq x$ be a positive integer. Let $1 < z \leq x$, and let $R\geq z$ be a parameter. Then
\begin{align*}
\rho(n,z) &= \sum_{\substack{n=rm \\ r \mid P(z) \\ r \leq R}} \mu(r) + \sum_{2 \leq j \ll \log x}(-1)^j \sum_{\substack{n=mp_1\cdots p_j \\ p_j < \cdots < p_1 < z \\ p_1 \cdots p_{j-1} \leq R \\ p_1 \cdots p_{j-1}p_j>R}} \rho(m,p_j).
\end{align*}
\end{lemma}
\begin{proof}
The main idea, as in the proof of \cite[Theorem 3.1]{HarmanPDS} is to repeatedly apply Lemma \ref{lem:buchstab} with $w=1$. As we repeatedly apply Buchstab, we obtain more prime variables, and we need to keep track of the sizes of the primes. After one application of Buchstab, we have
\begin{align*}
\rho(n,z) = 1 - \sum_{\substack{n=p_1 m \\ p_1 < z}} \rho(m,p_1).
\end{align*}
Note that $p_1 < z \leq R$. We apply Buchstab's identity again, and split sums based on the size of the variables:
\begin{align*}
\rho(n,z) = 1 - \sum_{\substack{n=p_1 m \\ p_1 < z}} 1 + \sum_{\substack{n=p_1p_2m \\ p_2 < p_1 < z \\ p_1 p_2 \leq R}} \rho(m,p_2) + \sum_{\substack{n=p_1p_2m \\ p_2 < p_1 < z \\ p_1 p_2 > R}} \rho(m,p_2).
\end{align*}
The term with $p_1p_2> R$ will give rise to suitable ``bilinear'' or ``Type II'' sums, but the term with $p_1 p_2 \leq R$ needs another application of Lemma \ref{lem:buchstab}. In general, we have
\begin{align*}
(-1)^j \sum_{\substack{n=mp_1\cdots p_j \\ p_j < \cdots < p_1 < z \\ p_1 \cdots p_j \leq R}}\rho(m,p_j) &= (-1)^j \sum_{\substack{n=mp_1\cdots p_j \\ p_j < \cdots < p_1 < z \\ p_1 \cdots p_j \leq R}} 1 + (-1)^{j+1}\sum_{\substack{n=mp_1\cdots p_jp_{j+1} \\ p_{j+1} < \cdots < p_1 < z \\ p_1 \cdots p_{j+1} \leq R}}\rho(m,p_{j+1}) \\
&+ (-1)^{j+1}\sum_{\substack{n=mp_1\cdots p_jp_{j+1} \\ p_{j+1} < \cdots < p_1 < z \\ p_1 \cdots p_j \leq R \\ p_1 \cdots p_j p_{j+1} > R}}\rho(m,p_{j+1}).
\end{align*}
Since an integer $n \leq x$ has $\omega(n) \ll \log x$, we see by expanding as much as possible that
\begin{align*}
\rho(n,z) &= \sum_{\substack{n=rm \\ r \mid P(z) \\ r \leq R}} \mu(r) + \sum_{2 \leq j \ll \log x}(-1)^j \sum_{\substack{n=mp_1\cdots p_j \\ p_j < \cdots < p_1 < z \\ p_1 \cdots p_{j-1} \leq R \\ p_1 \cdots p_{j-1}p_j>R}} \rho(m,p_j). \qedhere
\end{align*}
\end{proof}

With these tools in place, we are ready to prove Proposition \ref{prop:some di bigger than x^1/3}.

\begin{proof}[Proof of Proposition \ref{prop:some di bigger than x^1/3}]
Recall that
\begin{align*}
\mathcal{E}_2 &= \sum_{\substack{n \sim x \\ n \equiv \gamma (M) \\ P^-(\binom{n}{k})\geq z \\ \exists d \in (n/B,n] \textup{ such that} \\ d \mid \binom{n}{k} \textup{ and } \\ \exists d_i \textup{ such that } x^{101/300} < d_i \leq x^{101/200}}} 1.
\end{align*}
By a union bound, we have
\begin{align*}
\mathcal{E}_2 &\leq \sum_{i=0}^{k-1}\mathop{\sum\cdots\sum}_{\substack{d_0,\ldots,d_{i-1},d_{i+1},\ldots,d_{k-1} \\ x^{49/100} \leq D \leq x^{199/300} \\ P^-(D) \geq z}} \sum_{\substack{x^{101/300}< d_i \leq x^{101/200} \\ P^-(d_i) \geq z \\ \frac{x}{2B} < Dd_i \leq x \\ (d_j,d_{j'})=1 \text{ for } j \neq j'}} \,  \sum_{\substack{n \sim x \\ n \equiv \gamma (M) \\ P^-(\binom{n}{k}) \geq z \\ n \equiv i (d_i) \\ n \equiv \gamma_D(D)}} 1.
\end{align*}
We can no longer afford to treat error terms trivially when we sum over $n$. Instead, we show the error terms are sufficiently small by writing the error terms as exponential sums, and then bounding the resulting exponential sums.

We get an upper bound if we replace the condition $\mathbf{1}(n \sim x)$ by $\psi_0 (\frac{n}{x})$. We dyadically decompose the ranges of $Dd_i$ and $D$, then insert upper-bound sieve weights to control the condition on $\binom{n}{k}$. We therefore obtain
\begin{align*}
\mathcal{E}_2 &\leq \sum_{x/B \ll y \ll x}\sum_{x^{49/100} \ll E \ll x^{199/300}} \sum_{i=0}^{k-1}\sum_{\substack{g\leq G \\ g \mid P(z) \\ P^-(g) \geq 2k}}\lambda_g^+ \sum_{v \in S(g)} \mathop{\sum\cdots\sum}_{\substack{D\sim E \\ P^-(D) \geq z}} \sum_{\substack{\frac{y}{2E}< d_i \leq \frac{2y}{E} \\ P^-(d_i) \geq z \\ (d_j,d_{j'})=1 \text{ for } j \neq j'}} \,  \sum_{\substack{n \equiv \gamma (M) \\ n \equiv v (g) \\ n \equiv i (d_i) \\ n \equiv \gamma_D(D)}} \psi_0 \left( \frac{n}{x}\right),
\end{align*}
where we take $G = x^{100k\epsilon_k}$. The integers $M,g,d_i,D$ are pairwise coprime. We may therefore use the Chinese remainder theorem to combine the four congruences involving $n$ into a single congruence modulo $Mgd_iD$. However, we shall have to be more precise here than we have hitherto been. Let $a \pmod{Mgd_iD}$ be the residue class arising from the Chinese remainder theorem with
\begin{align}\label{eq:CRT prop one di bigger than x^1/3}
a &\equiv \gamma \pmod{M}, \ \ \ \ a \equiv v \pmod{g}, \ \ \ \ a \equiv i \pmod{d_i}, \ \ \ \ a \equiv \gamma_D \pmod{D}.
\end{align}
The sum over $n$ is therefore equal to
\begin{align*}
\sum_{n \equiv a (Mgd_iD)}\psi_0 \left( \frac{n}{x}\right).
\end{align*}
We set $H = GM (\log x)^4$, and use Lemma \ref{lem:poisson summation} to obtain
\begin{align*}
\sum_{n \equiv a (Mgd_iD)}\psi_0 \left( \frac{n}{x}\right) = \frac{x}{Mgd_iD}\widehat{\psi_0}(0) + \frac{x}{Mgd_iD}\sum_{1 \leq |h| \leq H} \widehat{\psi_0} \left(\frac{hx}{Mgd_iD} \right)e \left(\frac{ha}{Mgd_iD} \right) + O(x^{-10}).
\end{align*}
The contribution of the $O(x^{-10})$ term to $\mathcal{E}_2$ is negligibly small. The main term in our bound for $\mathcal{E}_2$ arises from the zero frequency (the term with $\widehat{\psi_0}(0)$) and the part with complex exponentials gives rise to an error term that we must bound.

Let us first examine the main term. We treat the sum over $g$ by Lemma \ref{lem:main sieve result}. The main term is therefore
\begin{align*}
\ll \mathfrak{S}_k \frac{x}{M} (\log z)^{-k}\sum_{x/B \ll y \ll x}\sum_{x^{49/100} \ll E \ll x^{199/300}} \sum_{i=0}^{k-1}\mathop{\sum\cdots\sum}_{\substack{D\sim E \\ P^-(D) \geq z}} \frac{1}{D} \sum_{\substack{\frac{y}{2E}< d_i \leq \frac{2y}{E} \\ P^-(d_i) \geq z}} \frac{1}{d_i}.
\end{align*}
The sum over $d_i$ is $\ll (\log z)^{-1}$. Once we have summed over $d_i$, we then sum over $E$ to de-localize the scale of $D$. The main term for $\mathcal{E}_2$ is then
\begin{align*}
\ll \mathfrak{S}_k \frac{x}{M} (\log z)^{-k} \cdot \frac{k \log B}{\log z}\mathop{\sum\cdots\sum}_{\substack{x^{49/100} \ll D \ll x^{199/300} \\ P^-(D) \geq z}} \frac{1}{D}.
\end{align*}
We handle the sum over $D$ as in the proof of Proposition \ref{prop:one very large di bigger than x^1/2}. We thereby deduce that the main term for $\mathcal{E}_2$ is
\begin{align*}
\ll \mathfrak{S}_k \frac{x}{M} (\log z)^{-k} \cdot \frac{k \epsilon_k^{-k} (\log B)}{\log x},
\end{align*}
and this is sufficiently small by \eqref{eq:upp bound on k in terms of x} and Proposition \ref{prop:main term lower bound for S}.

We turn to studying the error term $\mathcal{E}_{2,\text{error}}$, which arises from the Fourier frequencies $1 \leq |h| \leq H$. We take the maximum over $x/B \ll y \ll x, x^{49/100} \ll E \ll x^{199/300}$, and $0 \leq i \leq k-1$ to obtain
\begin{align*}
|\mathcal{E}_{2,\text{error}}| &\ll (\log x)^{O(1)} \sum_{1 \leq |h| \leq H} \sum_{\substack{\substack{g \leq G \\ g \mid P(z) \\ P^-(g) \geq 2k}}} \frac{\mu^2(g)}{g}\sum_{v \in S(g)}\mathop{\sum\cdots\sum}_{\substack{D\sim E \\ P^-(D) \geq z}} \\
&\times \Bigg| \sum_{\substack{\frac{y}{2E}< d_i \leq \frac{2y}{E} \\ P^-(d_i) \geq z \\ (d_i,D)=1}}\frac{y}{Ed_i} \widehat{\psi_0} \left(\frac{hx}{Mgd_iD} \right)e \left(\frac{ha}{Mgd_iD} \right) \Bigg|.
\end{align*}
Recall that $a$ depends on $M,g,v,d_i,i$, and $D$. 

The trivial bound for $|\mathcal{E}_{2,\text{error}}|$ is $\lessapprox Hx (\log x)^{O(k)}$, and since $H \approx G = x^{100k\epsilon_k}$, we only need to save a small power of $x$ over this trivial bound. We accomplish this by getting cancellation in the sum over $d_i$. We cannot do this immediately, as $d_i$ is entangled with other variables through the Fourier transform $\widehat{\psi_0}$, and the nature of the dependence of the exponential phase on $d_i$ is also not yet clear.

The residue class $a \pmod{Mgd_iD}$ arises from the Chinese remainder theorem as a solution to
\begin{align*}
a &\equiv a_0 \pmod{gDM}, \ \ \ \ a \equiv i \pmod{d_i},
\end{align*}
where the residue class $a_0 = a_0(g,v,D,M) \pmod{gDM}$ is itself determined by the Chinese remainder theorem from the congruences in \eqref{eq:CRT prop one di bigger than x^1/3}. By Lemma \ref{lem:bezout}, we have
\begin{align*}
\frac{ha}{Mgd_iD} \equiv \frac{ha \overline{d_i}}{MgD} + \frac{ha \overline{MgD}}{d_i} \equiv \frac{h a_0 \overline{d_i}}{MgD} + \frac{hi \overline{MgD}}{d_i} \pmod{1}.
\end{align*}
It is crucial here that $a \equiv i \pmod{d_i}$, so that $hi$ in the second term does not depend on the variable $d_i$. By another use of Lemma \ref{lem:bezout}, we have
\begin{align*}
\frac{hi \overline{MgD}}{d_i} &\equiv -\frac{hi \overline{d_i}}{MgD} + \frac{hi}{Mgd_iD} \pmod{1}.
\end{align*}
It follows that
\begin{align*}
e \left(\frac{ha}{Mgd_iD} \right) = e\left(\frac{h (a_0-i)\overline{d_i}}{MgD} \right)e \left(\frac{hi}{Mgd_iD} \right).
\end{align*}
We can now clearly see the Kloosterman-like dependence of the first exponential factor on $d_i$. Since
\begin{align*}
\left|\frac{hi}{Mgd_iD} \right| \ll \frac{kH}{y} \ll \frac{BkH}{x}
\end{align*}
is so small, we may write
\begin{align*}
e \left(\frac{hi}{Mgd_iD} \right) = 1 + O \left(\frac{BkH}{x} \right)
\end{align*}
and estimate the contribution of the error term trivially, and find the contribution is negligible.

We now separate $d_i$ from the other variables in the Fourier transform. By definition, we have
\begin{align*}
\widehat{\psi_0}\left(\frac{hx}{Mgd_iD} \right) = \int_\mathbb{R} \psi_0(u) e \left(-\frac{hxu}{Mgd_iD} \right) du.
\end{align*}
We change variables $u=d_i t$ to get
\begin{align*}
\widehat{\psi_0}\left(\frac{hx}{Mgd_iD} \right) = d_i \int_\mathbb{R}\psi_0(d_i t) e \left(-\frac{hxt}{MgD} \right) dt.
\end{align*}
Note that the support of $\psi_0$ restricts $t$ to an interval of size $\asymp E/y$.

Combining all of the above, we find that, up to negligible errors,
\begin{align*}
|\mathcal{E}_{2,\text{error}}| &\ll (\log x)^{O(1)} \sum_{1 \leq |h| \leq H} \sum_{\substack{\substack{g \leq G \\ g \mid P(z) \\ P^-(g) \geq 2k}}} \frac{\mu^2(g)}{g}\sum_{v \in S(g)}\mathop{\sum\cdots\sum}_{\substack{D\sim E \\ P^-(D) \geq z}} \\
&\times \Bigg| \sum_{\substack{\frac{y}{2E}< d_i \leq \frac{2y}{E} \\ P^-(d_i) \geq z \\ (d_i,MgD)=1}} \psi_0(d_i t)e\left(\frac{h (a_0-i)\overline{d_i}}{MgD} \right) \Bigg|
\end{align*}
for some $t \asymp E/y$. We have also suppressed the dependence of $a_0$ on $g,v,M,D$ for simplicity (though we cannot forget this dependence). Observe that the condition $P^-(d_i)\geq z$ implies $d_i$ is coprime to $M$ and $g$, but applying a combinatorial identity will lose sight of this fact, so we have recorded this information.

We apply Lemma \ref{lem:buchstab comb decomp} with $R = x^{1/k}$, say, to the sum over $d_i$, obtaining
\begin{align*}
&\sum_{\substack{\frac{y}{2E}< d_i \leq \frac{2y}{E} \\ P^-(d_i) \geq z \\ (d_i,MgD)=1}} \psi_0(d_i t)e\left(\frac{h (a_0-i)\overline{d_i}}{MgD} \right) = \mathop{\sum\sum}_{\substack{\frac{y}{2E} < rs \leq \frac{2y}{E} \\ (rs,MgD)=1 \\ r \mid P(z) \\ r \leq R}} \mu(r)\psi_0(rs t)e\left(\frac{h (a_0-i)\overline{rs}}{MgD} \right) \\
&+ \sum_{2 \leq j \ll \log x}(-1)^j \mathop{\sum\cdots \sum}_{\substack{\frac{y}{2E} < mp_1 \cdots p_j \leq \frac{2y}{E} \\ (mp_1\cdots p_j,MgD)=1 \\ p_j < \cdots < p_1 < z \\ p_1 \cdots p_{j-1} \leq R \\ p_1 \cdots p_{j-1}p_j > R}}\psi_0(mp_1\cdots p_j t) \rho(m,p_j) e\left(\frac{h (a_0-i)\overline{mp_1\cdots p_j}}{MgD} \right).
\end{align*}
This reduces the problem to bounding a ``Type I'' exponential sum (where the variable $s$ is long and smooth), and ``Type II'' exponential sums (with the variables $m,p_1,\ldots,p_j$, where we can group variables into convenient ranges and obtain a bilinear structure).

We first treat the Type I sum, which we arrange as
\begin{align*}
\sum_{\substack{r \leq R \\ r \mid P(z) \\ (r,MgD)=1}}\mu(r)\sum_{\substack{\frac{y}{2rE} < s \leq \frac{2y}{rE} \\ (s,MgD)=1}} \psi_0(rs t)e\left(\frac{h (a_0-i)\overline{rs}}{MgD} \right).
\end{align*}
By partial summation, the sum over $s$ is
\begin{align*}
\ll \sup_{s^* \ll \frac{y}{Er}} \Big|\sum_{\substack{s \leq s^* \\ (s,MgD)=1}} e\left(\frac{h (a_0-i)\overline{rs}}{MgD} \right) \Big|.
\end{align*}
We bound the sum with Lemma \ref{lem:bd on incomplete kloosterman sum}. 

The bound involves
\begin{align*}
\gcd(h(a_0-i),MgD)^{1/2} &\leq (GM)^{1/2} \gcd(h(a_0-i),D)^{1/2}.
\end{align*}
We claim that $\gcd(a_0-i,D)=1$. Assume for contradiction that this is false. Then there is some prime $p$ with $p \mid D$ and $a \equiv i \pmod{p}$. Since $p \mid D$, we have $p \mid d_j$ for some $0 \leq j \leq k-1$ with $j \neq i$, and therefore $a \equiv \gamma_D \equiv j \pmod{p}$. Hence $j-i \equiv 0 \pmod{p}$, but this is a contradiction since $0 < |j-i| < k$ and $p \geq z > k$. It follows that
\begin{align*}
\gcd(h(a-i),MgD)^{1/2} \leq (GM)^{1/2} \gcd(h,D)^{1/2}.
\end{align*}
Hence, by Lemma \ref{lem:bd on incomplete kloosterman sum} and the divisor bound, the sum over $s$ is bounded by
\begin{align*}
\ll G E^{1/2} \gcd(h,D)^{1/2} x^{o(1)} + G^{1/2} \frac{y}{E^{3/2}r}\gcd(h,D)^{1/2} x^{o(1)}.
\end{align*}
We sum this bound over the remaining variables $h,g,v,D,r$, using the fact that
\begin{align*}
\sum_{1 \leq |h| \leq H} \gcd(h,D)^{1/2} \ll H x^{o(1)},
\end{align*}
to get that the contribution of the Type I sum is
\begin{align*}
\ll x^{o(1)} \left(G^{O(1)} E^{3/2} R + G^{O(1)} \frac{x}{E^{1/2}} \right).
\end{align*}
We see this bound is acceptably small once we recall $R = x^{1/k}$ and $x^{49/100} \ll E \ll x^{199/300}$.

We turn to discussing the Type II exponential sums. The key sum to bound is
\begin{align*}
\mathcal{R}_j := \mathop{\sum\cdots \sum}_{\substack{\frac{y}{2E} < mp_1 \cdots p_j \leq \frac{2y}{E} \\ (mp_1\cdots p_j,MgD)=1 \\ p_j < \cdots < p_1 < z \\ p_1 \cdots p_{j-1} \leq R \\ p_1 \cdots p_{j-1}p_j > R}}\psi_0(mp_1\cdots p_j t)\rho (m,p_j)e\left(\frac{h (a_0-i)\overline{mp_1\cdots p_j}}{MgD} \right).
\end{align*}
We want to group the variables $p_1,\ldots,p_j$ together into a single variable, but we are prevented from immediately doing this by the function $\rho(m,p_j)$, which entangles the $p_j$ and $m$ variables. The function $\rho(m,p_j)$ is the indicator function of the condition $P^-(m) \geq p_j$. By \cite[Lemma 9']{DFI1997}, we can write this condition as
\begin{align*}
\int_{-\infty}^\infty \alpha(y) (P^-(m)/p_j)^{iy} dy
\end{align*}
for some function $\alpha$ that satisfies $\int |\alpha(y)| \ll \log x$. We insert this Mellin-type integral to separate $m$ from $p_j$, then move the integral outside and the sum over $m,p_1,\ldots,p_j$ inside. We set $u = p_1 \cdots p_j$, and observe that $R < u \leq Rz$. We dyadically decompose the ranges of $u \sim U$ and $m \sim V$ such that $UV \asymp y/E$. We use \cite[Lemma 9]{DFI1997} to remove the conditions $\frac{y}{2E} < mu \leq \frac{2y}{E}$. We separate $m$ and $u$ inside of $\psi_0$ by using the inverse Mellin transform (using the fact that the Mellin transform of $\psi_0$ decays rapidly). We then have
\begin{align*}
|\mathcal{R}_j| &\ll (\log x)^{O(1)} \sum_{\substack{m \sim V \\ (m,MgD)=1}} \Big|\sum_{\substack{u \sim U \\ (u,MgD)=1}} \delta_u e \left(\frac{h(a_0-i) \overline{mu}}{MgD} \right) \Big|
\end{align*}
for some 1-bounded coefficients $\delta_u$. By Cauchy--Schwarz, we have
\begin{align*}
|\mathcal{R}_j|^2 &\ll (\log x)^{O(1)} V \mathop{\sum\sum}_{\substack{u_1,u_2 \sim U \\ (u_1u_2,MgD)=1}} \left| \sum_{\substack{m \sim V \\ (m,MgD)=1}} e \left(\frac{h(a_0-i)\overline{u_1u_2}(u_2-u_1)\overline{m}}{MgD} \right) \right|.
\end{align*}
We separate the diagonal $u_1 = u_2$ from the off-diagonal $u_1 \neq u_2$ and bound the diagonal trivially to obtain
\begin{align*}
|\mathcal{R}_j|^2 &\ll (\log x)^{O(1)} UV^2 + (\log x)^{O(1)} V\mathop{\sum\sum}_{\substack{u_1,u_2 \sim U \\ u_1 \neq u_2 \\ (u_1u_2,MgD)=1}} \left| \sum_{\substack{m \sim V \\ (m,MgD)=1}} e \left(\frac{h(a_0-i)\overline{u_1u_2}(u_2-u_1)\overline{m}}{MgD} \right) \right|.
\end{align*}

We bound the sum over $m$ with Lemma \ref{lem:bd on incomplete kloosterman sum}, obtaining that the sum over $m$ is
\begin{align*}
\ll (\log x)\tau(MgD)(GM)^{O(1)}\gcd(h(u_2-u_1),D)^{1/2} \left( E^{1/2} + \frac{V}{E^{1/2}} \right).
\end{align*}
Here we have treated the GCD factor as in the Type I sum. We use
\begin{align*}
\gcd(h(u_2-u_1),D)^{1/2} \leq \gcd(h,D)^{1/2} \cdot \gcd(u_2-u_1,D)^{1/2}
\end{align*}
and sum over $u_1 \neq u_2$. Observe that
\begin{align*}
\mathop{\sum\sum}_{\substack{u_1,u_2 \sim U \\ u_1 \neq u_2}} \text{gcd}(u_2-u_1,D)^{1/2} &\ll U \sum_{1 \leq r \leq U} \text{gcd}(r,D)^{1/2} \ll U^2 \sum_{\substack{d \mid D \\ d \leq U}} \frac{1}{d^{1/2}} \ll U^2 \tau(D).
\end{align*}

Therefore
\begin{align*}
|\mathcal{R}_j|^2 &\ll (\log x)^{O(1)} UV^2 + (\log x)^{O(1)} (GM)^{O(1)} \tau(D)^{O(1)} \gcd(h,D)^{1/2} U^2V \left( E^{1/2} + \frac{V}{E^{1/2}} \right).
\end{align*}
We sum this bound over the variables $h,g,v,D$ to see that the total contribution of the Type II sums to $\mathcal{E}_{2,\text{error}}$ is
\begin{align*}
&\ll (\log x)^{O(1)} (GM)^{O(1)}\mathop{\sum\cdots\sum}_{\substack{D\sim E \\ P^-(D) \geq z}} \tau(D)^{O(1)} \left(U^{1/2}V + UV^{1/2}E^{1/4} + \frac{UV}{E^{1/4}} \right) \\
&\ll (GM \log x)^{O(1)} \epsilon_k^{-O(k)} EUV \left(U^{-1/2} + \left(\frac{E^{1/2}}{V} \right)^{1/2} + E^{-1/4} \right).
\end{align*}
This bound is acceptably small because $EUV \ll x, U \gg x^{1/k}, \frac{E^{1/2}}{V} \ll E^{3/2}U y^{-1} \ll B x^{1/k}z E^{3/2} x^{-1}$, and $x^{49/100} \ll E \ll x^{199/300}$. This completes the proof of Proposition \ref{prop:some di bigger than x^1/3}.
\end{proof}

\section{Bilinear forms in short incomplete Kloosterman sums}\label{sec:bilinear}

This section is devoted to proving Proposition \ref{prop:some di has convenient factorization}. As in the proof of Proposition \ref{prop:some di bigger than x^1/3}, we shall have to treat error terms that involve short incomplete Kloosterman sums. These sums are too short, in general, to bound via completion techniques and the Weil bound for Kloosterman sums (cf. Lemma \ref{lem:bd on incomplete kloosterman sum}), so instead we exploit bilinear structure in the error term and bound the exponential sums on average. The key exponential sum estimates are similar to those of Irving \cite{Irving2014} and Bourgain--Garaev \cite{BG2014}.

\begin{lemma}\label{lem:additive energy bound}
Let $S \geq z$, and let $\ell \geq 1$ be a positive integer. If $\ell \leq z^{1/6}$, then
\begin{align*}
\#\Big\{(s_1,\ldots,s_{2\ell}) : \mathbb{N}^{2\ell} &: s_i \sim S, P^-(s_i) \geq z,  \frac{1}{s_1} + \cdots + \frac{1}{s_\ell} = \frac{1}{s_{\ell+1}} + \cdots + \frac{1}{s_{2\ell}}\Big\} \\
&\ll S^\ell \cdot \exp\left(O \left(\frac{\ell^3}{z^{1/2}} + \frac{\ell^2}{\exp(c\sqrt{\log z})} \right) \right) \cdot \left(\frac{\log S}{\log z} \right)^{2\ell^2+\ell},
\end{align*}
where $c>0$ is an absolute constant, and all implied constants are absolute.
\end{lemma}
\begin{proof}
The argument is essentially that of \cite[Lemma 2.2]{Irving2014}, but with some extra work to make the estimate uniform in $\ell$.

Suppose we have
\begin{align*}
\frac{1}{s_1} + \cdots + \frac{1}{s_\ell} = \frac{1}{s_{\ell+1}} + \cdots + \frac{1}{s_{2\ell}}
\end{align*}
with $s_i \sim S$ and $P^-(s_i) \geq z$. We write $N = s_1 \cdots s_{2\ell}$, so that
\begin{align*}
\frac{N}{s_1} + \cdots + \frac{N}{s_\ell} = \frac{N}{s_{\ell+1}} + \cdots + \frac{N}{s_{2\ell}},
\end{align*}
where now every term is an integer. Let $p$ divide some $s_i$. Then every term except for $N/s_i$ is divisible by $s_i$ (and therefore divisible by $p$), so in order to have a solution we must have that $p \mid s_j$ for some $j \neq i$. Therefore, whenever $p \mid N$, we have that $p^2 \mid N$. In other words, $N$ is squarefull.

Since $s_i \sim S$, we have $N \leq S^{2\ell}$. Since $P^-(s_i)\geq z$, we also have that $P^-(N) \geq z$. Similarly, $P^+(N) \leq S$. Given $N$, the number of corresponding $s_1,\ldots,s_{2\ell}$ is $\leq \tau_{2\ell}(N)$. The total number of tuples $(s_1,\ldots,s_{2\ell})$ in question is therefore
\begin{align*}
&\leq \sum_{\substack{N \leq S^{2\ell} \\ N \text{ squarefull} \\ P^-(N) \geq z \\ P^+(N) \leq S}} \tau_{2\ell}(N) \leq S^\ell \sum_{\substack{N \leq S^{2\ell} \\ N \text{ squarefull} \\ P^-(N) \geq z \\ P^+(N) \leq S}} \frac{\tau_{2\ell}(N)}{N^{1/2}} \leq S^\ell \prod_{z \leq p \leq S} \left(1 + \sum_{j=2}^\infty \frac{\tau_{2\ell}(p^j)}{p^{j/2}} \right).
\end{align*}

Taking the $2\ell$-th power of the Euler product of $\zeta(s)$ and using the binomial series reveals
\begin{align*}
\tau_{2\ell}(p^j) = \frac{1}{j!} (2\ell)(2\ell+1)\cdots(2\ell+j-1).
\end{align*}
Therefore $\tau_{2\ell}(p^2) = 2\ell^2+\ell$, and for $j \geq 3$ we have $\tau_{2\ell}(p^j) \leq (2\ell)^j$. We then have
\begin{align*}
&\prod_{z \leq p \leq S} \left(1 + \sum_{j=2}^\infty \frac{\tau_{2\ell}(p^j)}{p^{j/2}} \right) \leq \prod_{z \leq p \leq S} \left(1 + \frac{2\ell^2+\ell}{p}\right)\left(1+ \sum_{j=3}^\infty \left(\frac{2\ell}{p^{1/2}} \right)^j\right) \\
&\leq \exp\left(O \left(\frac{\ell^3}{z^{1/2}} \right) \right)\prod_{z \leq p \leq S} \left(1 + \frac{2\ell^2+\ell}{p}\right) \leq \exp \left(O \left(\frac{\ell^3}{z^{1/2}} + \frac{\ell^2}{\exp(c\sqrt{\log z})} \right) \right) \left(\frac{\log S}{\log z} \right)^{2\ell^2+\ell},
\end{align*}
where in the last step we have used Mertens's theorem with classical prime number theorem error term.
\end{proof}

The following lemma is the main tool for proving Proposition \ref{prop:some di has convenient factorization}.

\begin{lemma}\label{lem:key bilinear estimate}
Let $G = x^{100k\epsilon_k}$, and let $H = GM (\log x)^4$. Assume $R,S \gg x^{k^{-100}}$ and $RS \ll x^{101/300}$. Assume also that $x/\log x \ll ERS \ll x$. Fix $0 \leq i \leq k-1$, and assume $a = a_D \pmod{D}$ is a residue class such that $a-i$ is coprime to $D$. 

For any $1$-bounded sequences $(\alpha_r),(\beta_s)$, we have
\begin{align*}
\sum_{1 \leq |h| \leq H} \sum_{\substack{g \leq G \\ g \mid P(z) \\ P^-(g) \geq 2k}} &\sum_{v \in S(g)} \mathop{\sum\cdots \sum}_{\substack{D \sim E \\ P^-(D) \geq z}} \Bigg|\mathop{\sum\sum}_{\substack{r \sim R, s \sim S \\ P^-(rs) \geq z \\ (rs,D)=1}} \alpha_r \beta_s e \left(\frac{h(a-i) \overline{rs}}{MgD} \right) \Bigg| \\ 
&\ll \exp(O(k^{101}+3^k\log k)) x G^{O(1)} x^{-k^{-201}}.
\end{align*}
\end{lemma}
\begin{proof}
Define
\begin{align*}
\mathcal{T} := \sum_{1 \leq |h| \leq H} \sum_{\substack{g \leq G \\ g \mid P(z) \\ P^-(g) \geq 2k}} \sum_{v \in S(g)} \mathop{\sum\cdots \sum}_{\substack{D \sim E \\ P^-(D) \geq z}} \Bigg|\mathop{\sum\sum}_{\substack{r \sim R, s \sim S \\ P^-(rs) \geq z \\ (rs,D)=1}} \alpha_r \beta_s e \left(\frac{h(a-i) \overline{rs}}{MgD} \right) \Bigg|.
\end{align*}
Let $\ell \geq 1$ be an integer parameter at our disposal. By the triangle inequality and H\"older's inequality, we have
\begin{align*}
|\mathcal{T}|^\ell &\ll (EGHR (\log x)^{O(k)})^{\ell-1} \\
&\times \sum_{1 \leq |h| \leq H} \sum_{\substack{g \leq G \\ g \mid P(z) \\ P^-(g) \geq 2k}} \sum_{v \in S(g)} \mathop{\sum\cdots \sum}_{\substack{D \sim E \\ P^-(D) \geq z}} \sum_{\substack{r \sim R \\ P^-(r) \geq z \\ (r,D)=1}}\Bigg|\sum_{\substack{s \sim S \\ P^-(s) \geq z \\ (s,D)=1}} \beta_s e \left(\frac{h(a-i) \overline{rs}}{MgD} \right) \Bigg|^\ell.
\end{align*}
We write $|\sum_s|^\ell$ as $\eta(\sum_s)^\ell$, where $\eta = \eta(h,g,v,D,r)$ is some unimodular complex number. We therefore have
\begin{align*}
\Bigg|\sum_{\substack{s \sim S \\ P^-(s) \geq z \\ (s,D)=1}} \beta_s e \left(\frac{h(a-i) \overline{rs}}{MgD} \right) \Bigg|^\ell &= \eta \mathop{\sum\cdots \sum}_{\substack{s_1,\ldots,s_\ell \sim S \\ P^-(s_i) \geq z \\ (s_i,D)=1}} \beta_{s_1} \cdots \beta_{s_\ell}e \left(\frac{h(a-i) \overline{r} (\overline{s_1} + \cdots + \overline{s_\ell})}{MgD} \right) \\
&= \eta \sum_{u (MgD)} \nu(u) e \left(\frac{h(a-i) u\overline{r}}{MgD} \right),
\end{align*}
where
\begin{align*}
\nu(u) := \mathop{\sum\cdots \sum}_{\substack{s_1,\ldots,s_\ell \sim S \\ P^-(s_i) \geq z \\ (s_i,D)=1 \\ \overline{s_1} + \cdots + \overline{s_\ell} \equiv u (MgD)}} \beta_{s_1} \cdots \beta_{s_\ell}.
\end{align*}
We therefore have
\begin{align*}
|\mathcal{T}|^\ell &\ll (EGHR (\log x)^{O(k)})^{\ell-1} \\
&\times \sum_{1 \leq |h| \leq H} \sum_{\substack{g \leq G \\ g \mid P(z) \\ P^-(g) \geq 2k}} \sum_{v \in S(g)} \mathop{\sum\cdots \sum}_{\substack{D \sim E \\ P^-(D) \geq z}} \sum_{u (MgD)} |\nu(u)| \left|\sum_{\substack{r \sim R \\ P^-(r) \geq z \\ (r,D)=1}} \eta(h,g,v,D,r) e\left(\frac{h(a-i) u \overline{r}}{MgD} \right) \right|.
\end{align*}

By Cauchy--Schwarz, we have
\begin{align*}
|\mathcal{T}|^{2\ell} &\ll (EGHR (\log x)^{O(k)})^{2\ell-2} \mathcal{T}_1 \cdot \mathcal{T}_2,
\end{align*}
where
\begin{align*}
\mathcal{T}_1 &:= \sum_{1 \leq |h| \leq H} \sum_{\substack{g \leq G \\ g \mid P(z) \\ P^-(g) \geq 2k}} \sum_{v \in S(g)} \mathop{\sum\cdots \sum}_{\substack{D \sim E \\ P^-(D) \geq z}} \sum_{u (MgD)} |\nu(u)|^2, \\
\mathcal{T}_2 &:= \sum_{1 \leq |h| \leq H} \sum_{\substack{g \leq G \\ g \mid P(z) \\ P^-(g) \geq 2k}} \sum_{v \in S(g)} \mathop{\sum\cdots \sum}_{\substack{D \sim E \\ P^-(D) \geq z}} \sum_{u (MgD)} \Big|\sum_{\substack{r \sim R \\ P^-(r) \geq z \\ (r,D)=1}} \eta(h,g,v,D,r) e\left(\frac{h(a-i) u \overline{r}}{MgD} \right) \Big|^2.
\end{align*}

We first bound $\mathcal{T}_1$. We open the square on $|\nu(u)|^2$, interchange the order of summation, and use the triangle inequality to obtain
\begin{align*}
\mathcal{T}_1 &\ll H \sum_{\substack{g \leq G \\ g \mid P(z) \\ P^-(g) \geq 2k}} k^{\omega(g)} \mathop{\sum\cdots \sum}_{\substack{D \sim E \\ P^-(D) \geq z}} \mathop{\sum\cdots \sum}_{\substack{s_1,\ldots,s_{{2\ell}}\sim S\\ P^-(s_i) \geq z \\(s_i,D)=1 \\ \overline{s_1} + \cdots + \overline{s_\ell} \equiv \overline{s_{\ell+1}} + \cdots + \overline{s_{2\ell}} (MgD)}} 1.
\end{align*}
If we define
\begin{align*}
F(s_1,\ldots,s_{2\ell}) = \sum_{j=1}^\ell \prod_{\substack{t = 1 \\ t \neq j}}^{2\ell} s_t - \sum_{j=\ell+1}^{2\ell} \prod_{\substack{t = 1 \\ t \neq j}}^{2\ell} s_t,
\end{align*}
then we see that the congruence on the $s_i$'s modulo $MgD$ is equivalent to $F(s_1,\ldots,s_{2\ell}) \equiv 0 \pmod{MgD}$ (multiply both sides of the congruence by $s_1 \cdots s_{2\ell}$ and then move everything to the same side). We split into two cases, depending on whether $F(s_1,\ldots,s_{2\ell}) = 0$ or not.

If $F(s_1,\ldots,s_{2\ell}) = 0$, then the congruence condition is automatically satisfied. The number of $(s_1,\ldots,s_{2\ell})$ such that $F(s_1,\ldots,s_{2\ell})=0$ is
\begin{align*}
\leq \#\Big\{(s_1,\ldots,s_{2\ell}) : \mathbb{N}^{2\ell} &: s_i \sim S, P^-(s_i) \geq z,  \frac{1}{s_1} + \cdots + \frac{1}{s_\ell} = \frac{1}{s_{\ell+1}} + \cdots + \frac{1}{s_{2\ell}}\Big\} ,
\end{align*}
and this latter quantity may be bounded by Lemma \ref{lem:additive energy bound}, provided we choose $\ell$ such that
\begin{align*}
\ell \leq z^{1/6}.
\end{align*}
We deduce that the total contribution from all those terms to $\mathcal{T}_1$ with $F(s_1,\ldots,s_{2\ell})=0$ is
\begin{align*}
\ll EGH S^\ell\epsilon_k^{-(2\ell^2+\ell)} (\log x)^{O(k)}\exp\left(O \left(\frac{\ell^3}{z^{1/2}} + \frac{\ell^2}{\exp(c\sqrt{\log z})} \right) \right).
\end{align*}

Now we turn to considering the contribution to $\mathcal{T}_1$ from $F=F(s_1,\ldots,s_{2\ell}) \neq 0$. We arrange this contribution as
\begin{align*}
&\ll H\sum_{\substack{g \leq G \\ g \mid P(z) \\ P^-(g) \geq 2k}} k^{\omega(g)} \mathop{\sum\cdots\sum}_{\substack{s_1,\ldots,s_{2\ell} \sim S \\ F \neq 0}} \mathop{\sum\cdots \sum}_{\substack{d_0\ldots d_{i-1}d_{i+1}\ldots d_{k-1} \mid F \\ P^-(d_j) \geq z }} 1.
\end{align*}
If we write $F_{\geq z}$ for the $z$-rough part of $F$ (i.e. the maximal divisor of $F$ all of whose prime factors are $\geq z$), then we see that the innermost sum over the $d_j$ is $\leq \tau_k(F_{\geq z})$. Since $\tau_k(p^a) \leq k^a$, the innermost sum is
\begin{align*}
\leq \tau_k(F_{\geq z}) \leq k^{\Omega(F_{\geq z})}.
\end{align*}
We observe that $|F| \leq 2\ell S^{2\ell-1} \leq S^{2\ell}$, the latter inequality holding if we ensure that
\begin{align*}
2\ell \leq S.
\end{align*}
Since $\Omega(F_{\geq z}) \leq \frac{\log |F|}{\log z} \leq \frac{\log (S^{2\ell})}{\log z}$, we see that the innermost sum over the $d_j$ is
\begin{align*}
\leq (S^{2\ell})^{\frac{\log k}{\log z}} = (S^{2\ell})^{\frac{\epsilon_k^{-1} \log k}{\log x}}.
\end{align*}
We deduce that the total contribution to $\mathcal{T}_1$ from those $s_i$ with $F \neq 0$ is
\begin{align*}
\ll GH S^{2\ell}(S^{2\ell})^{\frac{\epsilon_k^{-1} \log k}{\log x}} (\log x)^{O(k)}.
\end{align*}

Adding our bounds together, we find that
\begin{align*}
\mathcal{T}_1 &\ll \epsilon_k^{-O(\ell^2)}(\log x)^{O(k)}\exp\left(O \left(\frac{\ell^3}{z^{1/2}} + \frac{\ell^2}{\exp(c\sqrt{\log z})} \right) \right)GH \left(ES^\ell + S^{2\ell} (S^{2\ell})^{\frac{\epsilon_k^{-1} \log k}{\log x}}\right),
\end{align*}
where all implied constants are absolute.

We now turn to bounding $\mathcal{T}_2$. We open the square and interchange the order of summation, making the sum over $u$ the innermost sum. Applying character orthogonality and the triangle inequality, we have
\begin{align*}
\mathcal{T}_2 &\ll EGM \sum_{1 \leq |h| \leq H} \sum_{\substack{g \leq G \\ g \mid P(z) \\ P^-(g) \geq 2k}} \sum_{v \in S(g)} \mathop{\sum\cdots \sum}_{\substack{D \sim E \\ P^-(D) \geq z}} \mathop{\sum\sum}_{\substack{r_1,r_2 \sim R \\ P^-(r_j) \geq z \\ (r_1r_2,D)=1 \\ h(a-i)\overline{r_1r_2}(r_2-r_1) \equiv 0 (MgD)}} 1.
\end{align*}
We relax the congruence condition involving the $r_j$ to only be a congruence $\pmod{D}$. Recall that $a-i$ is coprime to $D$. Dividing out by common factors of $h$ and $D$, the congruence condition becomes
\begin{align*}
r_2-r_1 \equiv 0 \pmod{D/(h,D)}.
\end{align*}
Since $r_j \sim R \ll x^{101/300}$ and $\frac{D}{(h,D)} \gg E/H \gg x^{199/300-o(1)}/H$ and $H \leq x^{1/10}$, say, the congruence forces $r_1 = r_2$. We deduce that
\begin{align*}
\mathcal{T}_2 &\ll E^2G^2HMR (\log x)^{O(k)}.
\end{align*}

We put our bounds together and take the $2\ell$-th root to obtain
\begin{align*}
|\mathcal{T}|&\ll \epsilon_k^{-O(\ell)} \exp\left(O \left(\frac{\ell^2}{z^{1/2}} + \frac{\ell}{\exp(c\sqrt{\log z})} \right) \right) (\log x)^{O(k)} (EGHRS) \frac{G^{1/2\ell}}{R^{1/2\ell}} \left(\frac{E}{S^\ell} +  (S^{2\ell})^{\frac{\epsilon_k^{-1} \log k}{\log x}}\right).
\end{align*}
We choose $\ell = \lceil \frac{\log E}{\log S}\rceil \geq 2$. Then $E/S^\ell \leq 1$, and $S^{2\ell} \leq S^2 E^2 \leq x^{O(1)}$. By our lower bound on $S \gg A_0$, we see that $\ell \ll k^{100}$. It follows that
\begin{align*}
R^{1/2\ell} \geq x^{k^{-201}},
\end{align*}
say. 

Since $\epsilon_k = 3^{-k}, z=x^{\epsilon_k}$, and $k \ll (\log \log x)^{1/2}$, we see that $2 \ell \leq S$ and $\ell \leq z^{1/6}$. We also have
\begin{align*}
\frac{\ell^2}{z^{1/2}} + \frac{\ell}{\exp(c\sqrt{\log z})} \ll 1.
\end{align*}
Our final bound therefore becomes
\begin{align*}
|\mathcal{T}| &\ll \exp(O(k^{101}+3^k\log k)) x G^{O(1)} x^{-k^{-201}}. \qedhere
\end{align*}

\end{proof}

\begin{proof}[Proof of Proposition \ref{prop:some di has convenient factorization}]
Recall that
\begin{align*}
\mathcal{E}_3 &= \sum_{\substack{n \sim x \\ n \equiv \gamma (M) \\ P^-(\binom{n}{k})\geq z \\ \exists d \in (n/B,n] \textup{ such that} \\ d \mid \binom{n}{k} \textup{ and }d_i \leq x^{101/300} \, \forall i \textup{ and } \\ \exists d_i=rs \textup{ with } r,s > x^{k^{-100}}}} 1.
\end{align*}
By assumption, there is some $d_i$ with $0 \leq i \leq k-1$ such that $d_i$ factors as $d_i = rs$, where $r$ and $s$ are both $> x^{k^{-100}} =: A_0$. Given such a $d_i$, we may therefore factor $d = Dd_i = Drs$. We sum over all ways of writing $d_i = rs$ for an upper bound, and handle the condition $Drs \mid \binom{n}{k}$ in the typical way to obtain
\begin{align*}
\mathcal{E}_3 &\leq \sum_{i=0}^{k-1} \mathop{\sum\cdots \sum}_{\substack{P^-(D) \geq z}} \mathop{\sum\sum}_{\substack{r,s > A_0 \\ P^-(rs)\geq z \\ \frac{x}{2B} < Drs \leq x \\ (rs,D)=1}}\sum_{\substack{n \sim x \\ n \equiv \gamma (M) \\ P^-(\binom{n}{k})\geq z \\ n \equiv a_D (D) \\ n \equiv i (rs)}} 1.
\end{align*}
As usual, $a_D \pmod{D}$ arises from the Chinese remainder theorem via the congruences $a_D \equiv j \pmod{d_j}$ for all $d_j$ that divide $D$ ($j \neq i$). We insert a smooth factor $\psi_0 (\frac{n}{x})$ for an upper bound, and dyadically decompose the variables $D \sim E, r\sim R, s \sim S$ with $x/B \ll ERS \ll x$, $R,S \gg A_0$. Note that $RS \ll x^{101/300}$. We then insert upper-bound sieve weights to control the condition $P^-(\binom{n}{k}) \geq z$, and take the level of distribution $G$ of the sieve weights to be $G = x^{100k\epsilon_k}$. We therefore obtain
\begin{align*}
\mathcal{E}_3 &\leq \sum_{i=0}^{k-1}\sum_{E,R,S}\sum_{\substack{g\leq G \\ g \mid P(z) \\ P^-(g) \geq 2k}}\lambda_g^+ \sum_{v \in S(g)} \mathop{\sum\cdots \sum}_{\substack{D \sim E \\ P^-(D) \geq z}} \mathop{\sum\sum}_{\substack{r \sim R, s \sim S \\ P^-(rs)\geq z \\ \frac{x}{2B} < Drs \leq x \\ (rs,D)=1}}\sum_{\substack{n \equiv \gamma (M) \\ n \equiv v (g) \\ n \equiv a_D (D) \\ n \equiv i (rs)}} \psi_0 \left( \frac{n}{x}\right).
\end{align*}

The integers $M,g,D,rs$ are pairwise coprime, so we may combine all the congruence conditions into a single congruence condition modulo $a \pmod{MgDrs}$ by the Chinese remainder theorem. (We do not denote it in the notation, but it is vital to recall that $a$ depends on many other variables.) We set $H = GM (\log x)^4$ and use Poisson summation (Lemma \ref{lem:poisson summation}) to evaluate the sum over $n$. The $h=0$ frequency gives rise to the main term $\mathcal{E}_{3,\text{MT}}$, and the $h \neq 0$ frequencies give rise to the error term $\mathcal{E}_{3,\text{error}}$ (together with a completely negligible error of size $O(x^{-10})$, say).

We first consider the main term, which is
\begin{align*}
\mathcal{E}_{3,\text{MT}} := \frac{x}{M}\widehat{\psi_0}(0)\sum_{i=0}^{k-1}\sum_{E,R,S}\sum_{\substack{g\leq G \\ g \mid P(z) \\ P^-(g) \geq 2k}}\lambda_g^+ \frac{k^{\omega(g)}}{g} \mathop{\sum\cdots \sum}_{\substack{D,r,s \\ D \sim E, r \sim R, s \sim S  \\ \frac{x}{2B} < Drs \leq x \\ P^-(Drs) \geq z}} \frac{1}{Drs}.
\end{align*}
We use Lemma \ref{lem:main sieve result} to get that the sum over $g$ is $\ll \mathfrak{S}_k (\log z)^{-k}$. Note that, given $E$ and $R$, there are only $O(\log B)$ choices for $S$. For fixed $S$, the sum over $s$ is $\ll (\log z)^{-1}$. We then sum over $E$ and $R$ to de-localize the variables $D$ and $r$. We deduce that
\begin{align*}
\mathcal{E}_{3,\text{MT}} &\ll \frac{x}{M}\mathfrak{S}_k (\log z)^{-k} \cdot \frac{k\log B}{\log z} \Bigg(\sum_{\substack{f \ll x \\ P^-(f) \geq z}} \frac{1}{f} \Bigg)^k \ll \frac{x}{M}\mathfrak{S}_k (\log z)^{-k} \cdot \frac{\epsilon_k^{-O(k)}}{\log x},
\end{align*}
and this is $o(\mathcal{S})$ by Proposition \ref{prop:main term lower bound for S}.

We turn now to the error term, which is
\begin{align*}
\mathcal{E}_{3,\text{error}} &:= \frac{x}{MERS} \sum_{i=0}^{k-1}\sum_{E,R,S} \sum_{1 \leq |h| \leq H} \sum_{\substack{g \leq G \\ g \mid P(z) \\ P^-(g) \geq 2k}} \lambda_g^+ \sum_{v \in S(g)}\mathop{\sum\cdots \sum}_{\substack{D \sim E \\ P^-(D) \geq z}} \\ 
&\times\mathop{\sum\sum}_{\substack{r \sim R, s \sim S \\ P^-(rs)\geq z \\ \frac{x}{2B} < Drs \leq x \\ (rs,D)=1}} \frac{ERS}{Drs} \widehat{\psi_0} \left(\frac{h x}{MgDrs} \right) e \left(\frac{ha}{MgDrs} \right).
\end{align*}
There are several preparatory moves we need to make before we can apply Lemma \ref{lem:key bilinear estimate}. (We had to do similar maneuvers in the proof of Proposition \ref{prop:some di bigger than x^1/3}.)

By Lemma \ref{lem:bezout}, we have
\begin{align*}
\frac{ha}{MgDrs} &\equiv \frac{ha \overline{rs}}{MgD} + \frac{ha \overline{MgD}}{rs} \equiv \frac{ha_0 \overline{rs}}{MgD} + \frac{hi \overline{MgD}}{rs} \pmod{1},
\end{align*}
where have written $a_0$ in the first term which is independent of $r$ and $s$, and we have used that $a \equiv i \pmod{rs}$ to simplify the second term. Having done so, we now apply Lemma \ref{lem:bezout} again to get
\begin{align*}
\frac{ha}{MgDrs} &= \frac{h(a_0-i) \overline{rs}}{MgD} + \frac{hi}{MgDrs} \pmod{1}.
\end{align*}
The second term has size $\ll \frac{Hk}{ERS}$, and is therefore very small. The contribution of this term to the total error is negligible.

Next, we change variables inside the Fourier transform $\widehat{\psi_0}$ to separate $r$ and $s$ from the other variables:
\begin{align*}
\widehat{\psi_0} \left(\frac{h x}{MgDrs} \right) &= \int_{-\infty}^\infty \psi_0(t) e \left(-\frac{hxt}{MgDrs} \right) dt = \frac{rs}{RS} \int_{-\infty}^\infty \psi_0 \left(\frac{rs}{RS}w \right) e \left(-\frac{hxw}{MgDRS} \right) dw.
\end{align*}
We move the integral over $w$ outside over all the sums, and take the worst $w \asymp 1$. We then separate $r$ and $s$ from one another inside $\psi_0(\frac{rs}{RS}w)$ by using the Mellin transform (this can be done at no cost by the rapid decay of the Mellin transform of $\psi_0$).

Lastly, we separate $D$ from $r$ and $s$ in the conditions $\frac{x}{2B} < Drs \leq x$ by Mellin-type or Perron-type integrals (\cite[Lemmas 9 and 9']{DFI1997}).

We therefore find that
\begin{align*}
|\mathcal{E}_{3,\text{error}}| &\ll (\log x)^{O(1)} \sum_{1 \leq |h| \leq H} \sum_{\substack{g \leq G \\ g \mid P(z) \\ P^-(g) \geq 2k}} \sum_{v \in S(g)} \mathop{\sum\cdots \sum}_{\substack{D \sim E \\ P^-(D) \geq z}} \Bigg|\mathop{\sum\sum}_{\substack{r \sim R, s \sim S \\ P^-(rs) \geq z \\ (rs,D)=1}} \alpha_r \beta_s e \left(\frac{h(a_0-i) \overline{rs}}{MgD} \right) \Bigg|
\end{align*}
for some $i,E,R,S$, where $\alpha_r$ and $\beta_s$ are complex 1-bounded sequences.  Observe that $\text{gcd}(a_0-i,D)=1$ (see the proof of Proposition \ref{prop:some di bigger than x^1/3}).

At this point, if we apply a trivial bound, the error is too large by a factor of $\approx G^{O(1)}$. We may obtain suitable savings by applying Lemma \ref{lem:key bilinear estimate}, which saves a factor of $x^{-O(k^{-300})}$, say, over the trivial bound. (Here we use that $k \ll (\log \log x)^{1/2}$ to control the factor $\exp(O(k^{101}+3^k\log k))$.) This is sufficient to show that $|\mathcal{E}_{3,\text{error}}|$ is of lower order than $\mathcal{S}$.
\end{proof}

\section{Short incomplete Kloosterman sums and Weyl differencing}\label{sec:differencing}

We need a result on square-root cancellation for complete exponential sums.

\begin{lemma}[Bound for complete exponential sum]\label{lem:alg exp sum mod p}
Let $P,Q \in \mathbb{Z}[X]$ be coprime polynomials over $\mathbb{Z}$ in one indeterminate $X$. Let $p$ be a prime number such that $ Q \pmod{p} \in \mathbb{F}_p[X]$ is non-zero, such that there is no identity of the form
\begin{align*}
\frac{P}{Q} \, (\textup{mod }p) = g^p-g+c \in \mathbb{F}_p(X)
\end{align*}
for some rational function $g = g(X) \in \mathbb{F}_p(X)$ and some $c \in \mathbb{F}_p$. Then we have
\begin{align*}
\left| \sum_{\substack{x \in \mathbb{F}_p \\ (Q(x),p)=1}}e \left(\frac{P(x) \overline{Q(x)}}{p} \right)\right| \ll (\textup{deg}(P) + \textup{deg}(Q)) \sqrt{p}.
\end{align*}
\end{lemma}
\begin{proof}
This is \cite[Lemma 4.2]{Polymath2014}.
\end{proof}

The following lemma is our main tool in the proof of Proposition \ref{prop:bound for E5 that relies on q van der corput}. The argument is essentially that of \cite[Theorem 2]{HB2001}, but we require more uniformity in the parameters than is present there.

\begin{lemma}[Weyl differencing]\label{lem:weyl differencing}
Let $N \geq x^{1/(2k)}$, and let $s_0,s_1,\ldots,s_J$ be pairwise coprime integers with $s_1 < \ldots <s_J$ being distinct primes, where $J \geq 2$. Assume $1\leq s_0 \leq x^{1/k^{50}}$, and assume $x^{1/k^{100}} < s_j \leq N x^{-k^{-5}}$ for $1 \leq j \leq J \leq k$. Assume $s_J > x^{1/(10k)}$.

Let $a$ be an integer, and let $b$ be an integer that is coprime to $s_0\cdots s_J$. For any $\xi \in \mathbb{R}$, we have
\begin{align*}
\sum_{\substack{n \sim N \\ (n,s_0\cdots s_J)=1}} e(\xi n )e \left(\frac{a \overline{bn}}{s_0\cdots s_J} \right) \ll N \cdot \gcd(a,s_J) \cdot x^{-k^{-101}2^{-k}}.
\end{align*}
\end{lemma}
\begin{proof}
Our goal is to apply Weyl differencing repeatedly to reduce the conductor of the exponential phase. Before doing this, however, we separate the influence of $s_0$ from the sum, since differencing in multiples of $s_0$ would multiply technical conditions later on.

By Lemma \ref{lem:bezout}, we have
\begin{align*}
\frac{a \overline{bn}}{s_0\cdots s_J} \equiv \frac{a\overline{bns_1\cdots s_J}}{s_0} + \frac{a\overline{bs_0 n}}{s_1\cdots s_J} \pmod{1}.
\end{align*}
To control the term with denominator $s_0$, we split $n$ into primitive residue classes $c \pmod{s_0}$. We change variables $n \rightarrow c+ns_0$ to find the sum in question is
\begin{align}\label{eq:split mod s0}
\sum_{(c,s_0)=1} e(\xi c) e \left(\frac{a\overline{bc s_1\cdots s_J}}{s_0} \right)\sum_{\substack{n \sim N_1 \\ (c+ns_0,s_1 \cdots s_J)=1}} e(\xi s_0 n) e\left( \frac{a\overline{bs_0 (c+s_0 n)}}{s_1\cdots s_J} \right) + O(s_0),
\end{align}
where we have written $N_1 = N/s_0$.  We observe that a trivial bound at this point recovers the trivial bound $O(N)$ for the sum, so we have not lost anything by these maneuvers.

Let
\begin{align*}
\mathcal{B} := \sum_{\substack{n \sim N_1 \\ (c+ns_0,s_1 \cdots s_J)=1}} e(\xi s_0 n) e\left( \frac{a\overline{bs_0 (c+s_0 n)}}{s_1\cdots s_J} \right)
\end{align*}
denote the inner sum in \eqref{eq:split mod s0}. We define
\begin{align*}
K := x^{1/k^{10}},
\end{align*}
which will control the size of our differencing variables.

Suppose that for some nonnegative integer $j$ we have
\begin{align*}
|\mathcal{B}|^{2^j} &\leq \frac{2^{2^j-1}N_1^{2^j-1}}{K^{2j}} \mathop{\sum\cdots\sum}_{\substack{k_{1,1},k_{1,2},\ldots,k_{j,1},k_{j,2}\leq K \\ k_{\ell,1} \neq k_{\ell,2}, 1 \leq \ell \leq j}}\Bigg| \sum_{\substack{n \sim N_1 \\ (Q_j(n),s_{j+1}\cdots s_J)=1}} e(\xi_j n) e \left(a\overline{b_j}\frac{t_j(n)}{s_{j+1}\cdots s_J} \right) \Bigg| + E_j
\end{align*}
for some $E_j \geq 0$, some real $\xi_j$ (equal to zero when $j \geq 1$), some integer $b_j$ coprime to $s_{j+1}\cdots s_J$, and some rational function $t_j(n) = P_j(n) \overline{Q_j(n)}$ (which may depend on $k_{\ell,i}$) with $Q_j(n)$ being a product of $2^j$ linear factors of the form $\alpha n +\beta$, where $\alpha$ is coprime to $s_{j+1}\cdots s_J$. This bound trivially holds with $j=0$ by taking $P_0(n) = 1, Q_0(n) = c+s_0n, b_0 = bs_0$. We proceed by induction, and show that the bound holds when $j$ is replaced by $j+1$.

We shift $n$ by $k_{j+1}s_{j+1}$ and average over all $k_{j+1} \leq K$ to get
\begin{align*}
|\mathcal{B}|^{2^{j}} &\leq \frac{2^{2^j-1}N_1^{2^j-1}}{K^{2j+1}} \mathop{\sum\cdots\sum}_{\substack{k_{1,1},k_{1,2},\ldots,k_{j,1},k_{j,2}\leq K \\ k_{\ell,1} \neq k_{\ell,2}, 1 \leq \ell \leq j}} \sum_{n \sim N_1} \Bigg| \sum_{\substack{k_{j+1} \leq K \\ (Q_j(n+k_{j+1}s_{j+1}),s_{j+1}\cdots s_J)=1}} e(\xi_j k_{j+1}s_{j+1}) \\ 
&\times e \left(a\overline{b_j}\frac{t_j(n+k_{j+1}s_{j+1})}{s_{j+1}\cdots s_J} \right) \Bigg| + 2^{2^j}N_1^{2^j}x^{-k^{-6}}+E_j.
\end{align*}
We now square both sides of the inequality and apply Cauchy--Schwarz twice (the first time in the form $|a+b|^2 \leq 2|a|^2+2|b|^2$) to get
\begin{align*}
|\mathcal{B}|^{2^{j+1}} &\leq \frac{2^{2^{j+1}-1}N_1^{2^{j+1}-1}}{K^{2j+2}}\mathop{\sum\cdots\sum}_{\substack{k_{1,1},k_{1,2},\ldots,k_{j,1},k_{j,2}\leq K \\ k_{\ell,1} \neq k_{\ell,2}, 1 \leq \ell \leq j}} \mathop{\sum\sum}_{k_{j+1,1},k_{j+1,2} \leq K} \\
&\times \Bigg| \sum_{\substack{n \sim N_1 \\ (Q_j(n+k_{j+1,1}s_{j+1}),s_{j+1}\cdots s_J)=1 \\ (Q_j(n+k_{j+1,2}s_{j+1}),s_{j+1}\cdots s_J)=1}} e \left(a\overline{b_j} \frac{t_j(n+k_{j+1,1}s_{j+1})-t_j(n+k_{j+1,2}s_{j+1})}{s_{j+1}\cdots s_J} \right) \Bigg| \\
&+ 2 \left(2^{2^j}N_1^{2^j}x^{-k^{-6}}+E_j \right)^2.
\end{align*}
The contribution from $k_{j+1,1} = k_{j+1,2}$ has size $\leq 2^{2^{j+1}-1}N_1^{2^{j+1}}K^{-1}$. For $k_{j+1,1} \neq k_{j+1,2}$, we have
\begin{align*}
t_j(n+k_{j+1,1}s_{j+1})-t_j(n+k_{j+1,2}s_{j+1}) = t_{j+1}(n)= P_{j+1}(n) \overline{Q_{j+1}(n)},
\end{align*}
where
\begin{align*}
P_{j+1}(n) &= P_j(n+k_{j+1,1}s_{j+1})Q_j(n+k_{j+1,2}s_{j+1}) - P_j(n+k_{j+1,2}s_{j+1})Q_j(n+k_{j+1,1}s_{j+1}) \\
Q_{j+1}(n) &= Q_j(n+k_{j+1,1}s_{j+1})Q_j(n+k_{j+1,2}s_{j+1}).
\end{align*}
Note that $Q_{j+1}(n)$ is the product of $2^{j+1}$ linear factors. The polynomial $P_{j+1}$ is $\equiv 0 \pmod{s_{j+1}}$. By Lemma \ref{lem:bezout}, we have
\begin{align*}
a\overline{b_j} \frac{t_{j+1}(n)}{s_{j+1}\cdots s_J} &\equiv a\overline{b_j}\frac{P_{j+1}(n) \overline{Q_{j+1}(n) s_{j+2}\cdots s_J}}{s_{j+1}} + a\overline{b_j} \frac{t_{j+1}(n) \overline{s_{j+1}}}{s_{j+2}\cdots s_J} \\
&\equiv a\overline{b_{j+1}} \frac{t_{j+1}(n)}{s_{j+2}\cdots s_J} \pmod{1},
\end{align*}
where we have set $b_{j+1} = b_j s_{j+1}$.

We may rewrite the coprimality conditions on $n$ as $(Q_{j+1}(n),s_{j+1}\cdots s_J) = 1$. We may add in the contribution from those $n \sim N_1$ with $(Q_{j+1}(n),s_{j+1}) > 1$ at the cost of an error of size
\begin{align*}
\leq 2^{2^{j+1}+j} \frac{N_1^{2^{j+1}}}{s_{j+1}} < 2^{2^{j+1}+j} N_1^{2^{j+1}} x^{-k^{-100}}.
\end{align*}
We therefore have
\begin{align*}
|\mathcal{B}|^{2^{j+1}} &\leq \frac{2^{2^{j+1}-1}N_1^{2^{j+1}-1}}{K^{2j+2}}\mathop{\sum\cdots\sum}_{\substack{k_{1,1},k_{1,2},\ldots,k_{j+1,1},k_{j+1,2}\leq K \\ k_{\ell,1} \neq k_{\ell,2}, 1 \leq \ell \leq j+1}} \Bigg| \sum_{\substack{n \sim N_1 \\ (Q_{j+1}(n), s_{j+2}\cdots s_J)=1}} e \left(a\overline{b_{j+1}} \frac{t_{j+1}(n)}{s_{j+2}\cdots s_J} \right) \Bigg| \\
&+ E_{j+1},
\end{align*}
where
\begin{align}\label{eq:recursive bound on Ej}
E_{j+1} &= 2 \left(2^{2^j}N_1^{2^j}x^{-k^{-6}}+E_j \right)^2 + 2^{2^{j+1}-1}N_1^{2^{j+1}}K^{-1} + 2^{2^{j+1}+j} N_1^{2^{j+1}} x^{-k^{-100}}.
\end{align}
This completes the induction (with $\xi_{j+1} = 0$).

If we take $j = J-1$, we have
\begin{align}\label{eq:q van der corput mathcal B J-1 bound}
|\mathcal{B}|^{2^{J-1}} &\leq \frac{2^{2^{J-1}-1}N_1^{2^{J-1}-1}}{K^{2(J-1)}} \mathop{\sum\cdots\sum}_{\substack{k_{1,1},k_{1,2},\ldots,k_{J-1,1},k_{J-1,2}\leq K \\ k_{\ell,1} \neq k_{\ell,2}, 1 \leq \ell \leq J-1}}\Bigg| \sum_{\substack{n \sim N_1 \\ (Q_{J-1}(n),s_J)=1}} e \left(a\overline{b_{J-1}}\frac{t_{J-1}(n)}{s_J} \right) \Bigg| + E_{J-1}.
\end{align}
Starting with $E_0=0$, we can use \eqref{eq:recursive bound on Ej} to bound $E_{J-1}$. If we set $F_j = \frac{E_j}{(2N_1)^{2^j}}$ for $j \geq 0$ and
\begin{align*}
C = C_J = x^{-k^{-6}} + \frac{1}{2K} + 2^J x^{-k^{-100}},
\end{align*}
then by \eqref{eq:recursive bound on Ej}, if $j+ 1\leq J$, then
\begin{align}\label{eq:recursive bound on Fj}
F_{j+1} = \frac{E_{j+1}}{(2N_1)^{2^{j+1}}} \leq 2 \left(\frac{E_j}{(2N_1)^{2^j}} + x^{-k^{-6}} \right)^2 + \frac{1}{2K} + 2^j x^{-k^{-100}} \leq 2(F_j+C)^2 + C.
\end{align}
Since $x$ is large and $k$ is small compared to $x$, we may assume $C \leq \frac{1}{18}$. We claim that \eqref{eq:recursive bound on Fj} implies $F_\ell \leq 2C$ for each $\ell\leq J-1$. The bound trivially holds for $\ell = 0$ since $F_0 = 0$. Assuming the bound holds for $F_\ell$, we see that
\begin{align*}
F_{\ell+1} &\leq 2(F_\ell+C)^2 + C \leq 2 \cdot (3C)^2 + C = C \cdot (1+18C) \leq 2C.
\end{align*}
We therefore obtain
\begin{align*}
E_{J-1} &\leq 2\cdot (2N_1)^{2^{J-1}} \cdot \left( x^{-k^{-6}} + \frac{1}{2K} + 2^J x^{-k^{-100}}\right).
\end{align*}

Now we turn to studying the sum over $n \sim N_1$ in \eqref{eq:q van der corput mathcal B J-1 bound}. By breaking the sum over $n$ into arithmetic progressions modulo $s_J$ and estimating trivially, we have
\begin{align*}
\sum_{\substack{n \sim N_1 \\ (Q_{J-1}(n),s_J)=1}} e \left(a\overline{b_{J-1}}\frac{t_{J-1}(n)}{s_J} \right) &= \frac{N_1}{2s_J}\cdot \sum_{\substack{(Q_{J-1}(x),s_J)=1}} e \left(a\overline{b_{J-1}}\frac{t_{J-1}(x)}{s_J} \right) + O(s_J).
\end{align*}
The sum over $x$ is an exponential sum in one variable over a finite field of prime order. If $s_J \mid a$, then the sum has no cancellation. Thus, we may assume $s_J \nmid a$. We show there is square-root cancellation in this sum by checking the conditions of Lemma \ref{lem:alg exp sum mod p}.

By our induction argument above, $Q_{J-1}(n)$ is nonzero modulo $s_J$, since the leading term of $Q_{J-1}(n)$ has a nonzero coefficient modulo $s_J$. If we cancel common factors in $P/Q$ to get $P'/Q'$ with $P',Q'$ coprime, then we see that $Q'$ is nonzero modulo $s_J$.

Note that if $\frac{P}{Q} \, (\textup{mod }s_J) = g^{s_J}-g+c$ for some rational function $g$, then all the poles of $P/Q$ have order divisible by $s_J$. However, we can show by induction that any poles of $t_{J-1}$ are, at most, simple. The result is true for $j=0$. If $t_j(n)$ has only simple poles, then
\begin{align*}
t_{j+1}(n) = t_j(n+k_{j+1,1}s_{j+1})-t_j(n+k_{j+1,2}s_{j+1})
\end{align*}
has at most simple poles, since subtracting rational functions with only simple poles can only change the residues of the poles, but not the order of the poles. Hence, we deduce that $t_{J-1}(n)$ has, at most, simple poles. Thus, it follows that $\frac{P}{Q} \, (\textup{mod }s_J)$ is not of the form $g^{s_J}-g+c$ if we can show that $\frac{P}{Q} \, (\textup{mod }s_J)$ has at least one pole. 

We may also prove this by induction. The result is clearly true for $t_0(n)$. Now let $\mathcal{P}_j=\{r_1,\ldots,r_m\}\subseteq \mathbb{F}_{s_J}$ be the poles of $t_j(n)$, where $1\leq m \leq 2^j$. Thus, $\mathcal{P}_j$ is nonempty, and is a proper subset of $\mathbb{F}_{s_J}$ (since $s_J$ is much larger than $2^j$). Set $d = (k_{j+1,1}-k_{j+1,2})s_{j+1} \in \mathbb{F}_{s_J}^\times$. Since $\mathcal{P}_j$ is a nonempty proper subset of $\mathbb{F}_{s_J}$, we have $\mathcal{P} \neq \mathcal{P} - d$. Let $r_\ell \in \mathcal{P}_j$ such that $r_\ell - d \not \in \mathcal{P}_j$. Then $t_j (n+k_{j+1,1}s_{j+1})$ has a pole at $n = r_\ell - k_{j+1,1}s_{j+1}$, but $t_j(n+k_{j+1,2}s_{j+1})$ does not have a pole at $n = r_\ell - k_{j+1,1}s_{j+1}$, since $r_\ell +k_{j+1,2}s_{j+1}-k_{j+1,1}s_{j+1} = r_\ell - d \not \in \mathcal{P}_j$. We deduce that $t_{j+1}(n)$ has a pole, and $\mathcal{P}_{j+1} \neq \emptyset$.

We may therefore apply Lemma \ref{lem:alg exp sum mod p}. One can prove by induction that $\text{deg}(P_{J-1}) \leq 2^{J-1}-1$, and we have a similar bound for $\text{deg}(Q_{J-1})$. We deduce that
\begin{align*}
\sum_{\substack{(Q_{J-1}(x),s_J)=1}} e \left(a\overline{b_{J-1}}\frac{t_{J-1}(x)}{s_J} \right) \ll 2^J \cdot \gcd(a,s_J)^{1/2}\cdot  s_J^{1/2},
\end{align*}
where the implied constant is absolute. 

Putting all our bounds together, we find
\begin{align*}
|\mathcal{B}|^{2^{J-1}} &\ll (2N_1)^{2^{J-1}-1} \cdot 2^J\gcd(a,s_J)^{1/2} N_1 x^{-k^{-10}} + (2N_1)^{2^{J-1}}\cdot 2^Jx^{-k^{-100}} \\ 
&\ll 2^J \cdot (2N_1)^{2^{J-1}}\gcd(a,s_J)^{1/2} x^{-k^{-101}}.
\end{align*}
We take the $2^{J-1}$-st root of both sides to obtain
\begin{align*}
|\mathcal{B}| &\ll N_1 \cdot \gcd(a,s_J) \cdot x^{-k^{-101}2^{-k}},
\end{align*}
say, where we have weakened the bound on the GCD and used $J \leq k$. We finish by examining \eqref{eq:split mod s0}.
\end{proof}

\begin{proof}[Proof of Proposition \ref{prop:bound for E5 that relies on q van der corput}]
Recall that
\begin{align*}
\mathcal{E}_5 &= \sum_{\substack{n \sim x \\ n \equiv \gamma (M) \\ P^-(\binom{n}{k})\geq z \\ \exists d \in (n/B,n] \textup{ such that} \\ d \mid \binom{n}{k} \textup{ and } \\ d = \widetilde{f}q_0\cdots q_{k-1} \textup{ where} \\ q_j \leq q_{i_0} x^{-k^{-4}} \, \forall j \neq i_0 \textup{ and} \\ q_* > x^{1/2k}}} 1.
\end{align*}
The variables $q_j$ are equal to one, or are distinct primes $> x^{1/k^{100}}$, $q_{i_0}$ is the largest of these primes, and $q_*$ is the second-largest of these primes. Every $q_j$ with $j \neq i_0$ is a bit smaller than $q_{i_0}$. The number $\widetilde{f} = f_0 \cdots f_{k-1}$, where $f_j$ is the product of primes dividing $d_j$ that are $\leq x^{1/k^{100}}$. We may assume that $\widetilde{f}\leq x^{3/k^{99}}$. We sum over the different possibilities for $i_0$ and handle the condition $d = \widetilde{f}q_0 \cdots q_{k-1}$ in the usual way to get
\begin{align*}
\mathcal{E}_5 &\leq \sum_{i_0=0}^{k-1}\mathop{\sum\cdots\sum}_{\substack{\widetilde{f}=f_0\cdots f_{k-1} \leq x^{3/k^{99}} \\ P^+(f_i) \leq x^{1/k^{100}} \\ P^-(f_i) \geq z}} \mathop{\sum\cdots\sum}_{\substack{q_0,\ldots,q_{k-1} = 1 \text{ or } > x^{1/k^{100}} \\ (q_j,q_i) = 1, i \neq j \\ q_j \leq q_{i_0} x^{-k^{-4}}, \forall j \neq i_0 \\ \frac{x}{2B} < \widetilde{f}q_0\cdots q_{k-1} \leq x}} \sum_{\substack{n \sim x \\ n \equiv \gamma (M) \\ n \equiv j (q_j) \\ n \equiv a (\widetilde{f}) \\ P^-(\binom{n}{k}) \geq z}} 1.
\end{align*}
The residue class $a \pmod{\widetilde{f}}$ depends on all the $f_j$ through the Chinese remainder theorem. We introduce a smooth function $\psi_0 (\frac{n}{x})$ to smooth the $n$ variable, and drop the condition $n \sim x$.

Since $q_{i_0}$ is the largest of the $q_j$, we have that $q_{i_0} \gg x^{(1-o(1))/k}$. We relax the condition that $q_{i_0}$ is prime by noting
\begin{align*}
\mathbf{1}(q_{i_0} \text{ is prime}) \leq \mathbf{1}(p \mid q_{i_0}\Rightarrow p \geq x^{1/k^{100}}) = \rho(q_{i_0},x^{1/k^{100}}).
\end{align*}
This will give us more combinatorial flexibility in controlling the exponential sums, while not losing very much in the main terms. 

We next dyadically decompose the ranges of $q_{i_0} \sim Q$ and $w = q_0 \cdots q_{i_0-1}q_{i_0+1}\cdots q_{k-1} \sim W$. Observe that $Q \ll x^{101/300}$ and $Q \gg x^{(1-o(1))/k}$. Lastly, we introduce upper-bound sieve weights to control the condition $P^-(\binom{n}{k}) \geq z$. We therefore have
\begin{align*}
\mathcal{E}_5 &\leq \sum_{i_0=0}^{k-1}\sum_{\substack{g \leq G \\ g\mid P(z) \\ P^-(g) \geq 2k}}\lambda_g^+ \sum_{v \in S(g)}\mathop{\sum\cdots\sum}_{\substack{\widetilde{f}=f_0\cdots f_{k-1} \leq x^{3/k^{99}} \\ P^+(f_i) \leq x^{1/k^{100}} \\ P^-(f_i) \geq z}} \sum_{\substack{\frac{x}{B\widetilde{f}} \ll QW \ll \frac{x}{\widetilde{f}}}} \\ 
&\mathop{\sum\cdots\sum}_{\substack{q_0,\ldots,q_{k-1} = 1 \text{ or } > x^{1/k^{100}} \\ (q_j,q_i)=1, i \neq j \\ q_j \leq Q x^{-k^{-4}}, \forall j \neq i_0 \\ q_{i_0} \sim Q, w \sim W}} \rho(q_{i_0},x^{1/k^{100}}) \sum_{\substack{n \equiv \gamma (M) \\ n \equiv j (q_{j}), \forall j \\ n \equiv a (\widetilde{f}) \\ n \equiv v(g)}} \psi_0 \left( \frac{n}{x}\right),
\end{align*}
where $G = x^{100k\epsilon_k}$. The integers $M,w,q_{i_0},\widetilde{f},g$ are pairwise coprime, so we may combine the congruence conditions on $n$ into a single congruence condition $n \equiv b \pmod{g\widetilde{f}Mq_{i_0}w}$ via the Chinese remainder theorem. We must remember the dependence of $b$ on the different variables.

We use Lemma \ref{lem:poisson summation} to transform the sum over $n$ into
\begin{align*}
\frac{x}{g\widetilde{f}Mq_{i_0}w}\widehat{\psi_0}(0) + \frac{x}{g\widetilde{f}Mq_{i_0}w}\sum_{1 \leq |h| \leq H} \widehat{\psi_0} \left(\frac{hx}{g\widetilde{f}Mq_{i_0}w} \right)e \left(\frac{bh}{g\widetilde{f}Mq_{i_0}w} \right) + O(x^{-1000}),
\end{align*}
where $H = GM (\log x)^4$. The main term arises from the term with $\widehat{\psi_0}(0)$, and the other terms contribute only to the error.

The main term is
\begin{align*}
\mathcal{E}_{5,\text{MT}} &:= \frac{x}{M}\widehat{\psi_0}(0)\sum_{i_0=0}^{k-1}\sum_{\substack{g \leq G \\ g\mid P(z) \\ P^-(g) \geq 2k}}\lambda_g^+ \frac{k^{\omega(g)}}{g} \mathop{\sum\cdots\sum}_{\substack{\widetilde{f}=f_0\cdots f_{k-1} \leq x^{3/k^{99}} \\ P^+(f_i) \leq x^{1/k^{100}} \\ P^-(f_i) \geq z}}\frac{1}{\widetilde{f}} \sum_{\substack{\frac{x}{B\widetilde{f}} \ll QW \ll \frac{x}{\widetilde{f}}}} \\
&\mathop{\sum\cdots\sum}_{\substack{q_0,\ldots,q_{k-1} = 1 \text{ or } > x^{1/k^{100}} \\ (q_j,q_i)=1, i \neq j \\ q_j \leq Q x^{-k^{-4}}, \forall j \neq i_0 \\ q_{i_0} \sim Q, w \sim W}} \frac{\rho(q_{i_0},x^{1/k^{100}})}{q_{i_0}w}.
\end{align*}
We evaluate the sum over $g$ with Lemma \ref{lem:main sieve result}. Given $\widetilde{f}$ and $W$, there are $O(\log B)$ choices for $Q$, and the sum over $q_{i_0}$ is
\begin{align*}
\sum_{\substack{q_{i_0}\sim Q \\ P^-(q_{i_0}) \geq x^{1/k^{100}}}} \frac{1}{q_{i_0}} \ll \frac{k^{100}}{\log x}.
\end{align*}
We sum over all $W$ to de-localize the size of $w = q_0 \cdots q_{i_0-1}q_{i_0+1}\cdots q_{k-1}$, and by Mertens's theorem we see that the sum over $w$ is
\begin{align*}
\leq \left(1+\sum_{x^{1/k^{100}} < q \ll x^{101/300}} \frac{1}{q} \right)^{k-1} \leq (1000 \log k)^k.
\end{align*}
The sum over $\widetilde{f} = f_0 \cdots f_{k-1}$ is
\begin{align*}
\leq \Big(1+\sum_{\substack{z \ll f \ll x^{3/k^{99}} \\ P^-(f) \geq z}} \frac{1}{f} \Big)^k \leq \epsilon_k^{-O(k)}.
\end{align*}
It follows that
\begin{align*}
\mathcal{E}_{5,\text{MT}} &\ll \mathfrak{S}_k \frac{x}{M}(\log z)^{-k} \cdot \frac{(\log B) \epsilon_k^{-O(k)}}{\log x},
\end{align*}
and this is $o(\mathcal{S})$ by Proposition \ref{prop:main term lower bound for S}.

Now we turn to the error terms. The contribution of the $O(x^{-1000})$ error is negligibly small. The main error term we must consider is
\begin{align*}
\mathcal{E}_{5,\text{error}} &:= \frac{x}{M}\sum_{i_0=0}^{k-1}\sum_{1 \leq |h| \leq H}\sum_{\substack{g \leq G \\ g\mid P(z) \\ P^-(g) \geq 2k}}\frac{\lambda_g^+}{g} \sum_{v \in S(g)}\mathop{\sum\cdots\sum}_{\substack{\widetilde{f}=f_0\cdots f_{k-1} \leq x^{3/k^{99}} \\ P^+(f_i) \leq x^{1/k^{100}} \\ P^-(f_i) \geq z}} \frac{1}{\widetilde{f}} \sum_{\substack{\frac{x}{B\widetilde{f}} \ll QW \ll \frac{x}{\widetilde{f}}}} \mathop{\sum\cdots \sum}_{\substack{w \sim W \\ q_j \leq Q x^{-k^{-4}}, \forall j \neq i_0}} \frac{1}{w} \\ 
&\times\sum_{\substack{q_{i_0} \sim Q \\ (q_{i_0},g\widetilde{f}M w)=1}} \frac{1}{q_{i_0}} \rho(q_{i_0},x^{1/k^{100}}) \widehat{\psi_0} \left(\frac{hx}{g\widetilde{f}Mq_{i_0}w} \right)e \left(\frac{bh}{g\widetilde{f}Mq_{i_0}w} \right).
\end{align*}
As is typical, using the triangle inequality to eliminate the exponential phase gives a bound that is too large by a factor of $\approx G^{O(1)}$. For some choice of $Q, W$ (depending on $\widetilde{f}$) and $i_0$, we see by the triangle inequality that
\begin{align*}
\mathcal{E}_{5,\text{error}} &\ll (\log x)^{O(1)} \sum_{1 \leq |h| \leq H}\sum_{\substack{g \leq G \\ g\mid P(z) \\ P^-(g) \geq 2k}} \mu^2(g) \sum_{v \in S(g)}\mathop{\sum\cdots\sum}_{\substack{\widetilde{f}=f_0\cdots f_{k-1} \leq x^{3/k^{99}} \\ P^+(f_i) \leq x^{1/k^{100}} \\ P^-(f_i) \geq z}} \mathop{\sum\cdots \sum}_{\substack{w \sim W \\ q_j \leq Q x^{-k^{-4}}, \forall j \neq i_0}} |\mathcal{U}|,
\end{align*}
where
\begin{align*}
\mathcal{U} := \sum_{\substack{q_{i_0} \sim Q \\ (q_{i_0},g\widetilde{f}M w)=1}} \frac{Q}{q_{i_0}} \rho(q_{i_0},x^{1/k^{100}}) \widehat{\psi_0} \left(\frac{hx}{g\widetilde{f}Mq_{i_0}w} \right)e \left(\frac{bh}{g\widetilde{f}Mq_{i_0}w} \right).
\end{align*}
Note that a trivial bound for $\mathcal{E}_{5,\text{error}}$ is too large by a factor of roughly $G^{O(1)} = x^{O(k \epsilon_k)}$.

We change variables in the Fourier transform to separate $q_{i_0}$ from the other variables. We find that for some absolute $t \asymp 1$ we have
\begin{align*}
|\mathcal{U}| &\ll \Bigg| \sum_{\substack{q_{i_0} \sim Q \\ (q_{i_0},g\widetilde{f}M w)=1}} \rho(q_{i_0},x^{1/k^{100}}) \psi_0\left(\frac{q_{i_0} t}{Q} \right) e \left(\frac{bh}{g\widetilde{f}Mq_{i_0}w} \right)\Bigg|.
\end{align*}
By two applications of Lemma \ref{lem:bezout} we find
\begin{align*}
\frac{bh}{g\widetilde{f}Mq_{i_0}w} &\equiv \frac{h(b-i_0) \overline{q_{i_0}}}{g\widetilde{f}Mw} + \frac{hi_0}{g\widetilde{f}M q_{i_0}w} \pmod{1}.
\end{align*}
The contribution from this second term is negligibly small when summed over all the variables in $\mathcal{E}_{5,\text{error}}$.

We use Lemma \ref{lem:buchstab comb decomp} with $R = x^{1/k^{100}}$ to decompose the $q_{i_0}$ variable. For the ``Type II'' sums with $n=mp_1\cdots p_j$ and a factor $\rho(m,p_j)$, we group variables, separate variables, and dyadically decompose variables as in the proof of Proposition \ref{prop:some di bigger than x^1/3}. We therefore have $|\mathcal{U}| \ll (\log x)^{O(1)} ( |\mathcal{U}_1| + |\mathcal{U}_2|)$, where
\begin{align*}
\mathcal{U}_1 &\coloneqq \sum_{\substack{r \leq R \\ r \mid P(R) \\ (r,g\widetilde{f}Mw)=1}} \mu(r) \sum_{\substack{m \sim Q/r \\ (m,g\widetilde{f}Mw)=1}} \psi_0 \left(\frac{m t}{Q/r} \right) e \left(\frac{h(b-i_0) \overline{mr}}{g\widetilde{f}Mw} \right), \\
\mathcal{U}_2 &\coloneqq \mathop{\sum\sum}_{\substack{u \sim U, m \sim V \\ (um,g\widetilde{f}Mw)=1}} \alpha_u \beta_m e \left(\frac{h(b-i_0) \overline{um}}{g\widetilde{f}Mw} \right),
\end{align*}
for some $1$-bounded sequences $\alpha_u,\beta_m$, where in $\mathcal{U}_2$ we have $R \ll U \ll R^2$, and $UV \asymp Q$. To put $\mathcal{U}_2$ in a more suitable form, we put $u$ on the inside and then apply Cauchy--Schwarz to obtain
\begin{align}\label{eq:sum over m for U2}
|\mathcal{U}_2|^2 &\leq UV^2 + V\mathop{\sum\sum}_{\substack{u_1,u_2 \sim U \\ u_1 \neq u_2 \\ (u_1u_2,g\widetilde{f}Mw)=1}} \Big| \sum_{\substack{m \sim V \\ (m,g\widetilde{f}Mw)=1}} e \left(\frac{h(b-i_0) \overline{u_1u_2}(u_2-u_1) \overline{m}}{g\widetilde{f}Mw} \right)\Big|,
\end{align}
where we have split into the diagonal $u_1 = u_2$ and the off-diagonal $u_1 \neq u_2$.

Observe that $w$ has at least two prime factors. If $w$ has only a single prime factor, then $QW \gg x/(B\widetilde{f}) \gg x^{99/100}$, but $Q,W \ll x^{101/300}$, and this is a contradiction.

The upper bound \eqref{eq:sum over m for U2} on $\mathcal{U}_2$ is already in a suitable form to apply Lemma \ref{lem:weyl differencing}, but we must work with $\mathcal{U}_1$. By Fourier inversion, we may write
\begin{align*}
\psi_0 \left(\frac{m t}{Q/r} \right) = \int_{-\infty}^\infty \widehat{\psi_0}(y) e \left(-\frac{mty}{Q/r} \right) dy.
\end{align*}
Thus, to bound $\mathcal{U}_1$, it suffices to bound
\begin{align}\label{eq:sum over m for U1}
\sum_{\substack{m \sim Q/r \\ (m,g\widetilde{f}Mw)=1}} e(\xi m) e \left(\frac{h(b-i_0) \overline{mr}}{g\widetilde{f}Mw} \right)
\end{align}
uniformly in $\xi \in \mathbb{R}$. We may then bound the sums over $m$ in \eqref{eq:sum over m for U1} and \eqref{eq:sum over m for U2} by appealing to Lemma \ref{lem:weyl differencing}. 

In each case, we take $s_0 = g\widetilde{f}M \leq x^{1/k^{50}}$, and $s_1 \cdots s_J = w = q_0 \cdots q_{i_0-1}q_{i_0+1}\cdots q_{k-1}$, so that $J \leq k-1$ (some of the $q_j$ may be equal to one; we arrange the prime $q_j$'s in increasing order and relabel them as $s_j$'s). From our bounds on $r$ and $U$, it is easy to check that $N=Q/r,Q/U$ are both $> x^{1/(2k)}$. Each $q_j$ is at most $Qx^{-k^{-4}}$, while $r,U \ll x^{O(k^{-100})}$, so each $q_j \leq N x^{-k^{-5}}$ as required by Lemma \ref{lem:weyl differencing}.

 We have $s_J = q_*$ in the application of Lemma \ref{lem:weyl differencing}, so that $\gcd(h(b-i_0)r,s_J)=1$ (for \eqref{eq:sum over m for U1}) and $\gcd(h(b-i_0)\overline{u_1u_2}(u_1-u_2),s_J)=1$ (for \eqref{eq:sum over m for U2}). This is because $q_* \gg x^{1/(2k)}$ and $U$ is on the order of $x^{O(k^{-100})}$. Moreover, $b-i_0 \equiv j_*-i_0 \pmod{q_*}$ for some $j_* \neq i_0$. We therefore save a factor of $x^{-k^{-101}2^{-k}}$ over the trivial bound. This is large compared to the factor of $G^{O(1)} = x^{O(k\epsilon_k)}$ that we needed to save, since $\epsilon_k = 3^{-k}$. After summing over all variables, we deduce that
\begin{align*}
\mathcal{E}_{5,\text{error}} \ll x^{1 - k^{-1000}2^{-k}},
\end{align*}
say, and this is sufficient to complete the proof.
\end{proof}

\section*{Acknowledgments}

The authors learned about this problem from Thomas Bloom's Erd\H{o}s problems webpage \cite{BloomErdos387}. We thank Terence Tao for helpful conversations.

The third author is partially supported by the National Science Foundation (DMS-2418328) and the Simons Foundation (MPS-TSM-00007959).

The main ideas in the proof of Theorem \ref{thm:main covering theorem} were developed in interactive sessions between the authors and ChatGPT 5.5 Pro. Some documents and code in our GitHub repository \cite{Slava2026} were generated with AI assistance. The authors also used ChatGPT for literature searches, proofreading, and checking for possible mathematical errors in previous versions of this paper.

\bibliographystyle{plain}
\bibliography{refs}

\end{document}